\documentclass[10pt,twoside, a4paper, english, reqno]{amsart}

\usepackage[margin=1in]{geometry} 
\usepackage{tikz}
\usetikzlibrary{shapes.geometric, arrows.meta, positioning}
\usepackage[dvips]{epsfig}
\usepackage{amscd,cancel}

\usepackage[margin=1in]{geometry}
\usepackage{amssymb}
\usepackage{amsthm}
\newtheorem{assumption}{Assumption}

\usepackage{amsmath}
\usepackage{algorithmicx}
\usepackage{graphicx}
\usepackage{subcaption}
\usepackage{amsmath,amsthm,amssymb,enumerate}
\usepackage{euscript,mathrsfs}
\usepackage{bbm}

\usepackage{bm}
\usepackage{latexsym}
\usepackage{tikz} 
\usetikzlibrary{shapes.geometric, arrows.meta, positioning}
\usepackage{mathtools,dsfont}

\usepackage[linesnumbered,ruled,vlined]{algorithm2e}
\usepackage{upref}
\usepackage[colorlinks,backref]{hyperref}
\usepackage{color}

\usepackage[toc,page]{appendix}
\usepackage{pdfsync}
\theoremstyle{plain}
\newtheorem{thm}{Theorem}[section]
\theoremstyle{plain}
\newtheorem{lem}[thm]{Lemma}
\newtheorem{proposition}[thm]{Proposition}

\theoremstyle{definition}
\newtheorem{definition}{Definition}[section]
\newtheorem{remark}{Remark}[section]

\newtheorem*{maintheorem*}{Main Theorem}
\newtheorem*{maincorollary*}{Main Corollary}

\usepackage{adjustbox}   
\usetikzlibrary{arrows.meta,positioning}
\allowdisplaybreaks

\numberwithin{equation}{section} \allowdisplaybreaks
\usepackage{chngcntr}
\counterwithin{algocf}{section}
\begin{document}
	\title[{\bf SLQ} problem with wave equation]
	{Convergence Analysis for an implementable scheme to solve the Linear-quadratic stochastic optimal control problem with stochastic wave equation}
	\author{Abhishek Chaudhary}
	\maketitle
	\medskip
	\centerline{Mathematisches Institut Universität Tübingen, Auf der Morgenstelle 10, 72076 Tübingen, Germany}\centerline{chaudhary@na.uni-tuebingen.de}
	
	%
	
\begin{abstract}
We study an optimal control problem for the stochastic wave equation driven by affine multiplicative noise, formulated as a stochastic linear–quadratic (SLQ) problem. By applying a stochastic Pontryagin's maximum principle, we characterize the optimal state–control pair via a coupled forward–backward SPDE system. We propose an \emph{implementable} discretization using conforming finite elements in space and an implicit midpoint rule in time. By a new technical approach we obtain strong convergence rates for the discrete state–control pair \emph{without} relying on {\em Malliavin calculus}. For the practical computation we develop a gradient-descent algorithm based on {\em artificial} iterates that employs an exact computation for the arising conditional expectations, thereby eliminating costly Monte Carlo sampling. Consequently, each iteration has a computational cost that is proportional to the number of spatial degrees of freedom, producing a scalable method that preserves the established strong convergence rates. Numerical results validate its efficiency.
\end{abstract}
	\noindent
	
	
	\noindent
	
	\section{Introduction}

	Let $D\subset\mathbb{R}^d$ $(1\le d\le 3)$ be a bounded domain with a smooth enough boundary $\Gamma$, and let $T>0$ be a fixed time. Our aim is to numerically approximate the $\mathbb{L}^2(D)$-valued, $\mathbb{F}$-adapted distributed control process $U^*\equiv\{ U^*(t); t\in[0,T]\}$ on the filtered probability space $(\Omega,\mathcal{F},{\mathbb{F}=}\{\mathcal{F}_t\}_{t\in[0,T]},\mathbb{P})$ that minimizes the cost functional $(\alpha>0, \beta\ge 0)$
	\begin{align}\label{1.1}
		J(X,U)=\frac{1}{2}\mathbb{E}\bigg[\int_0^T\big(\|X(t)-{\widetilde{X}}(t)\|_{\mathbb{L}^2(D)}^2+\alpha\|U(t)\|_{\mathbb{L}^2(D)}^2\big) \,{\rm d}t\bigg]{+\beta \mathbb{E}\big[\|X(T)-\widetilde{X}(T)\|_{\mathbb{L}^2(D)}^2\big]}
	\end{align}
	subject to the (controlled forward) stochastic wave equation driven by the affine noise
	\begin{align}
		\label{1.2}
		{\scriptstyle 
		\begin{cases} 
				\,{\rm d} X_t(t)=(\Delta X(t)+ U(t))\,{\rm d}t +(\sigma(t)+ \gamma X(t))\,\,{\rm d}W(t)& {\rm in}\,D\times (0,T], \\
				X(0) = X_{1,0} &{\rm in}\,{D},\\
				X_t(0)=X_{2,0} &{\rm in}\,{D},\\
				X(t)= X_t(0)=0 &{\rm on}\,\Gamma\times(0,T],
		\end{cases}}
	\end{align}
	where $\gamma\in\mathbb{R}^m$, $W\equiv\{W(t); t\in[0,T]\}$ is a $\mathbb{R}^m$-valued Wiener process that generates a complete filtration $\{\mathcal{F}_t\}_{t\in[0,T]}$, with initial data $X_{1,0}\in \mathbb{H}_0^1(D), X_{2,0}\in \mathbb{L}^2(D)$, the notation $X_t\equiv\partial_t X$ ({\em i.e.,} a partial derivative of $X$ {\em w.r.t.} the time variable$)$, $\widetilde{X}\in C([0,T];H_0^1(D))$ ({\em i.e.,} the given deterministic target trajectory) and {additive noise coefficient} $\sigma\in \mathbb{L}^2_{\mathbb{F}}(\Omega\times[0,T];\mathbb{L}^2(D; \mathbb{R}^m))\cap C([0,T];\mathbb{L}^2(\Omega;\mathbb{L}^2(D;\mathbb{R}^m)))$.
    
    For every $U\in L_{\mathbb{F}}^2(\Omega\times[0,T];\mathbb{L}^2(D))$, there exists a unique weak solution {\em} $X\equiv \mathcal{X}[U]\in \mathbb{L}^2_{\mathbb{F}}(\Omega\times [0,T];\mathbb{H}_0^1(D))\cap \mathbb{L}^2_{\mathbb{F}}(\Omega;H^1([0,T];\mathbb{L}^2(D)))$ to the SPDE \eqref{1.2} (see Lemma~\ref{Lemma 2.1}), and there exists also a unique minimizer $(X^*, U^*)\in \mathbb{L}^2_{\mathbb{F}}(\Omega;C([0,T];\mathbb{H}_0^1(D))\cap \mathbb{L}^2_{\mathbb{F}}(\Omega;H^1([0,T];\mathbb{L}^2(D)))\times \mathbb{L}^2_{\mathbb{F}}(\Omega\times [0,T];\mathbb{L}^2(D))$ of the stochastic optimal control problem (see Proposition~\ref{existence of a unique optimal pair}): \lq{minimize \eqref{1.1} subject to \eqref{1.2}}\rq, which we later refer to as the {\bf SLQ} problem.

\noindent	
Let $X_1=X$ and $X_2=X_t$ then we rewrite the {\bf SLQ} problem~\eqref{1.1}-\eqref{1.2} as follows: find the unique optimal tuple $(X_1^*, X_2^*, U^*)\in \mathbb{L}^2_{\mathbb{F}}(\Omega;C([0,T];\mathbb{H}_0^1(D))\cap \mathbb{L}^2_{\mathbb{F}}(\Omega;H^1([0,T];\mathbb{L}^2(D)))\times \mathbb{L}^2_{\mathbb{F}}(\Omega\times [0,T];\mathbb{L}^2(D)) $ that minimizes the following cost functional
	\begin{align}\label{1.3}
	J(X_1,U)=\frac{1}{2}\mathbb{E}\bigg[\int_0^T\big(\|X_1(t)-\widetilde{X}(t)\|_{\mathbb{L}^2(D)}^2+\alpha\|U(t)\|_{\mathbb{L}^2(D)}^2\big) \,{\rm d}t\bigg]+ \beta\mathbb{E}\big[\|X_1(T)-\widetilde{X}(T)\|_{\mathbb{L}^2(D)}^2\big]
\end{align}
subject to the (controlled forward) stochastic system driven by the affine noise
\begin{align}
	\label{1.4}
	{\scriptstyle 	\begin{cases} 
			\,{\rm d}X_1(t)= X_2(t)\,{\rm d}t& {\rm in }\,D\times (0,T], \\
			\,{\rm d} X_2(t)=(\Delta X_1(t)+ U(t))\,{\rm d}t +(\sigma(t)+\gamma X_1(t))\,\,{\rm d}W(t)& {\rm in}\,D\times (0,T],  \\
			X_1(0) = X_{1,0} &{\rm in}\,{D},\\
			X_2(0)=X_{2,0} &{\rm in}\,{D},\\
			X_1(t)=X_2(t)= 0 &{\rm on}\,\Gamma\times(0,T].
	\end{cases}}
\end{align}	
Clearly, the {\bf SLQ} problem \eqref{1.1}–\eqref{1.2} is equivalent to the {\bf SLQ} problem \eqref{1.3}–\eqref{1.4}. To given $X_{1}^*$ and $\widetilde{X}$, the following system of {\bf BSPDE} 
	\begin{align}
		\label{1.5}
		\begin{cases} 
			\,{\rm d}Y_1(t)=-[\Delta Y_2(t)+\gamma\cdot Z_2(t)+ X_1^*(t)-\widetilde{X}(t)]\,{\rm d}t + Z_1(t)\,{\rm d}W(t) & {\rm in}\,D\times [0,T),  \\
			\,{\rm d}Y_2(t)=-Y_1(t)\,{\rm d}t +  Z_2(t)\,{\rm d}W(t) & {\rm in}\,D\times [0,T), \\
			Y_1(T) = \beta \big(X^*_1(T)-\widetilde{X}(T)\big) &{\rm in}\,{D},\\
			Y_2(T) = 0 &{\rm in}\,{D},\\
			Y_1(t)=Y_2(t)=0 &{\rm on}\,\Gamma\times [0,T),
		\end{cases}
	\end{align}
	has a unique strong solution quadruple $(Y_1, Y_2, Z_1, Z_2)\in L_{\mathbb{F}}^2(\Omega; C([0,T];\mathbb{L}^2(D))\times L_{\mathbb{F}}^2(\Omega; C([0,T];\mathbb{H}_0^1(D))\times \mathbb{L}^2_{\mathbb{F}}(\Omega\times[0,T];\mathbb{L}^2(D;\mathbb{R}^m))\times \mathbb{L}^2_{\mathbb{F}}(\Omega\times[0,T]; {\mathbb{H}_0^1(D;\mathbb{R}^m)})$; see Lemma~\ref{Lemma 2.3}. The adjoint variable $Y_2$ is then related to the optimal control by Pontryagin's maximum principle (see Theorem~\ref{Pontryagin's Maximum Principle}), which in the case of problem {\bf SLQ} \eqref{1.3}-\eqref{1.4} is 
	\begin{align}\label{1.6}
		\alpha U^*=-{Y}_2\qquad{\rm in}\,\mathbb{L}^2_{\mathbb{F}}(\Omega\times[0,T];\mathbb{H}^1_0(D)).
	\end{align}
Stochastic wave equations driven by additive or multiplicative noise arise naturally in many applications, such as structural vibration control under random excitations \cite{KloedenPlaten2001}, acoustic wave propagation in uncertain media, and energy harvesting from random ocean-wave fields \cite{XiaEtAl2019}. These systems are modeled by a second-order hyperbolic SPDE with Gaussian forcing \cite{CohenLarssonSigg2017}. In this context, one can formulate optimal control problems in a stochastic linear–quadratic ({\bf SLQ}) framework, aiming to minimize a quadratic cost functional subject to stochastic wave dynamics; see \cite[Example 7.1]{Lu2014}. The present work numerically addresses this class of {\bf SLQ} problems by using an open-loop approach via the stochastic maximum principle for wave equations with additive-multiplicative noise.

\subsection{Previous works}

For the {\em deterministic} linear–quadratic control of the wave equation, the foundational existence and uniqueness theory was laid out by Lions \cite{Lions1971}, and further detailed by Tr\"oltzsch \cite{Troeltzsch2010}, where the coupled state–adjoint system is shown to be well-posed in the natural energy spaces. The well-posedness of analytic solutions is established via abstract weak compact embeddings, which are not suitable for numerical computation. Zuazua \cite{Zuazua2005} analyzed finite-difference discretizations of the \emph{deterministic} wave equation and showed that, unlike exact controllability, the discrete \textbf{LQ} controls converge despite of spurious high-frequency numerical artifacts.

Löscher and Steinbach \cite{LoscherSteinbach2022} introduced a space–time finite‐element discretization for the distributed \textbf{LQ} control of the wave equation and established convergence of the fully discrete scheme without any CFL‐type restriction. Building on this, Langer et al.\ \cite{LangerEtAl2024} developed block‐preconditioned iterative solvers for the resulting global systems, demonstrating mesh‐independent convergence and parallel scalability in the Tikhonov‐regularized hyperbolic setting.

Engel et al. \cite{EngelTrautmannVexler2019} derived optimal finite‐element error estimates for wave‐equation control with bounded‐variation controls. In the one‐dimensional, measure‐valued setting, Trautmann et al. \cite{TrautmannVexlerZlotnik2018a,TrautmannVexlerZlotnik2018b} proved convergence rates via three‐level time‐stepping and conforming finite elements.

On the algorithmic front, Kröner et al. \cite{KronerKunischVexler2009} proved local superlinear convergence of semismooth Newton and primal–dual active‐set methods for both distributed and boundary control problems, and Steinbach and Zank~\cite{SteinbachZank2022} developed an inf-sup stable variational formulation for linear-quadratic optimal control problems that facilitates the deign of robust and scalable space--time solvers, including parallel implementations.

In contrast, numerical investigations of \emph{stochastic} control problems remain relatively scarce. For systems governed by finite‐dimensional SDE, see \cite{Archibald_Bao_Yong_Zhou, Archibald_Baoand_Young;, MenaDammStillfjord, Yanqing 2021, Yanqing 2023}. In the context of SPDE‐constrained distributed control, key references include \cite{Duns&Prohl, Li Zhou, Prohl&Wang1, ProhlWang2021, ProhlWangbook, B.Li,  wang2020}. Notably, \cite{Duns&Prohl} employs a data‐driven partitioning regression estimator to approximate the control and state, derives convergence rates for a conforming finite‐element semi‐discretization, and discusses practical implementation; the interaction between spatial and temporal discretization errors is further analyzed in \cite{Prohl&Wang1, ProhlWang2021}.

Our analysis is based on the FBSPDE system \eqref{1.4}-\eqref{1.5} with the optimality condition \eqref {1.6} and its fully discrete version. The extensive literature on numerical schemes for BSDEs includes, among others, \cite{BenderDenk2007, BenderSteiner, gobet2005, Lemor2006, Longstaff2001, Milstein2006}, which provide various approaches and theoretical insights into their discretization and practical implementation.
Notably, Chaudhary et al. \cite{ChaudharyProhl} provide an approach based on recursive formula to {\em avoid} the statistical approximation of arising conditional expectations for the simulation in the case of a different {\bf SLQ} problem, which would otherwise limit the space-time resolution of the FBSPDE system.
\subsection{Our contributions in this paper}
\begin{figure}[ht!]
  \centering
  \adjustbox{max width=\textwidth,max totalheight=0.60\textheight,keepaspectratio}{
    \begin{tikzpicture}[
        box/.style={
          rectangle, draw, fill=blue!5,
          text width=1.7cm,     
          minimum height=6mm,   
          text centered, rounded corners,
          font=\sffamily\scriptsize,
          inner sep=0.6mm
        },
        arrow/.style={thick, -{Stealth[length=2.0mm, width=1.1mm]}},
        row sep=1.5mm, column sep=5mm
      ]
      \node (slq) [box, text width=2.4cm] {{\bf SLQ} problem \eqref{1.1}-\eqref{1.2}};

      \node (error) [box, below left=1.5mm and 6mm of slq, text width=2.0cm]
        {{\bf Pontryagin's maximum principle} (Theorem~\ref{Pontryagin's Maximum Principle})};
      \node (algo)  [box, below right=1.5mm and 6mm of slq, text width=2.0cm]
        {Full discretization of {\bf SLQ} problem~\eqref{fully discrete cost functional}-\eqref{discrete state equation} ({\em ${\bf SLQ}_{h\tau}$ problem})};

      \node (semi)   [box, below=2.5mm of error, xshift=-15mm]
        {Space-discretization of {\bf SLQ} problem ({\em {\bf SLQ}$_h$ problem \eqref{3.1}-\eqref{3.2}})};
      \node (fully)  [box, below=2.5mm of error, xshift=15mm]
        {Full-discretization of {\bf SLQ} problem ({\em ${\bf SLQ}_{h\tau}$ problem~\eqref{fully discrete cost functional}-\eqref{discrete state equation}})};
      \node (discrete)[box, below=2.5mm of algo, xshift=-10mm]
        {Fully discrete {\bf PMP}\\(Theorem~\ref{fully discrete Pontryagin's maximum principle})};
      \node (gradient)[box, below=2.5mm of algo, xshift=15mm]
        {Gradient descent method ({\em i.e,} Algorithm~\ref{tt1})};

      \node (fem)     [box, below=2.5mm of semi, xshift=-15mm]
        {Error estimate for {\bf SLQ}$_h$ problem (Theorems ~\ref{Thm 3.8} \&~\ref{thm3.5})};
      \node (fulldisc) [box, below=2.5mm of fully, xshift=15mm]
        {Error estimate for ${\bf SLQ}_{h\tau}$ problem (Theorems~\ref{strong rate of convergence for time discretization} and \ref{strong rate of convergence for time discretization1})};
      \node (recursive) [box, below=2.5mm of gradient, xshift=-6mm]
        {Artificial iterates for the conditional expectation $\mathbb{E}[\cdot|\mathcal{F}_{t_n}]$ (Section~\ref{subsection 5.4})};
      \node (implement) [box, below=2.5mm of gradient, xshift=30mm]
        {{\em Implementable } algorithm for simulation (Algorithm~\ref{tt3})};

      \node (rateconv) [box, below=2.5mm of fulldisc, xshift=-15mm]
        {Convergence rate for ${\bf SLQ}_{h\tau}$ problem towards {\bf SLQ} problem (Theorem~\ref{final result for full discrete scheme})};
      \node (rateimpl) [box, below=3.5mm of implement, xshift=-10mm]
        {Convergence rate for {\em Implementable } Algorithm (Theorem~\ref{convergence of gradient descent method})};
      \node (example)  [box, below=2.5mm of implement, xshift=16mm]
        {Numerical simulation (Section~\ref{example 1})};

      \draw[arrow] (slq) -- (error);
      \draw[arrow] (slq) -- (algo);

      \draw[arrow] (error) -- (semi);
      \draw[arrow] (error.east) .. controls +(8mm,-6mm) and +(-8mm,6mm) .. (fully.north);
      \draw[arrow] (semi) -- (fem);
      \draw[arrow] (fully) -- (fulldisc);
      \draw[arrow] (fem.east) .. controls +(10mm,-4mm) and +(-10mm,4mm) .. (rateconv.north);
      \draw[arrow] (fulldisc) -- (rateconv);

      \draw[arrow] (algo) -- (discrete);
      \draw[arrow] (discrete) -- (gradient);
      \draw[arrow] (gradient) -- (recursive);
      \draw[arrow] (recursive) -- (implement);
      \draw[arrow] (implement) -- (example);

      \draw[arrow] (gradient.south) -- (rateimpl.north);
      \draw[arrow] (rateconv) -- (rateimpl);
      \draw[arrow] (implement) -- (rateimpl);
    \end{tikzpicture}
  }
  \caption{A flowchart outlining the error analysis and algorithmic approach for the {\bf SLQ} problem. Here, \textbf{PMP} denotes Pontryagin's maximum principle.}
  \label{fig:my_flowchart}
\end{figure}
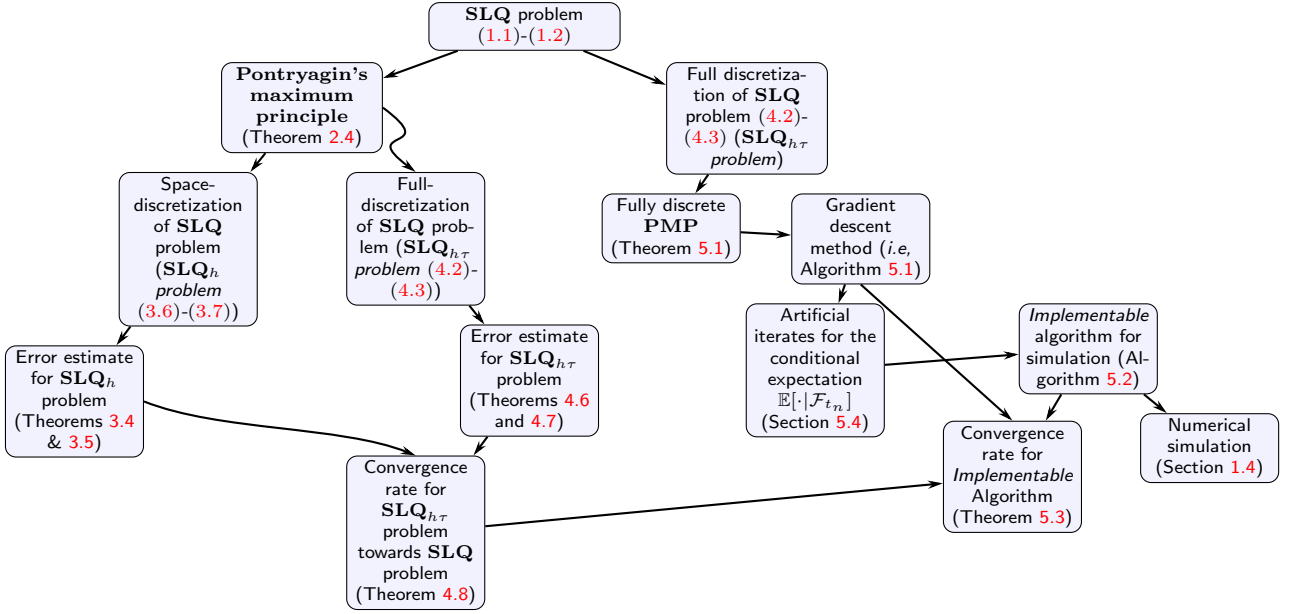
The main objective of this paper is to propose an \emph{efficient} and \emph{ implementable} numerical scheme—referred to as Algorithm~\ref{tt3}—for solving the {\bf SLQ} problem governed by a stochastic wave equation \eqref{1.2}. This algorithm is constructed to approximate the unique optimal control $U^*$ and the associated state $X^*$ for problem {\bf SLQ} \eqref{1.1}-\eqref{1.2}. Below, we detail the contributions of this work:

\begin{itemize}
	\item[(1)]\textbf{A coupled FBSPDE as optimality system:} We begin by establishing existence and uniqueness of the optimal tuple \((X^*_1, X_2^*, U^*)\) for the {\bf SLQ} problem~\eqref{1.3}-\eqref{1.4}.  Here, the state equation is posed in its standard variational (weak) form (see Definition~\ref{definition1}).  Applying a stochastic version of Pontryagin’s maximum principle yields a {\em coupled} forward–backward SPDE system~\eqref{1.4}-\eqref{1.5} that characterizes a unique optimal tuple \((X^*_1, X_2^*, U^*)\) via an optimality condition~\eqref{1.6}; see Theorem \ref{Pontryagin's Maximum Principle}.
	\item[(2)]\textbf{First discretize then optimize:} For the practical implementation, we propose a fully discrete approximation of the {\bf SLQ} problem, denoted by \({\bf SLQ}_{h\tau}~\eqref{fully discrete cost functional}-\eqref{discrete state equation}\), which combines a conforming finite element method in space with an implicit midpoint scheme in time. The implicit midpoint rule is selected for its time-reversibility, unconditional stability, and conserved energy-behavior in the deterministic wave setting. This space–time discretization yields a coupled discrete optimality system (see Propositions~\ref{Semi-discrete Pontryagin's maximum principle} and~\ref{fully discrete Pontryagin's maximum principle}).  
	
	\item[(3)]\textbf{Avoidance of {\em Malliavin calculus}:} A common approach for deriving error estimates in stochastic control problems, particularly those involving parabolic equations \cite{Prohl&Wang1,ProhlWang2021}, is to rely on {\em Malliavin calculus} to handle the involved BSPDE \eqref{1.5}; see Remark~\ref{Remark 4.3}. However, in our setting—due to the distinct structure of the BSPDE \eqref{1.5} arising from \textbf{SLQ} problem \eqref{1.3}-\eqref{1.4}—it may become difficult to apply {\em Malliavin calculus}, especially for estimating error terms associated with the diffusion component \(Z=(Z_1, Z_2)\) in the analysis of the time discretization; see Remark~\ref{Remark 4.3}. To overcome this difficulty, we develop a key proposition (see Proposition~\ref{Proposition02}) that {\em avoids} the use of {\em Malliavin calculus}; see Remarks~\ref{Remark 4.3} and ~\ref{Remark 4.5}. These results allow us to prove strong convergence of the fully discrete optimal tuple \((X^*_{1, h\tau}, X_{2,h\tau}^*, U^*_{h\tau})\) towards the continuous solution tuple \
    \((X_1^*, X_2^*, U^*)\) without invoking Malliavin derivatives (see Theorem~\ref{strong rate of convergence for time discretization}). This approach forms one of the central novelties of our work. 
	
	\item[(4)] \textbf{Artificial gradient iterates:} A subsequent step then is to decouple the discrete optimality system; see Proposition~\ref{fully discrete Pontryagin's maximum principle}. To compute the discrete optimal control in practice, we employ a gradient‐descent method, \({\bf SLQ}_{h\tau}^{\mathrm{grad}}\) (see Algorithm \ref{tt1}), that alternates updates of the state and control iterates $(X_{h\tau}^{(\ell)}=(X_{1,h\tau}^{(\ell)}, X_{2,h\tau}^{(\ell)}), U_{h\tau}^{(\ell)})$. A major computational bottleneck is the need to evaluate conditional expectations $\mathbb{E}[\cdot|\mathcal{F}_{t_n}]$ at each time step in the computation of adjoint iterate $Y_{h\tau}^{(\ell)}$, which usually is approximated by Monte Carlo least squares-regression methods; see Section~\ref{subsection 1.3} and Remark~\ref{5.1}. In the additive‐noise setting (i.e., $\gamma=0$), we {\em avoid} these Monte Carlo methods by introducing  a concept of {\em artificial} gradient iterates; see Section~\ref{5.4.1}. Consequently, each iteration has a computational cost proportional to the number of spatial degrees of freedom, making Algorithm~\ref{tt3} both efficient and scalable in high‐dimensional discretizations (see Section~\ref{example 1}), while preserving the strong convergence rate of the underlying scheme (see Theorem~\ref{convergence of gradient descent method}). This concept of {\em artificial} gradient iterate forms another novelty of our work.
\end{itemize}

\subsection{High complexity problem to approximate conditional expectations}\label{subsection 1.3}

Approximating the conditional expectation \(\mathbb{E}[\cdot \mid \mathcal{F}_{t_n}]\approx\mathbb{E}[\cdot \mid X_{1,h\tau}^{(\ell)}(t_n)]  \)---which occurs in the equation~\eqref{backward spde} for the adjoint iterates---becomes notoriously difficult in situations when the path of state $X_{1,h\tau}^{(\ell)}(t_n, \omega) \in \mathbb{V}_h\cong \mathbb{R}^{d_h}$ realizes a high dimension\footnote{In the setting of {\bf SLQ} problem, the state space $\mathbb{V}_h$ is a high-dimensional subspace of the infinite dimensional space $\mathbb{H}_0^1(D)$-- whose dimension depends on the mesh size $h>0$; see Section~\ref{Section 3} for its definition.}. Classical statistical techniques, which rely on probabilistic Monte Carlo regression, encounter the \textbf{curse of dimensionality} \cite{bellman1957dynamic, ChessariEtAl2023}. As the dimension \(d_h\) increases, the state-space volume grows rapidly, causing data sparsity and slowing statistical convergence—approximately at a rate of \(\mathrm{M}^{-2/(d_h+2)}\) for least-squares regression methods, where \(\mathrm{M}\) is the number of Monte Carlo samples \cite[Theorem 4.2]{Gyorfi2002}. This makes such methods highly non-efficient or even impractical for a higher dimension $d_h$.

Specific approaches, like the least-squares Monte Carlo (LSMC) method \cite{Longstaff2001}, originally designed for option pricing, have been adapted for \textbf{BSDE} \cite{gobet2005, Lemor2006, ChessariEtAl2023}. Refinements, such as those by Bender \& Steiner \cite{BenderSteiner}, replace generic regression bases with martingale systems tailored to the Markovian structure, simplifying projections and improving stability. However, these methods still demand a combinatorial number of basis functions and samples \cite[Section 4]{ChessariEtAl2023}. Alternative techniques—including {\em Malliavin calculus}, quantization, tree-based methods, cubature, and forward numerical methods—perform well in few dimensions but again falter in high-dimensional state spaces due to the same curse of dimensionality \cite[Table 1, Section 7]{ChessariEtAl2023}. In \cite{Duns&Prohl}, Dunst and Prohl used a random partitioning estimator-based strategy to approximate arising conditional expectations in the approximation of high-dimensional \textbf{BSDE}, but again this approach becomes increasingly costly when numerical parameters $h, \tau$ tend to zero.

Our algorithm ({\em i.e.,} Algorithm \ref{tt3}) overcomes these challenges with the help of {\em artificial} gradient iterate in Section~\ref{5.4.1} for the {\em exact} computation of conditional expectations on the high-dimensional space \(\mathbb{V}_h\), eliminating the need for Monte Carlo sampling. As a result, its runtime scales \emph{proportionally} with the problem size, rather than exponentially in \(d_h\) (see Remark \ref{Remark 1.1}), and this removes the curse of dimensionality to simulate appearing conditional expectations. Additionally, our {\em implementable } algorithm maintains an explicit convergence rate tied to the numerical parameters \(h\), \(\tau\), and \(\ell\); see Theorem \ref{convergence of gradient descent method}.

\subsection{Numerical simulation}\label{example 1}
We motivate the capabilities of Algorithm~\ref{tt3} by a numerical simulation. For this purpose, we consider the spatial domain \(D=(0,1)\) and final time \(T=1\).  The initial data are chosen as
\[
X_{1,0}(x) = x^2 (1-x),\quad\text{and}\quad
X_{2,0}(x) = 0\quad\forall\,x\in [0,1],
\]
and the noise coefficients are given, for \(1\le i\le m=10\), by
\[
\sigma_i(t, x)
= 2\sin((i+1)\pi x) \cos(0.5(i+1)\pi t)(1+x)\quad\forall\,(t, x)\in [0,1]\times[0,1],
\]
with $\mathbb{R}^m$-valued Wiener process $W$. For the quadratic cost functional we take \(\beta=9\), \(\alpha=0.01\), and set the target profile
\[
\widetilde X(t, x)
=\sin(3\pi x)(0.5+\cos(2\pi t))\qquad\forall\,(t,x)\in [0,1]\times [0,1].
\]
The space–time discretization parameters are
$\tau = \frac1{60},\,\text{and}\,h = \frac1{100} $ (so $d_h=99$),
while the gradient‐descent iteration in Algorithm~\ref{tt3} uses
$
\ell = 10,\,\text{and}\, \kappa = 2.8.
$
Moreover, for the decay of the cost functional, we define the approximated cost functional
\begin{align*}
    J_{h\tau}(X_{h\tau}^{(\ell)},U_{h\tau}^{(\ell)})&\approx J_{h\tau}^{\rm M}(X_{h\tau}^{(\ell)},U_{h\tau}^{(\ell)})\\&=\frac{1}{\rm 2M}\sum_{{\rm m}=1}^{\rm M} \bigg[\int_0^T\big(\|X_{1,h\tau}^{(\ell,\mathrm{m})}(t)-\widetilde{X}(t)\|_{\mathbb{L}^2(D)}^2+\alpha\|U_{h\tau}^{(\ell, \mathrm{m})}(t)\|_{\mathbb{L}^2(D)}^2\big) \,{\rm d}t+\beta\|X_{1,h\tau}^{(\ell,\mathrm{m})}(T)-\widetilde{X}(T)\|_{\mathbb{L}^2(D)}^2\bigg],
\end{align*}
where $\{(X_{1, h\tau}^{(\ell,,{\rm m})},U_{h\tau}^{(\ell,{\rm m})})\}_{{\rm m}=1}^{\rm M}$ is the collection of ${\rm M}$ - Monte Carlo copies of $(X_{h\tau}^{(\ell)},U_{h\tau}^{(\ell)}).$ Note that upon convergence of Algorithm~\ref{tt3}, the discrete approximations satisfy, for all \((t, x)\in[0,T]\times D\),
\[
X^*(t, x)\approx X_{1,h\tau}^{(\ell)}(t, x), 
\quad
\partial_t X^*(t, x)\approx X_{2,h\tau}^{(\ell)}(t, x),
\quad
U^*(t, x)\approx U_{h\tau}^{(\ell)}(t, x).
\]
\subsubsection*{\textbf{Example 1}}
In this example, first we simulate a single path of control iterate $U_{h\tau}^{(\ell)}$ and state iterate $X_{1,h\tau}^{(\ell)}$ computed by Algorithm~\ref{tt3}; see Figure~\ref{fig:wave_control_state}. Secondly, we plot the discrete cost functional~\eqref{fully discrete cost functional} and the marginal histogram plot for the control iterate $U_{h\tau}^{(\ell)}$ in Figure~\ref{jjj1}.
\bigskip

\begin{figure}[htbp]
	\centering
	\begin{subfigure}[b]{0.49\textwidth}
		\centering
		\includegraphics[width=\textwidth]{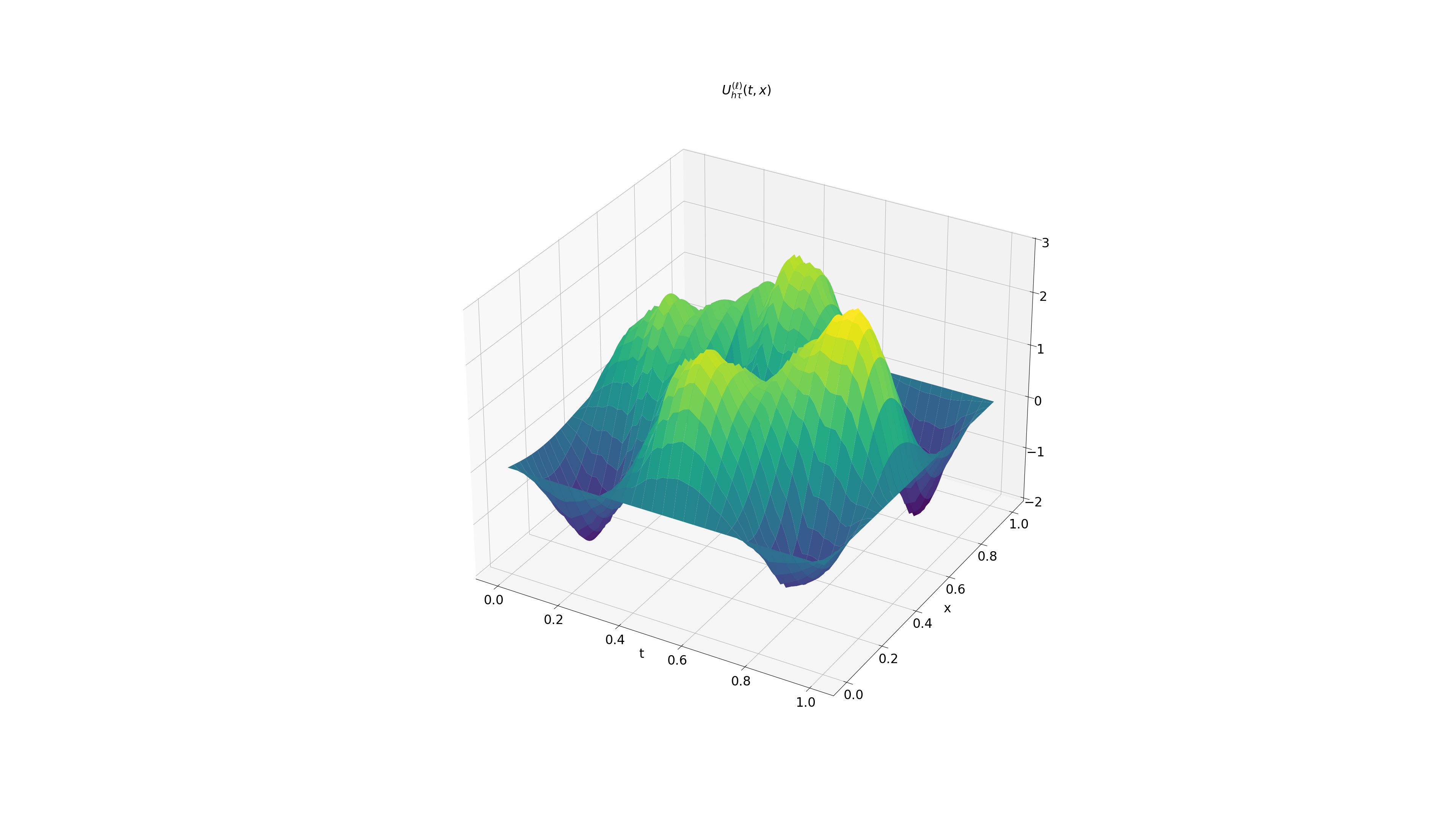}
		\caption{A path of control iterate $U_{h\tau}^{(\ell)}$}
		\label{fig:control_wave}
	\end{subfigure}
	\hfill
	\begin{subfigure}[b]{0.49\textwidth}
		\centering
		\includegraphics[width=\textwidth]{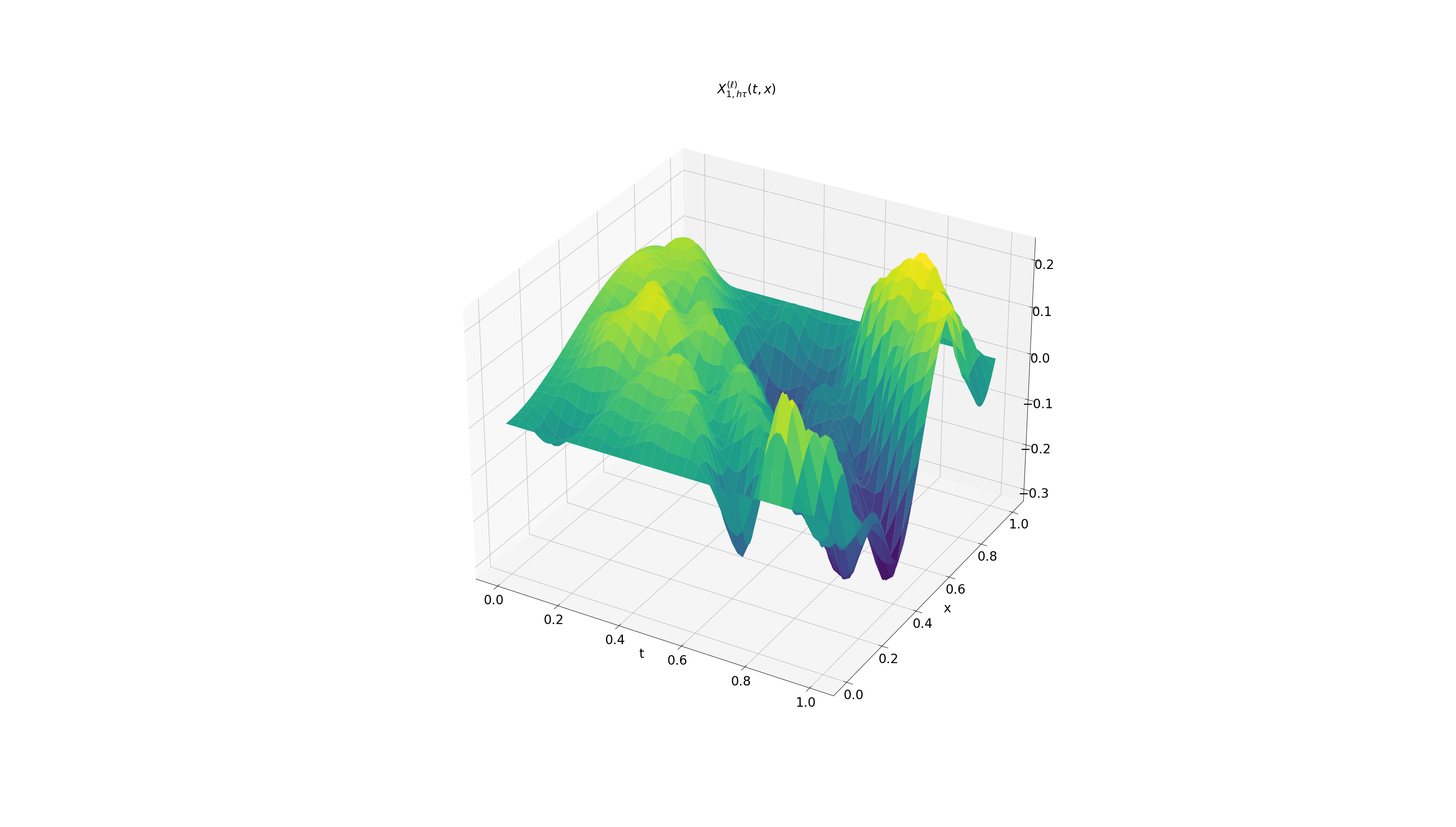}
		\caption{A path of displacement state iterate $X_{1,h\tau}^{(\ell)}$}
		\label{fig:state_wave}
	\end{subfigure}
	\caption{Surface plots for a path of the $\ell$‑th iterate over the space–time domain: (A) control iterate $(t, x)\mapsto U_{h\tau}^{(\ell)}(\omega, t, x)$; (B) displacement state iterate $(t, x)\to X_{1,h\tau}^{(\ell)}(\omega,t, x)$.}
	\label{fig:wave_control_state}
\end{figure}

\begin{figure}[htbp]
	\centering
	\begin{subfigure}[b]{0.49\textwidth}
		\centering
		\includegraphics[width=\textwidth]{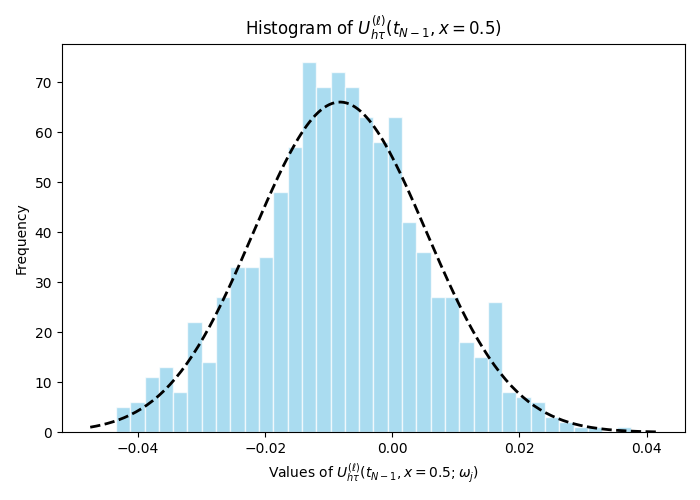}
		\caption{Histogram of control iterate $U_{h\tau}^{(\ell)}$ at $(t_{N-1}, 0.5)$ }
		\label{fig:xvstime1}
	\end{subfigure}
\hfill
	\begin{subfigure}[b]{0.47\textwidth}
		\centering
		\includegraphics[width=\textwidth]{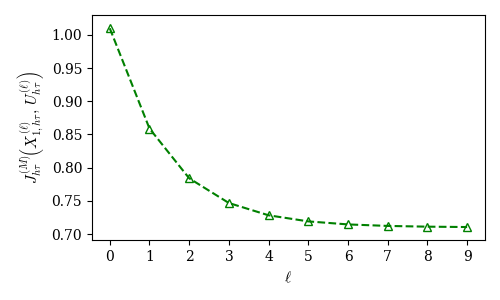}
		\caption{Decay of the (approximated) cost functional $\ell\mapsto J_{h\tau}^{\rm M}(X_{1,h\tau}^{(\ell)},U_{h\tau}^{(\ell)})$ with $\beta=0$.}
		\label{fig:uvstime1}
	\end{subfigure}
 	\caption{(A) Histogram ({empirical density}) of $\big\{U_{h \tau}^{(\ell)}(t_{N-1}, 0.5;\omega_{i})\big\}_{i=1}^{\rm M}$, and (B) decay of the (approximated) cost functional $\ell\mapsto J_{h\tau}^{\rm M}(X_{1,h\tau}^{(\ell)},U_{h\tau}^{(\ell)})$ for ${\rm M}=1000$.}
	\label{jjj1}
\end{figure} 
\begin{remark}[Computational time]\label{Remark 1.1}
	In our case, simulating one path of the optimal state iterate \(X_{1,h\tau}^{(\ell)}\) and the optimal control iterate \(U_{h\tau}^{(\ell)}\) via Algorithm~\ref{tt3} required less than \(10\)\,seconds. For comparison, we mention the work \cite{ChaudharyProhl}, where a convergent discretization for a Dirichlet-boundary {\bf SLQ} control problem was constructed. That work employed a technique based on a recursive formula for the adjoint iterate, and compared CPU times for the computation of a single sample path of the approximated control in their way \emph{vs.} a regression-based estimator method. It was found there that the regression-based estimator method was more than \(500\) times slower;  see \cite[Remark 1.1]{ChaudharyProhl}. We expect a corresponding improved performance in CPU time for the present {\bf SLQ} problem \eqref{1.3}--\eqref{1.4} as well.
    \end{remark}
The next example is intended to highlight the difference between optimal control tuples--which are computed by our algorithm in the deterministic case ({\em i.e.,} $\sigma \equiv 0$ in \eqref{1.2}{\em }) and in the stochastic case ({\em i.e.,} $\sigma \neq 0$ in \eqref{1.2}{\em }).

\subsubsection*{\textbf{Example 2}} In this example, we study the results of our algorithm ({\em i.e.,} Algorithm~\ref{tt3}) for the wave‐equation system~\eqref{1.4} under three noise regimes: zero, small, and large.  Let the noise coefficients satisfy
\[
\sigma'_i =
\begin{cases}
	0, & \text{(zero noise)},\\
	0.1\,\sigma_i, & \text{(small noise)},\\
	\sigma_i, & \text{(large noise)},
\end{cases}
\quad i=1,\dots,m=10,
\]
where \(\sigma_i\) denotes the noise coefficients.  The evolution of the displacement, and velocity iterates under these settings is displayed in Figures~\ref{j22} and \ref{j33}.

\medskip

Figures~\ref{j22} and \ref{j33} show how the solution profiles change as \(\sigma'\) increases.  Under \textbf{zero noise} (\(\sigma' = 0\)), both iterates follow their deterministic, periodic pattern for some fixed times $t$ as expected due to our target profile $\widetilde{X}$.  When \textbf{small noise} (\(\sigma' = 0.1\sigma\)) is introduced:
\begin{itemize}
	\item The \emph{displacement} The iterate $X_{1, h\tau}^{(\ell)}$ deviates only slightly from their noise‐free trajectories; see {\em columns}~ 1 \& 2 in Figures~\ref{j22} and \ref{j33}.
	\item The \emph{velocity} iterate $X_{2,h\tau}^{(\ell)}$ already exhibits more noticeable fluctuations; see in particular Figures~\ref{j22}(E) and \ref{j33}(E), since the stochastic perturbation enters directly into the velocity component $X_{t}$ of the wave equation~\eqref{1.2}.
\end{itemize}
As we move to \textbf{large noise} (\(\sigma' =\sigma\)):
\begin{itemize}
	\item The displacement iterate $X_{1, h\tau}^{(\ell)}$, and velocity iterate $X_{2,h\tau}^{(\ell)}$—display significant, rapid variations; see {\em column} 3 in Figures~\ref{j22} and \ref{j33}.
	\item The clear periodicity seen at lower noise levels is effectively lost, overwhelmed by the stronger stochastic disturbances.
\end{itemize}
Overall, these plots suggest that the velocity component is most sensitive to noise, and that sufficiently large noise levels can completely disrupt the system’s regular oscillatory behavior.

\begin{figure}[htbp]
	\centering
	\setlength{\tabcolsep}{4pt}
	\renewcommand{\arraystretch}{1.2} 
	
	\begin{tabular}{ccc}
		\textbf{Zero noise ($\sigma'=0$)} &
		\textbf{Small noise ($\sigma' = 0.1\sigma$)} &
		\textbf{Large noise ($\sigma' = \sigma$)} \\
		\\[-1ex] 

		\begin{subfigure}[b]{0.32\textwidth}
			\includegraphics[width=\linewidth]{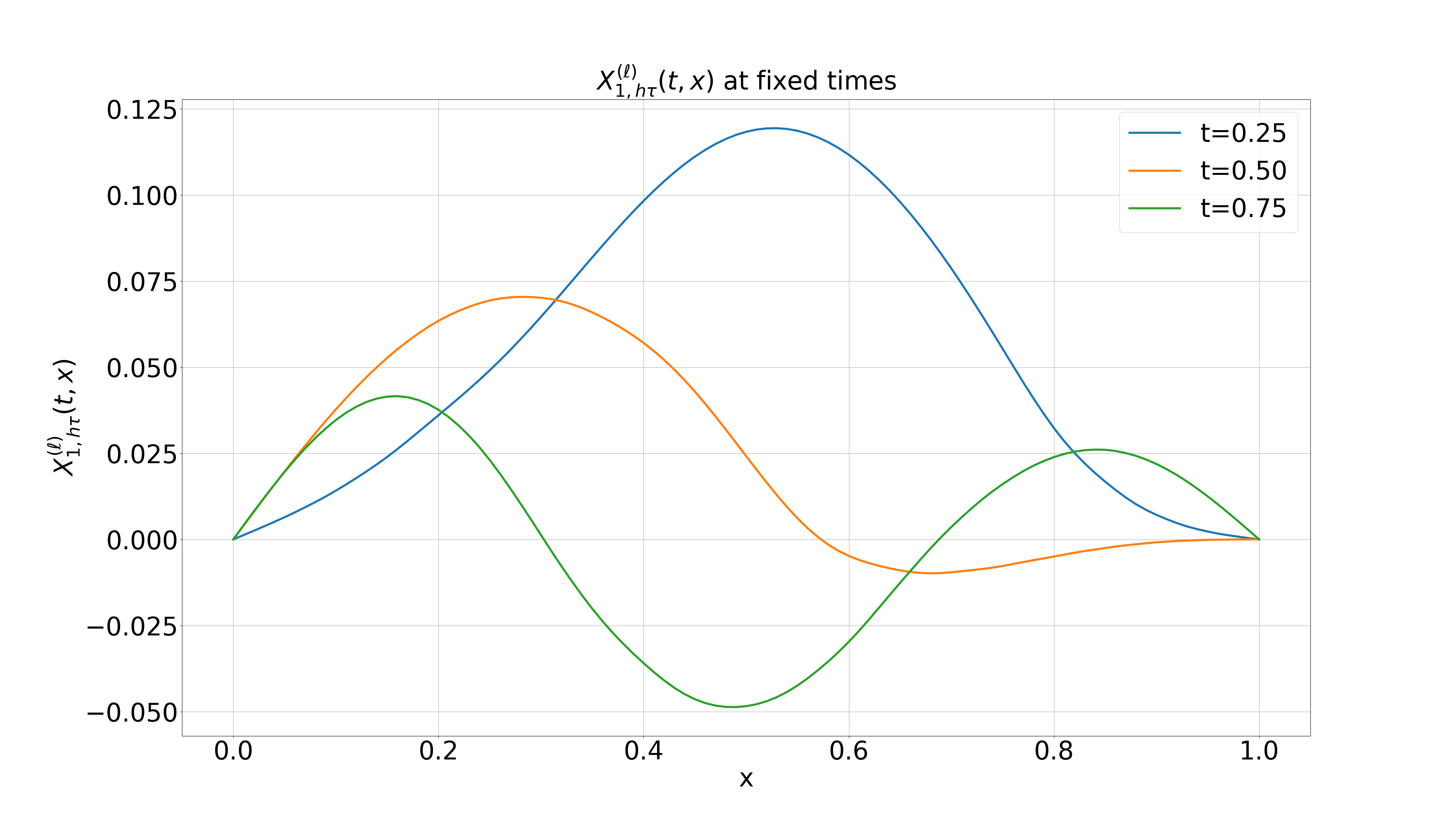}
			\caption{$x\mapsto X_{1,h\tau}^{(\ell)}(t,x)$}
		\end{subfigure}
		&
		\begin{subfigure}[b]{0.32\textwidth}
			\includegraphics[width=\linewidth]{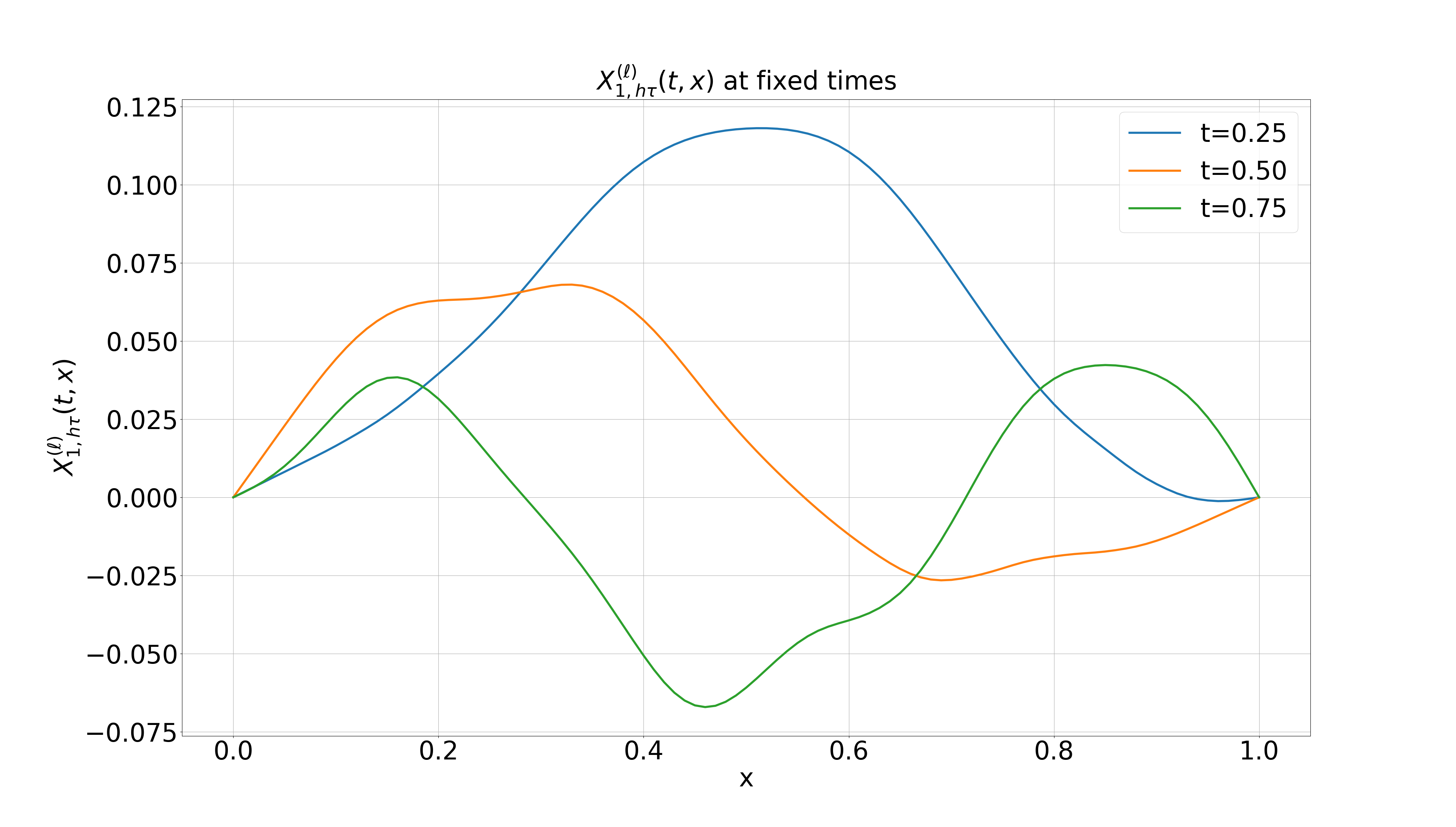}
			\caption{$x\mapsto X_{1,h\tau}^{(\ell)}(t,x)$}
		\end{subfigure}
		&
		\begin{subfigure}[b]{0.32\textwidth}
			\includegraphics[width=\linewidth]{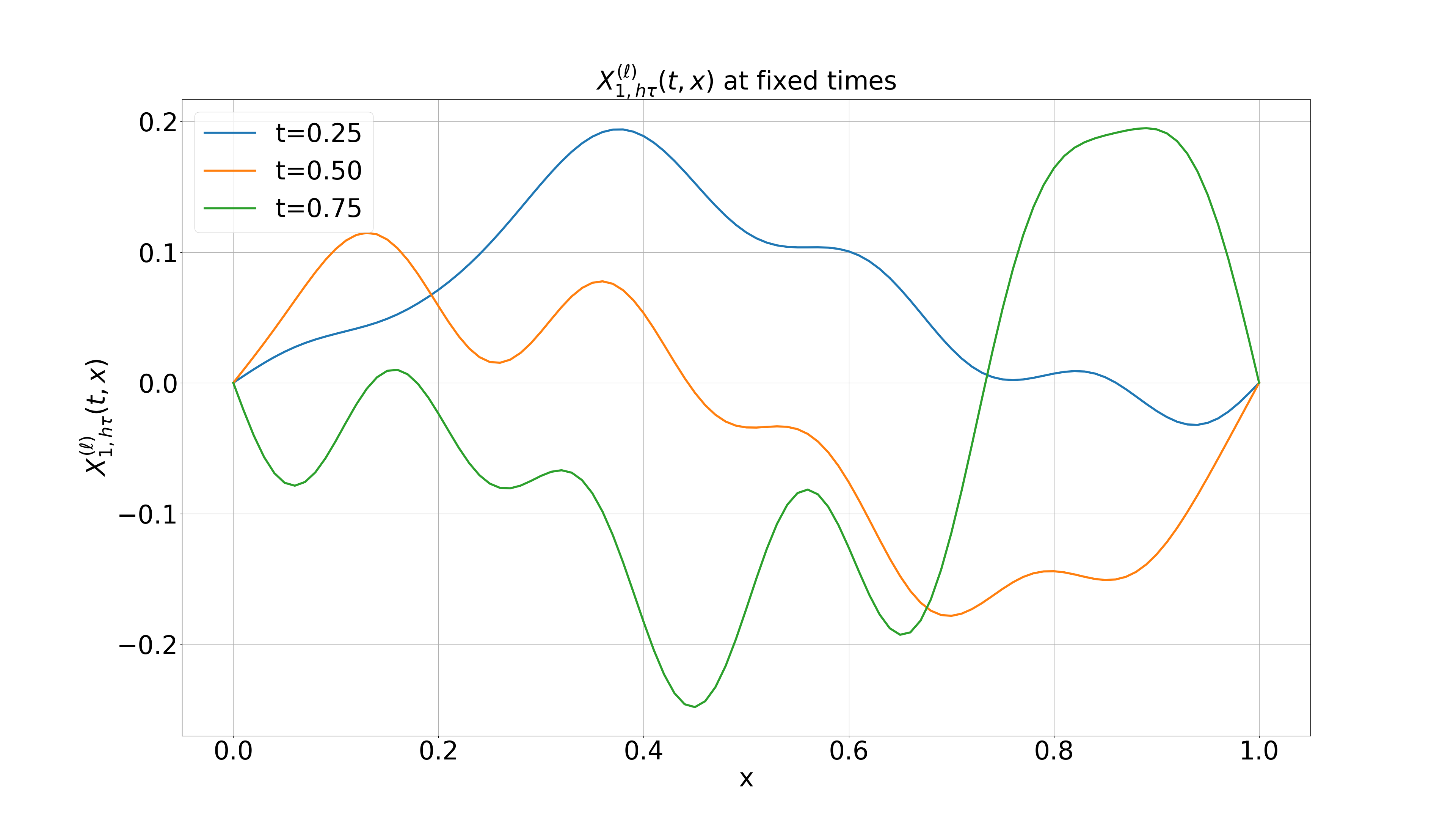}
			\caption{$x\mapsto X_{1,h\tau}^{(\ell)}(t,x)$}
		\end{subfigure}
		\\
		
		\begin{subfigure}[b]{0.32\textwidth}
			\includegraphics[width=\linewidth]{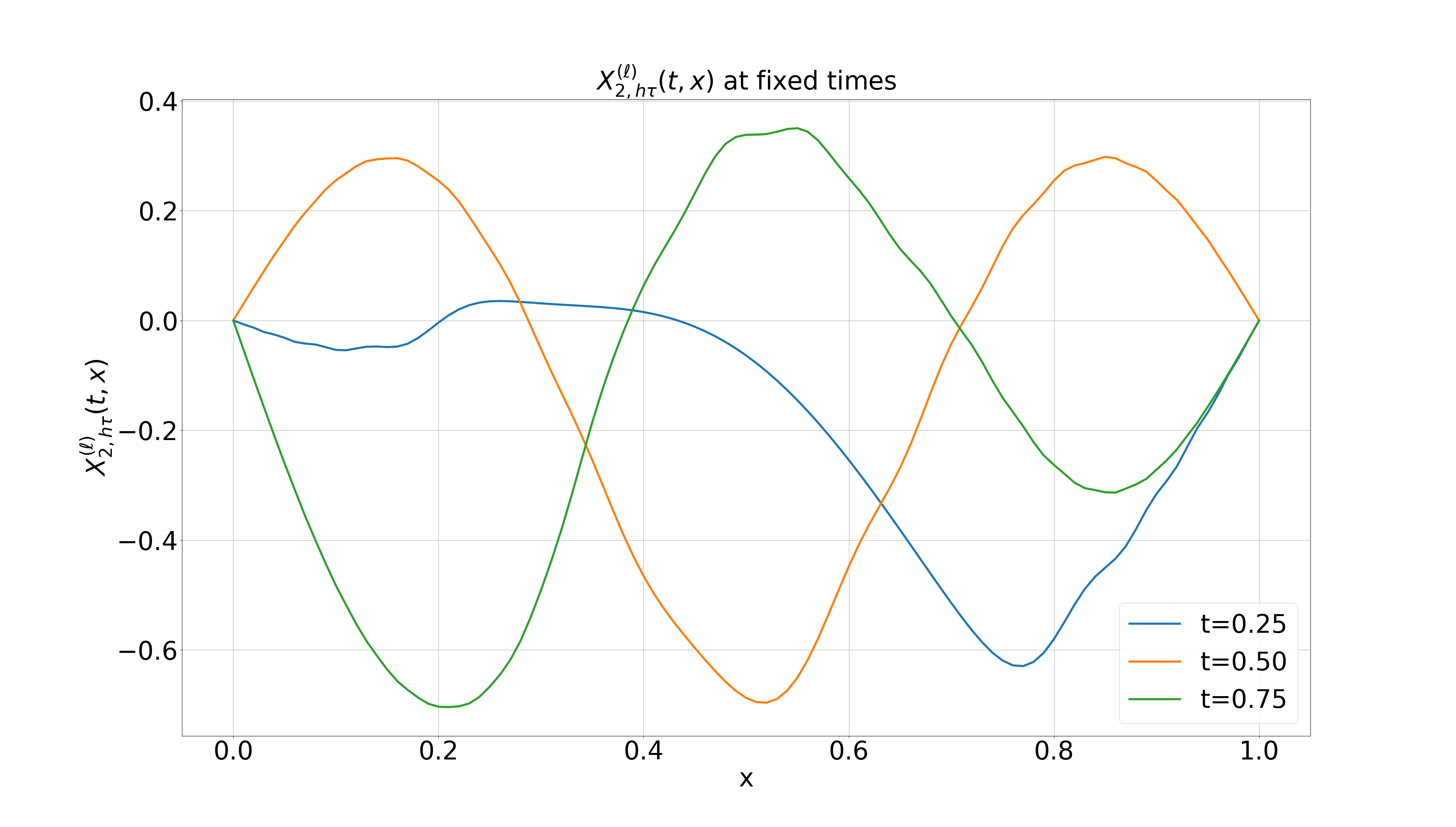}
			\caption{$x\mapsto X_{2,h\tau}^{(\ell)}(t,x)$}
		\end{subfigure}
		&
		\begin{subfigure}[b]{0.32\textwidth}
			\includegraphics[width=\linewidth]{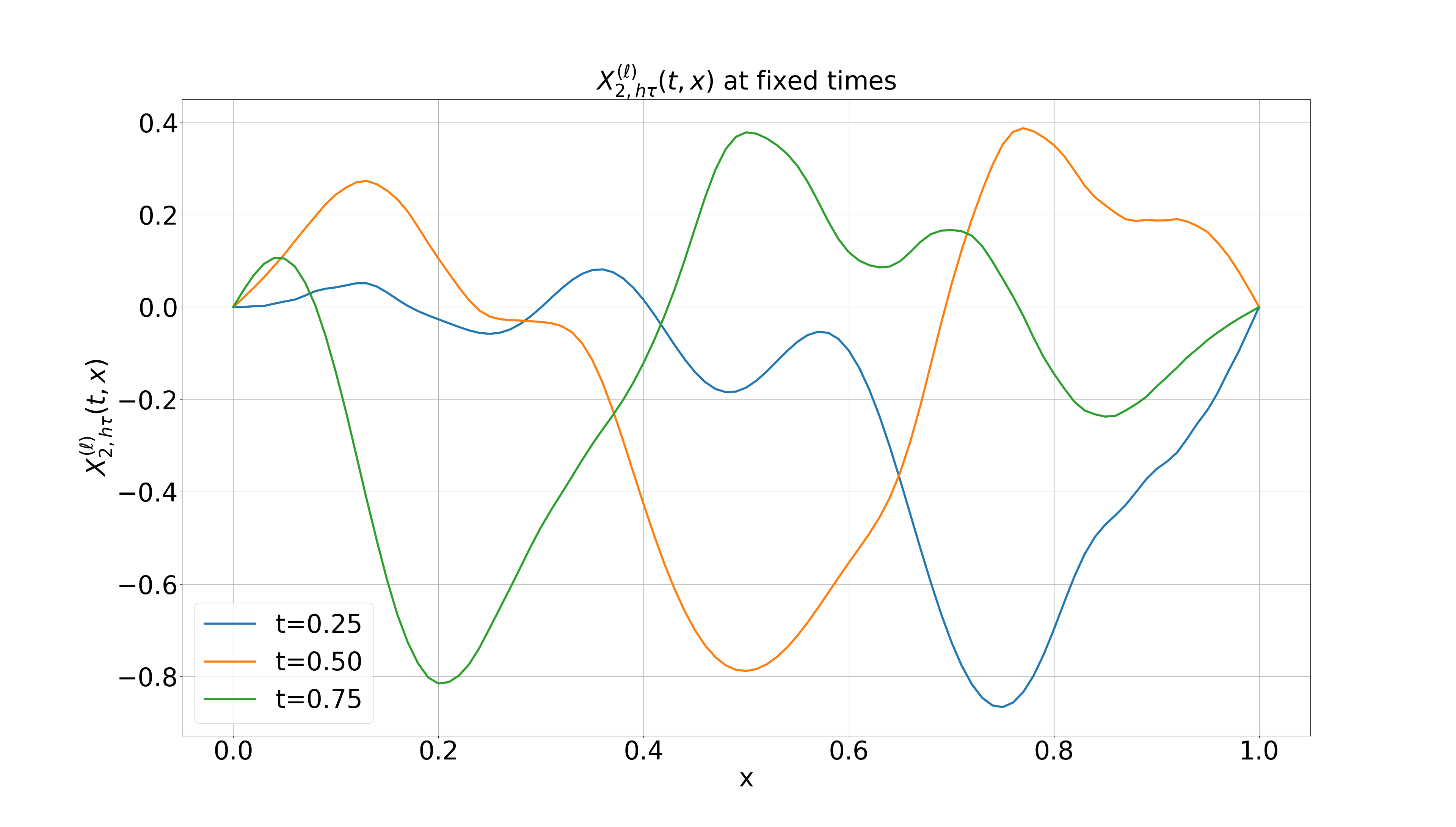}
			\caption{$x\mapsto X_{2,h\tau}^{(\ell)}(t,x)$}
			\label{j4}
		\end{subfigure}
		&
		\begin{subfigure}[b]{0.32\textwidth}
			\includegraphics[width=\linewidth]{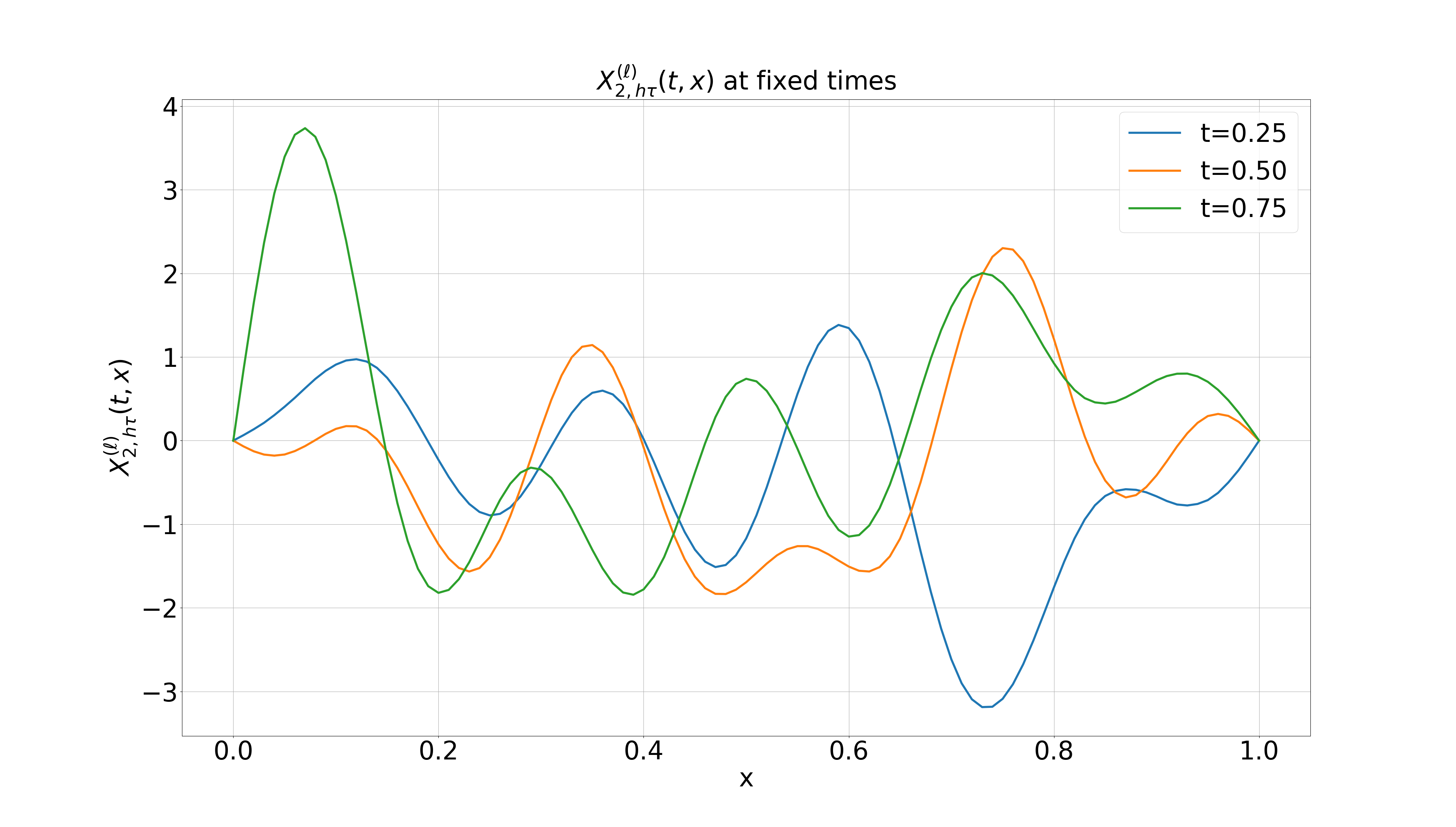}
			\caption{$x\mapsto X_{2,h\tau}^{(\ell)}(t,x)$}
		\end{subfigure}
		\\
		\end{tabular}
		\caption{Comparison of the iterates under three noise levels (columns). Rows show various profiles of a single path of a displacement iterate $X_{1, h\tau}^{(\ell)}(\cdot; \omega)$, and velocity iterate $X_{2,h\tau}^{(\ell)}(\cdot; \omega)$. In \emph{Row 1,2,3}: Displacement iterate $x\mapsto X_{1, h\tau}^{(\ell)}(t, x, \omega)$ and velocity iterate $x\mapsto X_{2, h\tau}^{(\ell)}(t, x, \omega)$, respectively,  for different times $t=0.25, 0.50, 0.75$.}
		\label{j22}
		\end{figure}
	\begin{figure}[htbp]
		\centering
		\setlength{\tabcolsep}{4pt}
		\renewcommand{\arraystretch}{1.2}
		\begin{tabular}{ccc}
			\textbf{Zero noise ($\sigma'=0$)} & \textbf{Small noise ($\sigma'=0.1\sigma$)} & \textbf{Large noise ($\sigma' = \sigma$)} \\
		
		\begin{subfigure}[b]{0.32\textwidth}
			\includegraphics[width=\linewidth]{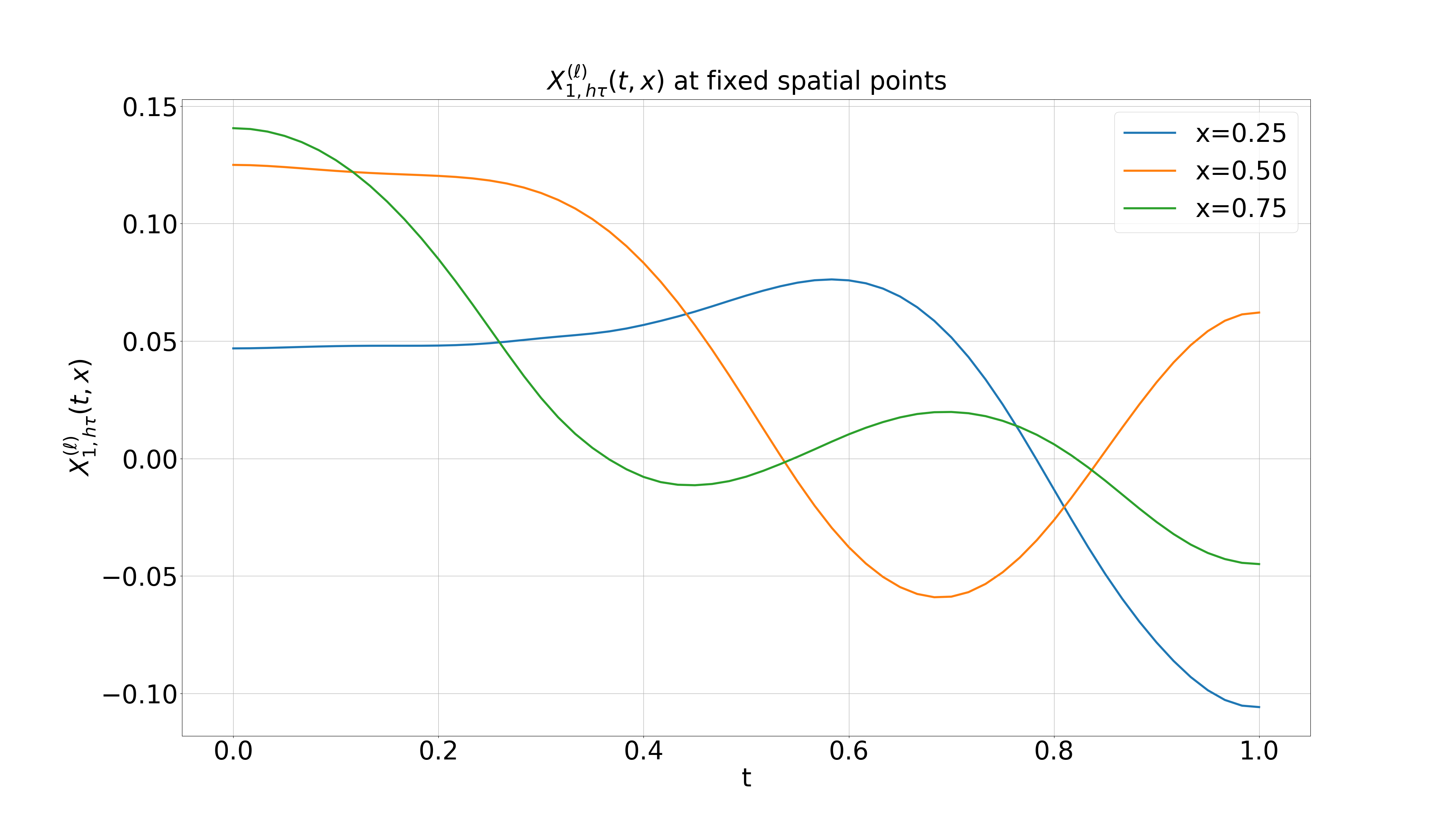}
			\caption{$t\mapsto X_{1, h\tau}^{(\ell)}(t,x)$}
		\end{subfigure}
		&
		\begin{subfigure}[b]{0.32\textwidth}
			\includegraphics[width=\linewidth]{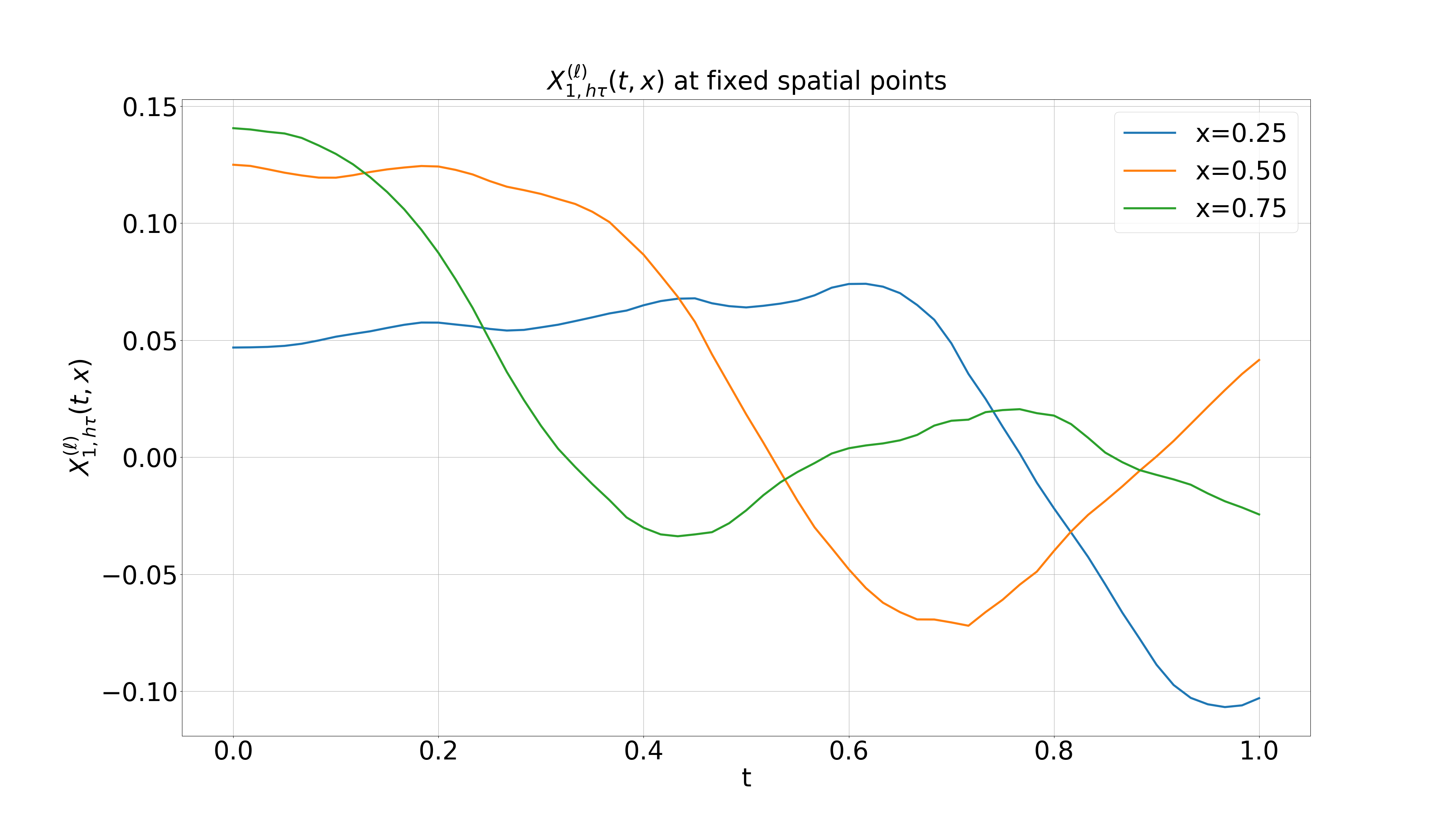}
			\caption{$t\mapsto X_{1, h\tau}^{(\ell)}(t,x)$}
		\end{subfigure}
		&
		\begin{subfigure}[b]{0.32\textwidth}
			\includegraphics[width=\linewidth]{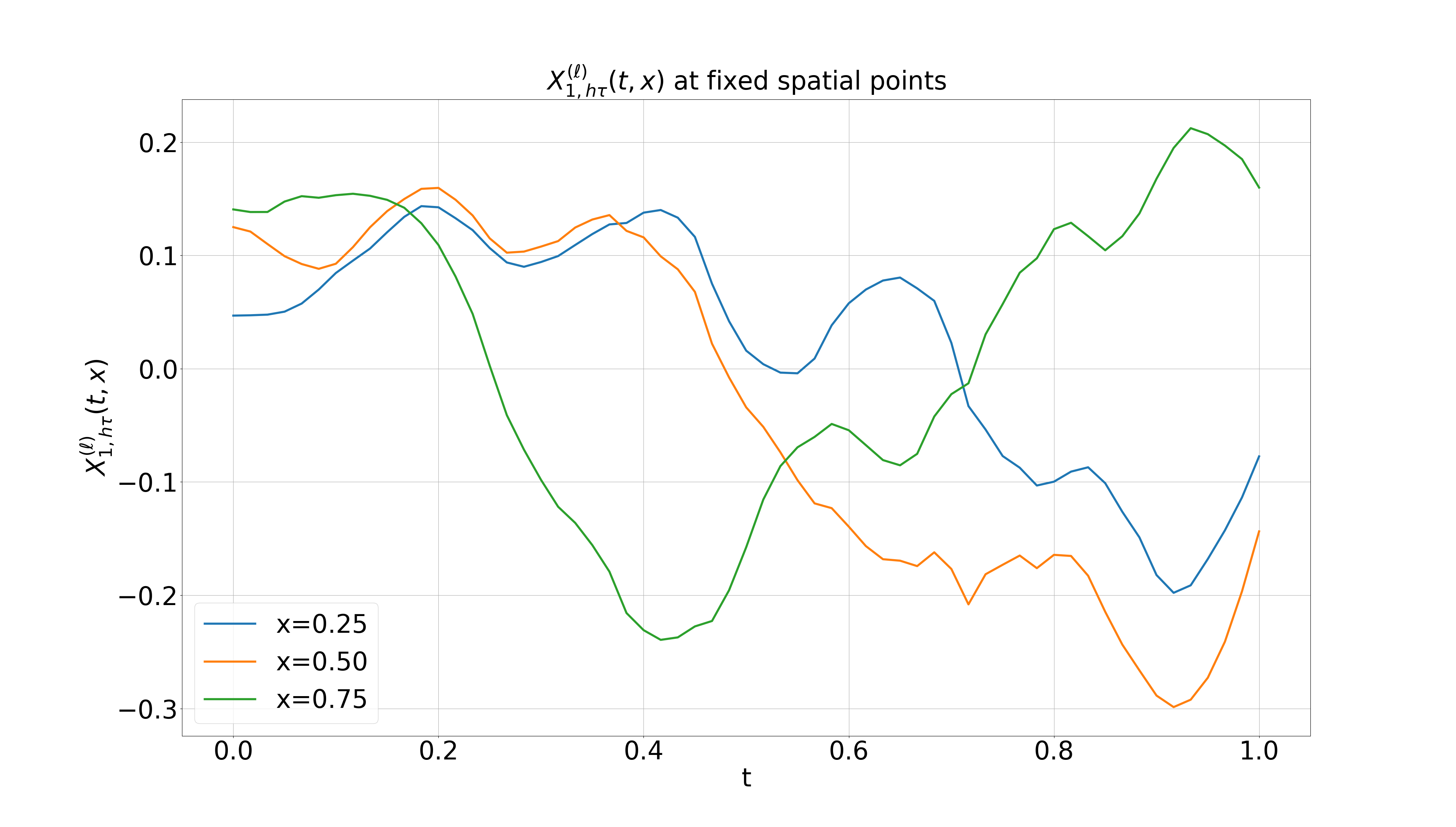}
			\caption{$t\mapsto X_{1, h\tau}^{(\ell)}(t,x)$}
		\end{subfigure}
		\\
		
		\begin{subfigure}[b]{0.32\textwidth}
			\includegraphics[width=\linewidth]{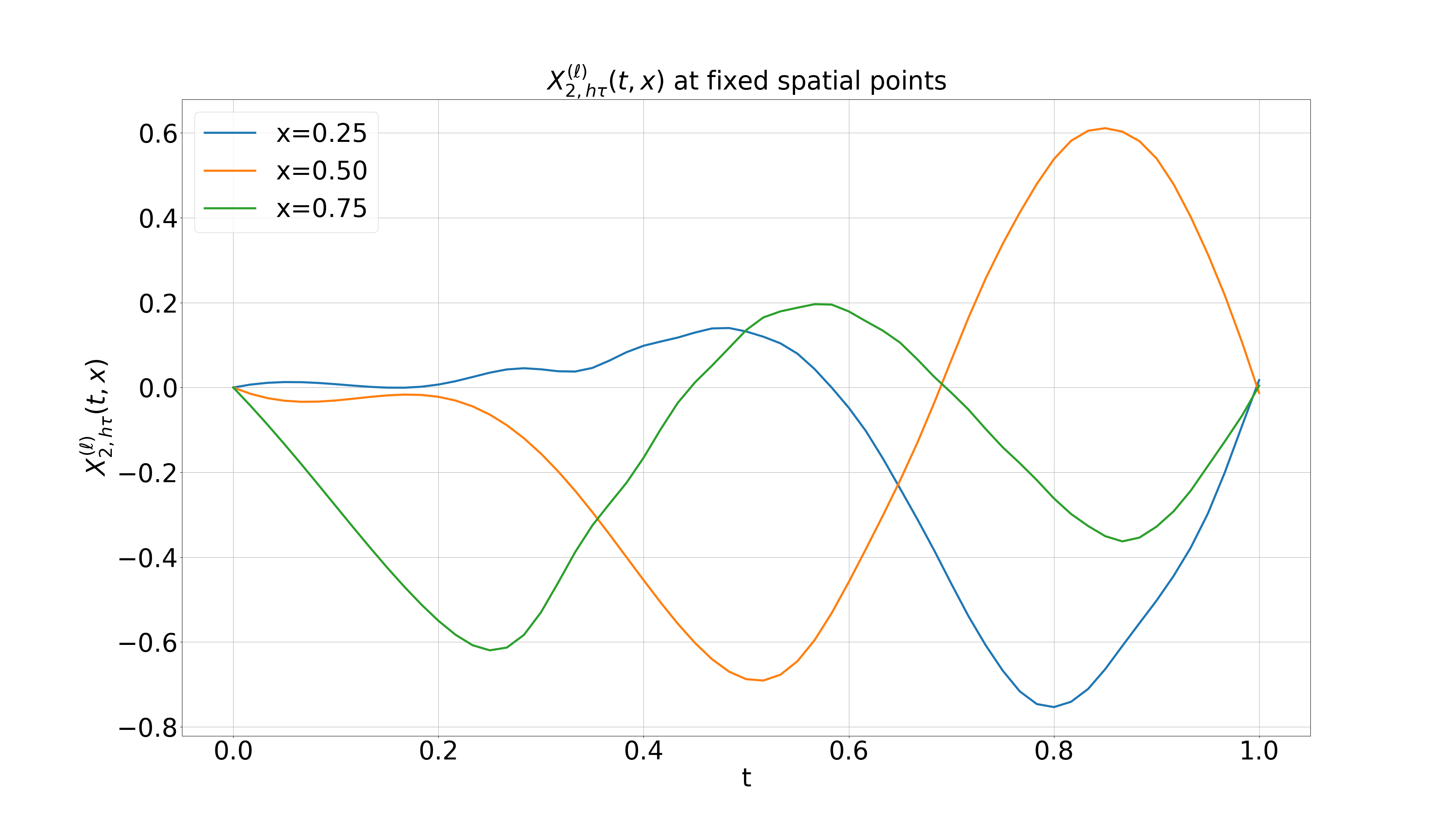}
			\caption{$t\mapsto X_{2, h\tau}^{(\ell)}(t,x)$}
		\end{subfigure}
		&
		\begin{subfigure}[b]{0.32\textwidth}
			\includegraphics[width=\linewidth]{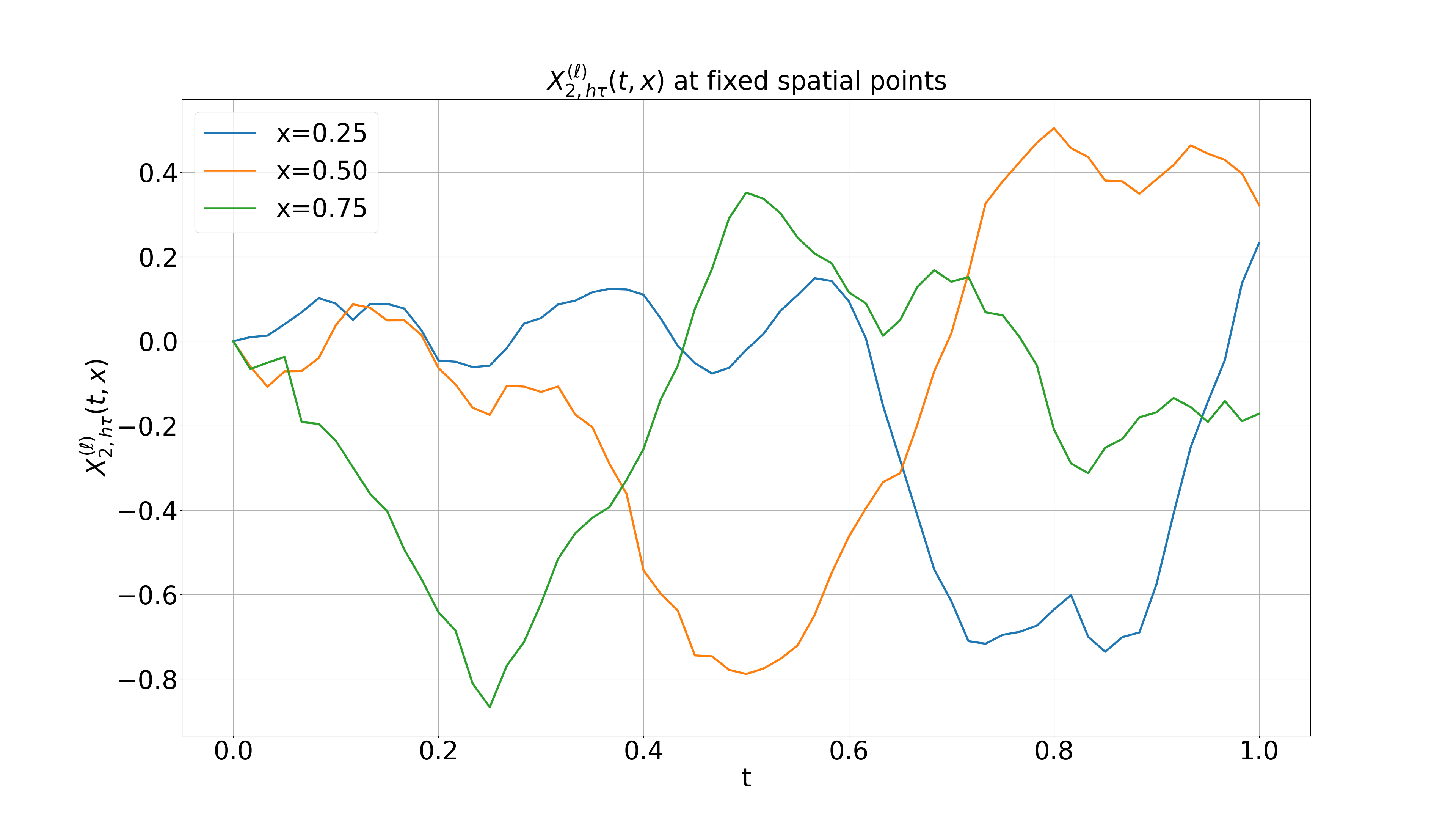}
			\caption{$t\mapsto X_{2, h\tau}^{(\ell)}(t,x)$}
			\label{j5}
		\end{subfigure}
		&
		\begin{subfigure}[b]{0.32\textwidth}
			\includegraphics[width=\linewidth]{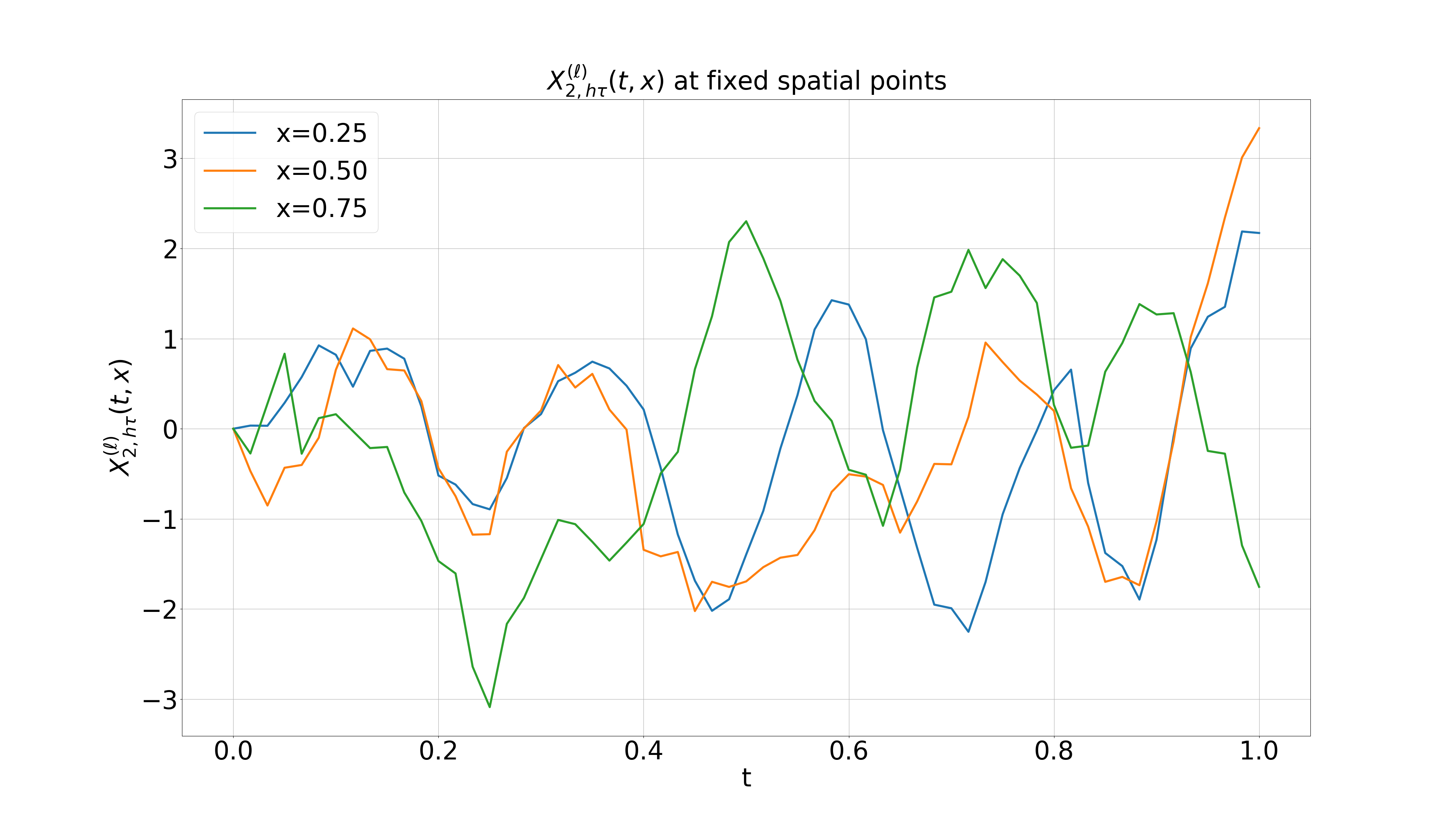}
			\caption{$t\mapsto X_{2, h\tau}^{(\ell)}(t,x)$}
		\end{subfigure}
		
	\end{tabular}
	
	\caption{Comparison of the iterates under three noise levels (columns). Rows show various profiles of the single path of the  displacement iterate $X_{1,h\tau}^{(\ell)}(\cdot; \omega)$, and velocity iterate $X_{2,h\tau}^{(\ell)}(\cdot; \omega)$. In \emph{Row 1,2,3}: Displacement iterate $t\mapsto X_{1, h\tau}^{(\ell)}(t, x, \omega)$ and velocity iterate $t\mapsto X_{2, h\tau}^{(\ell)}(t, x, \omega)$, respectively,  for different spatial points $x=0.25, 0.50, 0.75$.}
	\label{j33}
\end{figure}

\section{Preliminary results and Pontryagin's maximum principle}
\subsection{Notations for function spaces and assumptions on data} \label{subsection 2.1}
Let \(\bigl(\mathbb{K},(\!\cdot\,,\!\cdot\!)_{\mathbb{K}}\bigr)\) be a separable Hilbert space with norm 
\(\|\phi\|_{\mathbb{K}}=\left\langle \phi,\phi\right\rangle_{\mathbb{K}}^{1/2}\).  
On a bounded domain \(D\subset\mathbb{R}^d\) we set 
\(\mathbb{L}^2_x:=\mathbb{L}^2(D)\) with norm \(\|\cdot\|_{\mathbb{L}^2_x}\) and inner product \(\left\langle \cdot,\cdot \right\rangle_{\mathbb{L}^2_x}\), and define 
\[
\mathbb{H}_0^1:={H}_0^1(D), 
\quad
\mathbb{H}^i_x:=H^i(D)\cap \mathbb{H}_0^1\quad (i=2,3,4),
\]
each equipped with its usual norm \(\|\cdot\|_{\mathbb{H}^i_x}\).  
Let \((\Omega,\mathcal{F},\{\mathcal{F}_t\}_{t\in[0,T]},\mathbb{P})\) be a complete filtered probability space whose filtration is generated by $\mathbb{R}^m$-valued Wiener process \(W\) (augmented by all \(\mathbb{P}\)–null sets).  We write
\[
\mathbb{L}^2_{\mathbb{F}}(0,T;\mathbb{K})
=\Bigl\{X:\Omega\times[0,T]\to\mathbb{K}\ \text{be}\, \mathbb{F}\text{-adapted}\ \big|\ \mathbb{E}\bigg[\!\int_0^T\|X(t)\|_{\mathbb{K}}^2\,dt\bigg]<\infty\Bigr\},
\]
\[
\mathbb{L}^2_{\mathbb{F}}(\Omega;C([0,T];\mathbb{K}))
=\Bigl\{X:\Omega\times[0,T]\to\mathbb{K}\ \text{be}\, \mathbb{F}\text{-adapted, continuous}\ \big|\ \mathbb{E}\big[\!\sup_{t\in[0,T]}\|X(t)\|_{\mathbb{K}}^2\big]<\infty\Bigr\},
\]
and for each \(t\in[0,T]\),
\[
\mathbb{L}^2_{\mathcal{F}_t}(\Omega;\mathbb{K})
=\bigl\{\eta:\Omega\to\mathbb{K}\ \text{be}\, \mathcal{F}_t\text{-measurable}\ \big|\ 
\mathbb{E}\big[\|\eta\|_{\mathbb{K}}^2\big]<\infty\bigr\}.
\]
Finally, for brevity, we set
\[
\mathbb{L}^2_{t, x}:=L^2(0,T;\mathbb{L}^2_x),
\quad
\mathbb{L}^2_t\mathbb{K}:=L^2(0,T;\mathbb{K}),
\quad
\mathbb{L}^2_{\mathbb{F}}\mathbb{L}^2_{t, x}
:=L^2_{\mathbb{F}}(\Omega\times(0,T);\mathbb{L}^2_x),\]
\[\mathbb{L}^2_{\mathbb{F}}\mathbb{L}^2_t\mathbb{K}
:=L^2_{\mathbb{F}}(\Omega\times[0,T];\mathbb{K}), \qquad
\mathbb{L}^2_{\mathbb{F}}C_t\mathbb{K}=\mathbb{L}^2_{\mathbb{F}}(\Omega;C([0,T];\mathbb{K})),\qquad\text{and}\qquad \mathbb{L}^2_{\mathbb{F}}C_t^{1/2}\mathbb{K}=\mathbb{L}^2_{\mathbb{F}}(\Omega;C^{1/2}([0,T];\mathbb{K})).
\]
Note that for the sake of simplicity, throughout in the mathematical analysis of this paper, we set \( m = 1 \) in the case of  $\mathbb{R}^m$-valued Wiener process and $\gamma\in \mathbb{R}^m$. However, all results remain valid for any \( m \in \mathbb{N} \).
\subsection{Preliminary results for SPDE \eqref{1.4}}
Next, we define a weak variational solution to forward SPDE \eqref{1.2}.
	\begin{definition}\label{definition1}  Let $U, \sigma \in \mathbb{L}^2_{\mathbb{F}}\mathbb{L}^2_{t, x}$. We call the pair $(X_1, X_2)$ a weak variational solution of \eqref{1.4} on the interval $[0,T]$ with initial data $(X_{1,0}, X_{2,0})\in \mathbb{H}_0^1\times\mathbb{L}^2_x$ if the pair
	$(X_1, X_2) \in \mathbb{L}^2_{\mathbb{F}}C_t\mathbb{H}_0^1 \times \mathbb{L}^2_{
        \mathbb{F}}C_t\mathbb{L}^2_x$ satisfies the following variational formulation
	\begin{equation}
		\left\langle X_1(t), \phi\right\rangle = \int_0^t \left\langle X_2(t), \phi\right\rangle \,{\rm d}t + \left\langle X_{1,0}, \phi\right\rangle\quad \forall\, \phi \in \mathbb{L}_x^2,
	\end{equation}
    and for each $t \in [0,T]$ $\mathbb{P}$-a.s.
	\begin{equation}
		\left\langle X_2(t), \psi\right\rangle = -\int_0^t \left[ \left\langle \nabla X_1(t), \nabla \psi\right\rangle + \left\langle U(t), \psi\right\rangle \right] \,{\rm d}t + \int_0^t \left\langle \psi, (\sigma(t)+ \gamma X_1(t)) \,{\rm d}W(t)\right\rangle\,{\rm d}t + \left\langle X_{2,0}, \psi\right\rangle \quad \forall\, \psi \in \mathbb{H}_0^1.
	\end{equation}
\end{definition}
In the following lemma, we state a priori estimates in high-order Sobolev spaces.
	\begin{lem} \label{Lemma 2.1} Let $U, \sigma \in \mathbb{L}^2_{\mathbb{F}}\mathbb{L}^2_{t, x}$, $X_{1,0}\in\mathbb{H}_0^1$ and $X_{2,0}\in\mathbb{L}^2_x$. Then there exists a unique weak (variational) solution $(X_1,X_2)$  to \eqref{1.4} with given control $U$ in the sense of Definition~\ref{definition1}. Moreover, the following estimates holds:
    \begin{itemize}
        \item[1.] For all $X_{1,0}\in\mathbb{H}_0^1$, $X_{2,0}\in\mathbb{L}_x^2$, $U\in \mathbb{L}^2_{t,x}$, $\sigma\in\mathbb{L}^2_{\mathbb{F}}\mathbb{L}^2_{t,x}$,
        \begin{align}
            \mathbb{E} \left[ \sup_{0 \leq t \leq T} \left( \|X_1(t)\|_{\mathbb{H}_0^1}^{2} + \|X_2(t)\|_{\mathbb{L}^2_x}^{2} \right) \right] &\leq C(\|X_{1,0}\|_{\mathbb{H}_0^1}^2+\|X_{2,0}\|_{\mathbb{L}_x^2}^2+\mathbb{E}\big[\|U\|_{\mathbb{L}^{2}_{t, x}}^2\big] + \mathbb{E}\big[\|\sigma\|_{\mathbb{L}^{2}_{t, x}}^2\big]),\label{001}
        \end{align}
        \item[2.] For all $X_{1,0}\in\mathbb{H}_x^2$, $X_{2,0}\in\mathbb{H}_0^1$, $U\in \mathbb{L}^2_{\mathbb{F}}\mathbb{L}^2_{t}\mathbb{H}_0^1$, $\sigma\in\mathbb{L}^2_{\mathbb{F}}\mathbb{L}^2_{t}\mathbb{H}^1_0$,
        \begin{align}
            \mathbb{E} \left[ \sup_{0 \leq t \leq T} \left( \|X_1(t)\|_{\mathbb{H}_x^2}^{2} + \|X_2(t)\|_{\mathbb{H}^1_0}^{2} \right) \right] &\leq C(\|X_{1,0}\|_{\mathbb{H}_x^2}^2 +\|X_{2,0}\|_{\mathbb{H}_0^1}^2 +\mathbb{E}\big[\|U\|_{\mathbb{L}^2_t\mathbb{H}_0^1}^2\big] + \mathbb{E}\big[\|\sigma\|_{\mathbb{L}^2_t\mathbb{H}_0^1}^2\big]),\label{002}
        \end{align}
        \item[3.] For all $X_{1,0}\in\mathbb{H}_x^3$ with $\Delta X_{1,0}\in \mathbb{H}_0^1$, $X_{2,0}\in\mathbb{H}_x^2$, $U\in \mathbb{L}^2_{\mathbb{F}}\mathbb{L}^2_{t}\mathbb{H}_x^2$, $\sigma\in\mathbb{L}^2_{\mathbb{F}}\mathbb{L}^2_{t}\mathbb{H}^2_x$,
        \begin{align}
            \mathbb{E} \left[ \sup_{0 \leq t \leq T} \left( \|X_1(t)\|_{\mathbb{H}_x^3}^{2} + \|X_2(t)\|_{\mathbb{H}_x^2}^{2} \right) \right] &\leq C(\|X_{1,0}\|_{\mathbb{H}_x^3}^2 +\|X_{2,0}\|_{\mathbb{H}_x^2}^2 +\mathbb{E}\big[\|U\|_{\mathbb{L}^2_t\mathbb{H}_x^2}^2\big] + \mathbb{E}\big[\|\sigma\|_{\mathbb{L}^2_t\mathbb{H}_x^2}^2\big]).\label{0002}
        \end{align}
    \end{itemize}
\end{lem}
\begin{proof}
	For the well-posedness result, we refer to \cite[Lemma 8.1]{Chow2014}. For a priori estimates, we can follow similar arguments as in the proof of \cite[Lemma 3.2]{FengPandaProhl2024}. We leave its proof to the interested reader.
\end{proof}

For convenience, we define a solution operator such that $\mathcal{X}[U]=(\mathcal{X}_1[U],\mathcal{X}_2[U])$, where $(\mathcal{X}_1[U],\mathcal{X}_2[U])$ is the unique weak variational solution to \eqref{1.4} with given distributed control $U\in\mathbb{L}^2_{\mathbb{F}}\mathbb{L}^2_{t, x}$.
\subsection{Assumptions on data}\label{sec assumption}
For our main result concerning the rate of convergence ({\em i.e.}, Theorem~\ref{convergence of gradient descent method}) of the numerical algorithms ({\em i.e.}, Algorithms~\ref{tt1} and~\ref{tt3}), we require the following set of assumptions on the data.

\begin{assumption}\label{BB}
Let $
X_{1,0}\in\mathbb{H}_x^3$ with $\Delta X_0\in \mathbb{H}_{0}^1$,\,$
X_{2,0}\in\mathbb{H}_x^2,\, 
\widetilde{X}\in C_t\mathbb{H}_x^2\cap C_t^{1/2}\mathbb{H}^1_0,
$
and
$
\sigma\in \mathbb{L}^2_{\mathbb{F}}C_t^{1/2}\mathbb{H}^1_0
\;\cap\;
\mathbb{L}^2_{\mathbb{F}}\mathbb{L}^2_t\mathbb{H}_x^2.
$
\end{assumption}
However, the setup of our main algorithms ({\em i.e.}, Algorithms~\ref{tt1} and~\ref{tt3}) remains valid under the following weaker regularity assumptions on the data.
\begin{assumption}\label{CC}
Let $
X_{1,0}\in\mathbb{H}_0^1$,\,$
X_{2,0}\in\mathbb{H}_0^1,\, 
\widetilde{X}\in C_t\mathbb{H}_0^1,
$
and
$
\sigma\in \mathbb{L}^2_{\mathbb{F}}C_t\mathbb{H}^1_0.
$
\end{assumption}
 
\subsection{Preliminary results for SLQ problem \eqref{1.3}-\eqref{1.4}}

\noindent In the following proposition, we discuss the well-posedness of the optimal tuple $(X_1^*, X_2^*, U^*)$ to the {\bf SLQ} problem \eqref{1.3}-\eqref{1.4}. 
\begin{proposition}[Existence of a unique optimal tuple]\label{existence of a unique optimal pair} Let Assumption~\ref{CC} hold. Then there exists a unique optimal tuple $(X_1^*, X_2^*, U^*)\in \mathbb{L}^2_{\mathbb{F}}\mathbb{C}_{t}\mathbb{H}_0^1\times\mathbb{L}^2_{\mathbb{F}}\mathbb{C}_{t}\mathbb{L}^2_x\times\mathbb{L}^2_{\mathbb{F}}\mathbb{L}^2_{t,x}$ to {\bf SLQ} problem \eqref{1.3}-\eqref{1.4}. Moreover, the following bound holds;
	\begin{align}\label{bound for optimal control}
		\mathbb{E}\bigg[\sup_{t\in[0,T]}(\|X_1^*(t)\|_{\mathbb{H}_0^1}^2+ \|X_2^*(t)\|_{\mathbb{L}^2_x}^2) +\|U^*\|_{\mathbb{L}^2_{t, x}}^2\bigg]\le C\big(\|X_{1,0}\|_{\mathbb{H}^1_x}^2+\|X_{2,0}\|_{\mathbb{L}^2_x}^2+\|\widetilde{X}\|_{C_t\mathbb{L}^2_x}^2+\mathbb{E}\big[\|\sigma\|_{\mathbb{L}^2_{t, x}}^2\big]\big).
	\end{align}
\end{proposition}
\begin{proof}This proof is standard. For the existence and uniqueness of the optimal control tuple $(X_1^*, X_2^*,  U^*)$, one can follow similar arguments as in the proof of \cite[Theorem 1.43]{Hinze et al. book}, for more details see \cite{LuZhang2021}. For the estimate \eqref{bound for optimal control}, one can follow similar lines as in the proof of \cite[Lemma 4.2]{ChaudharyProhl} and leave details to the interested reader.
\end{proof}
\begin{lem}[Existence and uniqueness of a solution to {\bf BSPDE \eqref{1.5}}]\label{Lemma 2.3} Let Assumption~\ref{CC} hold. There exists a unique weak solution $(Y_1, Y_2, Z_1,Z_2)\in \big(\mathbb{L}^2_{\mathbb{F}}\mathbb{L}^2_t(\mathbb{L}^2_x\times\mathbb{H}_0^1)\big)^2$ to {\bf BSPDE} \eqref{1.5}. Moreover, there exists $C>0$ such that
{\begin{align}\label{today020}
		&\mathbb{E}\bigg[\sup_{t\in[0,T]}\big[\|Y_1(t)\|_{\mathbb{L}^2_x}^2+\|\nabla Y_2(t)\|_{\mathbb{L}^2_x}^2\big]\bigg]+\mathbb{E}\bigg[\int_0^{T}\|Z_1(t)\|_{\mathbb{L}^2_x}^2\,{\rm d}t +\int_0^T\|\nabla Z_2(t)\|^2_{\mathbb{L}_x^2}\,{\rm d}t\bigg]\notag\\\qquad&\le C \mathbb{E}\big[\|X_1^*-\widetilde{X}\|_{\mathbb{L}^2_{t, x}}^2+\beta^2\|X_1^*(T)-\widetilde{X}(T)\|_{\mathbb{L}^2_x}^2\big],
	\end{align}
and
\begin{align}\label{space regularity H-2 bound for adjoint Y}
	&\mathbb{E}\big[\sup_{t\in[0,T]}\big[\|\nabla Y_1(t)\|_{\mathbb{L}^2_x}^2+\|\Delta Y_2(t)\|_{\mathbb{L}^2_x}^2\big]\big]+\mathbb{E}\bigg[\int_0^{T}\|\nabla Z_1(t)\|_{\mathbb{L}^2_x}^2\,{\rm d}t +\int_0^T\|\Delta Z_2(t)\|^2_{\mathbb{L}^2_x}\,{\rm d}t\bigg]\notag\\&\qquad\le C \mathbb{E}\big[\|\nabla X_1^*-\widetilde{X}\|_{\mathbb{L}^2_{t, x}}^2+\beta^2\|\nabla\big( X_1^*(T)-\widetilde{X}(T)\big)\|_{\mathbb{L}^2_x}^2\big].
\end{align}}
\end{lem}

\begin{proof} The derivation of existence and uniqueness follows from a standard Galerkin approximation argument, and we refer to \cite{MaYong1999, PardouxPeng1990, Tessitore} for more details related to well-posedness of \textbf{BSDE}. To first obtain estimate \eqref{today020}, we apply It\^o's formula to $f(Y_1)=\frac{1}{2}\|Y_1\|_{\mathbb{L}^2_x}^2$, which leads to $\mathbb{P}$-almost surely, $s\in[0,T],$
\begin{align}
	 -\|Y_1(s)\|^{2}_{\mathbb{L}^2_x}+\beta^2\|X_1^*(T)-\widetilde{X}(T)\|_{\mathbb{L}^2_x}^2&=2\bigg[\int_{s}^{T}\left\langle \nabla Y_2(t), \nabla Y_1(t)\right\rangle\,{\rm d}t\notag\\&\qquad-\int_s^T\left\langle \gamma Z_2(t), Y_1(t)\right\rangle\,{\rm d}t-\int_s^{T}\left\langle \big(X_1^*(t)-\widetilde{X}(t)\big), Y_1(t)\right\rangle\,{\rm d}t\notag \\
	 &\qquad+ \int_s^T \left\langle  Z_1(t), Y_1(t)\right\rangle\,{\rm d}W(t) \bigg]+\int_s^T\|Z_1(t)\|^2_{\mathbb{L}^2_x}\,{\rm d}t.\label{jj1}
 \end{align}
Again by applying It\^o's formula $Y_2\to \|\nabla Y_2\|_{\mathbb{L}^2_x}^2$, we have $\mathbb{P}$-almost surely, for all $s\in[0,T]$,
\begin{align}
	\|\nabla Y_2(s)\|^2&= 2\bigg[\int_s^T\left\langle \nabla Y_1(t), \nabla Y_2(t)\right\rangle\,{\rm d}t  +  \int_{s}^{T}\left\langle \nabla Z_2(t), \nabla Y_2(t)\right\rangle\,{\rm d}W(t)\bigg]-\int_s^{T}\|\nabla Z_2(t)\|^2\,{\rm d}t.\label{jj2}
\end{align}
From \eqref{jj1}-\eqref{jj2}, we obtain that $\mathbb{P}$-almost surely, for all $s\in[0,T]$,
\begin{align*}
	\|Y_1(s)\|^{2}_{\mathbb{L}^2_x}&+\|\nabla Y_2(s)\|^2_{\mathbb{L}^2_x}+\int_s^{T}\|Z_1(t)\|_{\mathbb{L}^2_x}^2\,{\rm d}t +\int_s^T\|\nabla Z_2(t)\|^2_{\mathbb{L}^2_x}\,{\rm d}t=  \beta^2\|X_1^*(T)-\widetilde{X}(T)\|_{\mathbb{L}^2_x}^2\\&\qquad+ 2\bigg[\int_{s}^{T}\left\langle \nabla Z_2(t), \nabla Y_2(t)\right\rangle\,{\rm d}W(t)+\int_s^T\left\langle \gamma Z_2(t), Y_1(t)\right\rangle\,{\rm d}t+\int_s^{T}\left\langle X_1^*(t)-\widetilde{X}(t), Y_1(t)\right\rangle\,{\rm d}t \\&\qquad-\int_s^T \left\langle  Z_1(t), Y_1(t)\right\rangle\,{\rm d}W(t)\bigg].
\end{align*}
As an application of Young's inequality, as well as BDG inequality and Gronwall's inequality, we can  conclude that there exists $C>0$ such that
\begin{align}\label{today02}
	&\mathbb{E}\big[\sup_{t\in[0,T]}\big[\|Y_1(t)\|_{\mathbb{L}^2_x}^2+\|\nabla Y_2(t)\|_{\mathbb{L}^2_x}^2\big]\big]+\mathbb{E}\bigg[\int_0^{T}\|Z_1(t)\|_{\mathbb{L}^2_x}^2\,{\rm d}t +\int_0^T\|\nabla Z_2(t)\|^2_{\mathbb{L}^2_x}\,{\rm d}t\bigg]\notag\\&\qquad\le C \mathbb{E}\big[\|X_1^*-\widetilde{X}\|_{\mathbb{L}^2_{t, x}}^2+\beta^2\|X_1^*(T)-\widetilde{X}(T)\|_{\mathbb{L}^2_x}^2\big].
	\end{align}
Similarly, we apply It\^o's formula to 
\[
Y_{1,n} \;\mapsto\; \|\nabla Y_{1,n}\|_{\mathbb{L}^2_x}^2 
\quad \text{and} \quad 
Y_{2,n} \;\mapsto\; \|\Delta Y_{2,n}\|_{\mathbb{L}^2_x}^2,
\]
where $Y_{1,n}$ and $Y_{2,n}$ denote the Galerkin approximations of $Y_1$ and $Y_2$, respectively. 
This allows to avoid the boundary terms arising in the integration by parts formula, as used the in the proof of \cite[Lemmas~3.6 and~3.7]{KronerKunischVexler2009}. 
By passing to the limit it then yields the desired estimate \eqref{space regularity H-2 bound for adjoint Y}.

\end{proof}
\subsection{Pontryagin's maximum principle} To derive the Pontryagin's maximum principle, we need the Fréchet derivative of the solution operators $\mathcal{X}_i[\cdot]$, for $i=1,2$. To find this, we proceed as follows.  For given $V\in \mathbb{L}^2_{\mathbb{F}}\mathbb{L}^2_{t, x}$, let $(\mathcal{X}_1^0[V], \mathcal{X}_2^0[V])\equiv(X_1^0, X_2^0)$ be the unique solution to the following auxiliary SPDE system:
\begin{align}\label{1.9}
	\begin{cases} 
		\,{\rm d}X_1^0(t) = X_2^0(t) \,{\rm d}t & \text{in } D \times (0,T], \\
		\,{\rm d} X_2^0(t) = (\Delta X_1^0(t) + V(t)) \,{\rm d}t + \gamma X_1^0(t) \,{\rm d}W(t) & \text{in } D \times (0,T], \\
		X_1^0(0) = X_2^0(0)= 0 & \text{in } D, \\
		X_1^0(t) = 0 & \text{on } \Gamma \times (0, T].
	\end{cases}
\end{align}
Note that in equation~\eqref{1.9} the noise coefficient \(\sigma\) and the initial data are set to zero, which is in contrast to equation~\eqref{1.4}. Consequently the solution map \(U\mapsto \mathcal{X}_i[U]\) is affine (indeed linear in the control increment) and one has
\begin{align}\label{Goud}
\mathcal{X}_i[U+V]=\mathcal{X}_i[U]+\mathcal{X}_i^0[V],\qquad i=1,2,
\end{align}
for all \(U,V\in\mathbb{L}^2_{\mathbb{F}}\mathbb{L}^2_{t,x}\), where \(\mathcal{X}_i^0[V]\) denotes the solution corresponding to zero initial data and zero noise with control \(V\). Hence the Fréchet derivatives of the solution operators at \(U\) for $i=1,2,$ are given by 
\begin{align}\label{toaday}
\mathcal{D}_U \mathcal{X}_i[U]=\mathcal{X}_i^0[U]\qquad\forall\,U\in\mathbb{L}^2_{\mathbb{F}}\mathbb{L}^2_{t,x}.
\end{align}
\begin{remark}
     We define the reduced cost function $\hat{\mathcal{J}}: \mathbb{L}^2_{\mathbb{F}}\mathbb{L}_{t,x}^2\to \mathbb{R}$ as follows:
\begin{align*}
	\hat{\mathcal{J}}(U)=\frac{1}{2}\mathbb{E}\bigg[\int_0^T\big(\|\mathcal{X}_1[U](t)-\widetilde{X}(t)\|_{\mathbb{L}^2_x}^2+\alpha\|U(t)\|_{\mathbb{L}^2_x}^2\big) \,{\rm d}t +\beta\|\mathcal{X}_1[U](T)-\widetilde{X}(T)\|_{\mathbb{L}^2_x}^2\bigg],
\end{align*} 
where $(\mathcal{X}_1[U], \mathcal{X}_2[U])\equiv (X_1, X_2)$ is the unique weak variational solution to the following SPDE \eqref{1.4} with the given distributed control $U$.
\end{remark}
In the following theorem, we derive Pontryagin's maximum principle, which provide the optimality condition~\eqref{1.6} and an integral identity~\eqref{today0001}. The optimality conditions enhance spatial regularity (see Proposition~\ref{Spatial regularity of optimal control}), while the integral identity plays a pivotal role in the error analysis of the spatial discretization ${\bf SLQ}_h$ (see Theorem~\ref{Thm 3.8}).
\begin{thm}[Pontryagin's maximum principle]\label{Pontryagin's Maximum Principle}
	Let Assumption~\ref{CC} hold. Let $(X^*_1, X_2^*,  U^*)$ be the unique optimal control tuple for the {\bf SLQ} problem \eqref{1.3}-\eqref{1.4}, and let the quadruple $(Y_1, Y_2, Z_1, Z_2)$ be the solution to the BSPDE \eqref{1.5}. Then the optimality condition \eqref{1.6} holds. Moreover, the following integral identity holds: for all $V\in\mathbb{L}^2_{\mathbb{F}}\mathbb{L}^2_{t, x}$,
	\begin{align}\label{today0001}
		\mathbb{E} \bigg[\int_0^T\big[\left\langle X_1^*(t)-\widetilde{X}(t),\mathcal{X}_1^0[V](t)\right\rangle + \alpha\left\langle U^*(t), V(t)\right\rangle\big]\,{\rm d}t \bigg] +\beta\mathbb{E}\big[\left\langle X_1^*(T),\mathcal{X}_1^0[V](T)\right\rangle\big]=0.
	\end{align} 
\end{thm} 

\begin{proof} 
	Since $(\mathcal{X}_1[U^*], U^*) \equiv (X_1^*, X_2^*, U^*)$ is the unique optimal control tuple for the {\bf SLQ} problem \eqref{1.3}-\eqref{1.4}, we then obtain the following variational equality
	\begin{align}\label{today0}
		\left\langle \mathcal{D}_{U}\hat{\mathcal{J}}(U^*), V\right\rangle_{\mathbb{L}^2_{\mathbb{F}}\mathbb{L}^2_{t, x}} = 0 \quad \forall\, V \in \mathbb{L}^2_\mathbb{F}\mathbb{L}^2_{t, x}.
	\end{align}
	A straightforward computation with the help of the identity \eqref{toaday} yields
	\begin{align}\label{today00}
		\left\langle \mathcal{D}_{U}\hat{\mathcal{J}}(U^*), V\right\rangle_{\mathbb{L}^2_{\mathbb{F}}\mathbb{L}^2_{t, x}}&= \mathbb{E} \bigg[\int_0^T\big[\left\langle  X_1^*(t)-\widetilde{X}(t),\mathcal{X}_1^0[V](t)\right\rangle + \alpha\left\langle U^*(t), V(t)\right\rangle\big]\,{\rm d}t \bigg]\notag\\&\qquad+\beta\mathbb{E}\big[\left\langle X_1^*(T),\mathcal{X}_1^0[V](T)\right\rangle\big].
	\end{align}    
Let $V\in \mathbb{L}^2_{\mathbb{F}}\mathbb{L}^2_{t, x}$. By applying It\^o's product formula to $(Y_1,\mathcal{X}_1^0[V])\to \left\langle Y_{1},\mathcal{X}_1^0[V]\right\rangle$, we obtain $\mathbb{P}$- almost surely
	\begin{align}\label{today1}
		\left\langle  Y_1(T), \mathcal{X}_1^0[V](T) \right\rangle - \left\langle  Y_1(0), \mathcal{X}_1^0[V](0) \right\rangle = & \int_0^T \left\langle  Y_1(t), \mathcal{X}_2^0[V](t) \right\rangle \,{\rm d}t + \int_0^T \left\langle  \nabla \mathcal{X}_1^0[V](t), \nabla Y_2(t) \right\rangle \,{\rm d}t\notag \\
		& - \int_0^T \left\langle \gamma Z_2(t), \mathcal{X}_1^0[V](t)\right\rangle \,{\rm d}t - \int_0^T \left\langle  X_1^*(t)-\widetilde{X}(t), \mathcal{X}_1^0[V](t) \right\rangle \,{\rm d}t\notag \\
		& + \int_0^T \left\langle Z_1(t), \mathcal{X}_1^0[V](t)\right\rangle \,{\rm d}W(t).
	\end{align}
	Similarly, by applying It\^o's product formula to $(Y_2,\mathcal{X}_2^0[V])\to \left\langle Y_{2},\mathcal{X}_2^0[V]\right\rangle$, we obtain $\mathbb{P}$-almost surely
	\begin{align}\label{today2}
		\left\langle  Y_2(T), \mathcal{X}_2^0[V](T) \right\rangle - \left\langle  Y_2(0), \mathcal{X}_2^0[V](0) \right\rangle = & - \int_0^T \left\langle  \nabla Y_2(t), \nabla \mathcal{X}_1^0[V](t) \right\rangle \,{\rm d}t + \int_0^T \left\langle  Y_2(t), V(t) \right\rangle \,{\rm d}t\notag \\
		& + \int_0^T \left\langle  \mathcal{X}_1^0[V](t), Y_2(t)\right\rangle \,{\rm d}W(t) - \int_0^T \left\langle  Y_1(t), \mathcal{X}_2^0[V](t) \right\rangle \,{\rm d}t \notag\\
		&+ \int_0^T \left\langle  \gamma Z_2(t), \mathcal{X}_2^0[V](t) \right\rangle \,{\rm d}W(t) + \int_0^T \left\langle \mathcal{X}_1^0[V](t), \gamma Z_2(t)\right\rangle \,{\rm d}t.
	\end{align}
	By adding \eqref{today1} and \eqref{today2}, using the facts $\mathcal{X}_{1}^0[V](0)=\mathcal{X}_2^0[V](0)=Y_{2}(T)=0$ and $Y_1(T)=\beta(X_1^*(T)-\widetilde{X}(T))$, and taking the expectation, we obtain for all $V \in \mathbb{L}^2_\mathbb{F}\mathbb{L}^2_{t, x}$,
	\begin{align}\label{today000}
		\mathbb{E} \bigg[\int_0^T \left\langle  X_1^*(t)-\widetilde{X}(t), \mathcal{X}_1^0[V](t) \right\rangle \bigg] + \beta\mathbb{E}\big[\left\langle \big(X^*_1(T)-\widetilde{X}(T)\big), \mathcal{X}_1^0[V](T)\right\rangle\big] =\mathbb{E} \bigg[\int_0^T \left\langle  Y_2(t), V(t) \right\rangle \,{\rm d}t \bigg].
	\end{align}
	Combining \eqref{today0}, \eqref{today00}, and \eqref{today000}, we conclude that 
	\begin{align*}
		\alpha U^*=- Y_2\qquad\text{in}\, \mathbb{L}^2_{\mathbb{F}}\mathbb{L}^2_{t}\mathbb{H}_0^1.
	\end{align*}
	This completes the proof.
\end{proof}
\begin{remark}[Vanishing on the boundary and enhanced spatial regularity]\label{vanishing on the boundary}
In Proposition~\ref{existence of a unique optimal pair}, the optimal control \(U^*\) is shown to satisfy $U^*\in \mathbb{L}^2_{\mathbb F}\mathbb{L}^2_{t,x}.$ However, the optimality condition~\eqref{1.6} yields the improved spatial regularity $U^*\in \mathbb{L}^2_{\mathbb F}C_t\mathbb{H}_0^1,$ which is essential for the error estimates in Section~\ref{Section 3}.  In particular, the optimal control \(U^*\) vanishes on the boundary of \(D\) in the sense of traces.
\end{remark}

\begin{remark}[Equivalent formulation]
Theorem~\ref{Pontryagin's Maximum Principle} shows that solving the \(\mathbf{SLQ}\) problem \eqref{1.3}--\eqref{1.4} is equivalent (in the sense of necessary and sufficient optimality conditions) to solving the optimality system consisting of the state SPDE \eqref{1.4}, the adjoint BSPDE \eqref{1.5}, and the optimality condition \eqref{1.6}. As it will be seen in Section~\ref{Section 5}, we introduce a space–time discretized version of this system for practical implementation; see in particular Proposition~\ref{fully discrete Pontryagin's maximum principle}.
\end{remark}

\section{Space discretization}\label{Section 3}
We partition the bounded domain $D \subset \mathbb{R}^d$ via a regular triangulation $\mathcal{T}_h$ into elements $K$ with maximum mesh
\[
h := \max_{K\in\mathcal T_h}\,\mbox{diam}(K).
\]
We work in the following discrete space
\[
\mathbb{V}_h \;:=\; \{\phi\in \mathbb{H}_0^1(D)\;:\;\phi|_K\in\mathbb{P}_1(K)\,\;\forall\, K\in\mathcal{T}_h\},
\]
where $\mathbb{P}_1(K)$ denotes the space of affine polynomials on a finite element $K$. 

\subsection{Projection operators and approximation estimates}
Recall the following projections:
\begin{definition}[$\mathbb{L}_x^2$–projection]
	The $\mathbb{L}_x^2$–projection $\Pi_h:\mathbb{L}_x^2\to\mathbb{V}_h$ is defined as follows: for all $v\in\mathbb{L}_x^2$,
    \[
	(\Pi_h v - v,\;\phi_h) \;=\;0
	\qquad\forall\,\phi_h\in\mathbb{V}_h.
	\]
\end{definition}
\begin{definition}[Discrete Laplacian] The discrete Laplacian $\Delta_h:\mathbb{V}_h\to \mathbb{V}_h$ is defined as follows: for all $\xi_h\in\mathbb{V}_h$,
	\begin{align*}
		\left\langle  \Delta_h \xi_h, \varphi_h\right\rangle =-\left\langle \nabla \xi_h, \nabla\varphi_h\right\rangle \qquad\forall\,\varphi_h\in \mathbb{V}_h.
	\end{align*}
\end{definition}

\begin{definition}[Ritz projection]
	The Ritz (or elliptic) projection $\mathcal{R}_h:\mathbb{H}_0^1\to\mathbb{V}_h$ is defined as follows: for all $u\in\mathbb{H}_0^1$,
	\[
	(\nabla(\mathcal{R}_h u - u),\;\nabla\phi_h)\;=\;0
	\qquad\forall\,\phi_h\in\mathbb{V}_h.
	\]
\end{definition}

Both operators satisfy relevant stability and approximation properties.  In particular, for all $v\in \mathbb{H}^2_x$, there exists a constant $C>0$, independent of $h$, such that

\medskip
\noindent\textbf{$\mathbb{L}_x^2$–projection estimates:}
\begin{align}
	\|v - \Pi_h v\|_{\mathbb{L}_x^2} 
	&\;\le\; C\,h^s\,\|v\|_{\mathbb{H}_x^s}\qquad\forall\,v\in\mathbb{H}_x^s,\qquad s=1, 2,
	\label{eq:L2proj-L2} \\
	\|\nabla(v - \Pi_h v)\|_{\mathbb{L}_x^2}
	&\;\le\; C\,h\,\|v\|_{\mathbb{H}_x^2}\qquad\forall\,v\in\mathbb{H}_x^2.
	\label{eq:L2proj-H1}
\end{align}

\noindent\textbf{Ritz‐projection estimates:}
\begin{align}
	\|\nabla(\mathcal{R}_h v - v)\|_{\mathbb{L}_x^2}
	&\;\le\; C\,h\,\|v\|_{\mathbb{H}_x^2}\qquad\forall\, v\in \mathbb{H}_x^2,
	\label{eq:Ritz-H1} \\
	\| \mathcal{R}_h v - v\|_{\mathbb{L}_x^2}
	&\;\le\; C\,h^s\,\|v\|_{\mathbb{H}_x^s}\qquad\forall\,  v\in \mathbb{H}_x^s,\qquad s=1,2.
	\label{eq:Ritz-L2}
\end{align}
Moreover, both $\Pi_h$ and $\mathcal{R}_h$ enjoy the following \emph{stability} bounds:
\begin{align}
	\|\Pi_h v\|_{\mathbb{L}_x^2} \;\le\;\|v\|_{\mathbb{L}_x^2},
	\qquad
	\|\nabla \mathcal{R}_h v\|_{\mathbb{L}_x^2}\;\le\;\|\nabla v\|_{\mathbb{L}_x^2}.
\end{align}
All of the above estimates are followed by the classical interpolation theory on each $K\in\mathcal{T}_h$ together with the summation over the mesh; see, e.g., \cite{Brenner&Scott, Ciarlet}. We define also $\widetilde{X}_h=\mathcal{R}_h\widetilde{X}$.

\subsection{Space-discretization of {\bf SLQ} problem}
The spatial semi-discretization \({\bf SLQ}_h\) of problem \({\bf SLQ}\) \eqref{1.3}-\eqref{1.4} reads as follows: Find an optimal tuple \((X_{1, h}^*, X_{2,h}^*, U_h^*) \in [\mathbb{L}^2_{\mathbb{F}}C_t\mathbb{V}_h]^2\times \mathbb{L}^2_{\mathbb{F}}\mathbb{L}^2_t\mathbb{V}_h \) that minimizes the following functional  
\begin{equation}\label{3.1}
	{J}(X_{1,h}, U_h) = \frac{1}{2} \mathbb{E} \Bigg[ \int_0^T \big[\| X_{1,h}(t)-\widetilde{X}_h(t) \|^2_{\mathbb{L}^2_x} + \alpha\| U_h(t) \|^2_{\mathbb{L}^2_x} \big]\, \,{\rm d}t+ \beta\|X_{1,h}(T)-\widetilde{X}_{h}(T)\|_{\mathbb{L}^2_x}^2 \Bigg]
\end{equation}
subject to the following SDE;  
\begin{equation}
	\begin{cases}\label{3.2}
		\,{\rm d}X_{1,h}=X_{2,h}(t)\,{\rm d}t & \qquad\forall\, t\in (0,T],\\
		\,{\rm d}X_{2,h}(t)= [\Delta_h X_{1,h}(t) + U_h(t)] \, \,{\rm d}t + [\gamma X_{1,h}(t)+\mathcal{R}_h \sigma(t)] \,{\rm d}W(t)& \qquad \forall\, t \in (0,T], \\
		X_{1,h}(0) = \mathcal{R}_h X_{1,0},\\
		X_{2,h}(0)=\mathcal{R}_h X_{2,0}.
		\end{cases}
\end{equation}
Note that, in view of Remark~\ref{vanishing on the boundary}, the space of the semi-discrete control is $\mathbb{L}^2_{\mathbb{F}} \mathbb{L}^2_t \mathbb{V}_h.$

\subsection{Semi-discrete Pontryagin's maximum principle}
We define the reduced cost as follows: for all $U_h\in\mathbb{L}_{\mathbb{F}}^2\mathbb{L}^2_t\mathbb{V}_h$,
\begin{align*}
    \hat{J}_h(U_{h})=J(\mathcal{X}_{1,h}[U_h],U_h),
\end{align*}
where $\mathcal{X}_{1,h}[U_h]$ is the first component of the unique solution $(\mathcal{X}_{1,h}[U_h],\mathcal{X}_{2,h}[U]_h)\equiv(X_{1,h},X_{2,h})$ to the semi-discrete SDE \eqref{3.2} with the semi-discrete distributed control $U_h$. 

Let the adjoint quadruple \(\big((Y_{1,h}, Y_{2,h}), (Z_{1,h}, Z_{2,h})\big) \in \mathbb{L}^2_{\mathbb{F}} C_t(\mathbb{V}_h\times \mathbb{V}_h) \times \mathbb{L}^2_{\mathbb{F}}\mathbb{L}^2_t(\mathbb{V}_h\times\mathbb{V}_h)\) solve the following \({\bf BSPDE}_h\)  
\begin{equation}\label{discrete adjoint ode}
	\begin{cases}
		\,{\rm d}Y_{1,h}(t) = -[\Delta_h Y_{2,h}(t)+ Z_{2,h}(t) + X_{1,h}^*(t)-\widetilde{X}_h] \,{\rm d}t + Z_{1,h}(t)\,{\rm d}W(t) &\qquad \forall\, t \in [0,T], \\
		\,{\rm d}Y_{2,h}(t) = -Y_{1,h}(t)\,{\rm d}t + Z_{2,h}(t)\,{\rm d}W(t)& \qquad \forall\, t \in [0,T],\\
		Y_{1,h}(T) = \beta(X^*_{1,h}(T)-\widetilde{X}_h(T)),\\
		Y_{2,h}(T)=0.
	\end{cases}
\end{equation}
For given $V_h\in\mathbb{L}_{\mathbb{F}}^2\mathbb{L}^2_t\mathbb{V}_h$, let $(\mathcal{X}_{1,h}^0[{V_h}], \mathcal{X}_{2,h}^0[V_h]) \equiv (X^0_{1,h}, X^0_{2,h})\in \mathbb{L}^2_{\mathbb{F}}C_t(\mathbb{V}_h\times\mathbb{V}_h)$ be the unique solution to the following semi-discrete SDE:
\begin{align}\label{1.91}
	\begin{cases} 
		\,{\rm d}X^0_{1,h}(t) = X_{2,h}^0(t) \,{\rm d}t &\qquad\forall\,t\in (0,T], \\
		\,{\rm d} X_{2,h}^0(t) = (\Delta X_{1,h}^0(t) + V_h(t)) \,{\rm d}t + \gamma X_{1,h}^0(t) \,{\rm d}W(t) &\qquad\forall\,t\in(0,T], \\
		X_{1,h}^0(0) = 0, \\
		X_{2,h}^0(0) = 0,
	\end{cases}
\end{align}
which is the space-discretization of SPDE \eqref{1.9}.
Note that  for all $U_h, V_h\in\mathbb{L}^2_{\mathbb{F}}\mathbb{L}^2_{t}\mathbb{V}_h$, for $i=1,2,$
\begin{align}\label{goud}
\mathcal{X}_{i,h}[U_h+V_h]=\mathcal{X}_{i,h}[U_h]+\mathcal{X}_{i,h}^0[V_h],\qquad \mathcal{X}_{i,h}[U_h]-\mathcal{X}_{i,h}[V_h]=\mathcal{X}_{i,h}^0[U_h-V_h].
\end{align}
\begin{proposition} Let $U_h\in \mathbb{L}^2_{\mathbb{F}}\mathbb{L}^2_t\mathbb{V}_h$, then there exists $C>0$ such that for all $U_h\in\mathbb{L}_{\mathbb{F}}^2\mathbb{L}^2_t\mathbb{V}_h$,
\begin{align}
	\mathbb{E}\big[\sup_{s\in[0,T]}&\big[\|\mathcal{X}_{2,h}^0[U_h](t)\|_{\mathbb{L}^2_x}^2+\|\nabla \mathcal{X}_{1,h}^0[U_h](t)\|_{\mathbb{L}^2_x}^2\big]\big]\le C\,\mathbb{E}\big[\|U_{h}\|_{\mathbb{L}^2_{t, x}}^2\big].\label{today0002}
\end{align}
\begin{proof}The proof is a simple consequence of It\^o's formula and Gronwall's inequality. For the proof, one can follow similar arguments as in the proof of \cite[Lemma 3.2]{FengPandaProhl2024}.
\end{proof}
\end{proposition}
In the following theorem, we derive the semi-discrete Pontryagin's maximum principle, which provide optimality condition~\eqref{3.3} and the integral identity~\eqref{today0000}. 
	\begin{thm}[Semi-discrete Pontryagin's maximum principle]\label{Semi-discrete Pontryagin's maximum principle}
	Let Assumption~\ref{CC} hold. There exists the unique optimal control tuple $(X_{1,h}^*, X_{2,h}^*, U_h^*)$ for {\bf SLQ}$_h$ problem \eqref{3.1}-\eqref{3.2}. Let $(Y_{1,h}, Y_{2,h}, Z_{1,h}, Z_{2,h})$ be the unique solution to $\textbf{\textbf{BSDE}}_h$ \eqref{discrete adjoint ode}. Then, the following optimality condition holds:
		\begin{equation}\label{3.3}
			\alpha U_h^*(t) =-Y_{2,h}(t) \quad \forall\,t \in [0,T].
		\end{equation}
	 Moreover, the following integral identity holds: for all $V_h\in \mathbb{L}^2_{\mathbb{F}}\mathbb{L}^2_{t}\mathbb{V}_h$,
		\begin{align}\label{today0000}
			&\left\langle \mathcal{D}_{U}\hat{\mathcal{J}}_h(U_h^*),V_h\right\rangle_{\mathbb{L}^2_{\mathbb{F}}\mathbb{L}^2_{t, x}} \notag\\&= \mathbb{E} \bigg[\int_0^T\big[\left\langle X_{1,h}^*(t)-\widetilde{X}_h(t), \mathcal{X}_{1,h}^0[V_h](t)\right\rangle + \alpha\left\langle U_h^*(t), V_h(t)\right\rangle \big]\,{\rm d}t +\beta\left\langle X^*_{1,h}(T),\mathcal{X}_{1,h}^0[V_h](T)\right\rangle\bigg]=0.
		\end{align} 
	\end{thm}

\begin{proof}
For the existence and uniqueness of the optimal control tuple $(X_{1,h}^*, X_{2,h}^*, U_h^*)$, one can follow similar arguments as in the proof of \cite[Theorem 1.43]{Hinze et al. book}; see  also \cite{LuZhang2021}. For the proof of optimality condition~\eqref{3.3} and equation \eqref{today0000}, one can follow similar lines as in the proof of Theorem \ref{Pontryagin's Maximum Principle}.
\end{proof}
\begin{remark}
    The optimality condition~\eqref{3.3} enhances time regularity of the semi-discrete optimal control $U_h^*$ (see Proposition~\ref{Time-regularity for semi-discrete optimal control} in the Appendix), while the integral identity~\eqref{today0000} which plays a pivotal role in the error analysis for the space--time discretization (see Theorems~\ref{Thm 3.8} and \ref{strong rate of convergence for time discretization}).
\end{remark}
\subsection{Convergence with rates for ${\bf SLQ}_h$ problem}
In this subsection, we establish a strong convergence results for the semi-discrete problem \(\mathbf{SLQ}_h\) towards the continuous \(\mathbf{SLQ}\) problem. 
We now state the following proposition, which provides the error estimate between the analytic state $\mathcal{X}_{1}[\Pi_h U^*]$ and the semi-discrete state $\mathcal{X}_{1,h}[\Pi_h U^*]$ corresponding to the same semi-discrete control $\Pi_h U^*$. This result will be useful in the proof of Theorem~\ref{Thm 3.8}. 

\begin{proposition}\label{Proposition 3.7} Let Assumption~\ref{BB} hold. Let $(\mathcal{X}_{1,h}[\Pi_h U^*],\mathcal{X}_{2,h}[\Pi_h U^*])$ and $(X_{1}[\Pi_h U^*],X_{2}[\Pi_h U^*])$ be the unique solutions to \eqref{3.1} and \eqref{1.4} with distributed semi-discrete control $\Pi_h U^*$, respectively. Then there exists $C>0$ such that for all $t\in [0,T]$,
	\begin{align}\label{today44}
		&\mathbb{E}\big[\|\nabla \mathcal{X}_1[\Pi_h U^*](t)-\nabla \mathcal{X}_{1,h}[\Pi_h U^*](t)\|_{\mathbb{L}^2_{x}}^2\big]+\mathbb{E}\big[\|X_2(t)-X_{2,h}(t)\|_{\mathbb{L}^2_x}^2\big]\notag\\&\qquad\le C\,h^2\big(\|X_{1,0}\|_{\mathbb{H}_x^3}^2+ \|X_{2,0}\|_{\mathbb{H}_x^2}^2+\|\widetilde{X}\|_{C_t\mathbb{H}_0^1}^2+ \mathbb{E}\big[\|\sigma\|_{\mathbb{L}^2_t\mathbb{H}_x^2}^2]\big).
	\end{align}
\end{proposition}
\begin{proof}
For convenience, we set $$(X_1, X_2)=(\mathcal{X}_1[\Pi_h U^*], \mathcal{X}_2[\Pi_h U^*])\qquad\text{and}\qquad (X_{1,h}, X_{2,h})=(\mathcal{X}_{1,h}[\Pi_h U^*], \mathcal{X}_{2,h}[\Pi_h U^*]).$$ 	Now, from \eqref{1.4}, $(X_1, X_2)$ satisfies the following projected SDE with given control $\Pi_h U^*$
	\begin{equation}
		\begin{cases}\label{3.21}
			\,{\rm d}\mathcal{R}_h X_{1}=\mathcal{R}_h X_{2}(t)\,{\rm d}t & \qquad\forall\, t\in (0,T],\\
			\,{\rm d}\Pi_h X_{2}(t)= [\Delta_h\mathcal{R}_h X_{1}(t) +\Pi_h U^*(t)] \, \,{\rm d}t + [\gamma\Pi_h X_{1}(t)+\Pi_h \sigma(t)] \,{\rm d}W(t)& \qquad \forall\, t \in (0,T], \\
			\mathcal{R}_h X_{1}(0) = \mathcal{R}_h X_{1,0},\\
			\Pi_h X_{2}(0)=\Pi_h X_{2,0},
		\end{cases}
	\end{equation}
    where the fact $\Pi_h\Delta X_{1}= \Delta_h\mathcal{R}_h X_{1}$ is used.
Further from \eqref{3.2} and \eqref{3.21}, we obtain that 
\begin{equation}
	\begin{cases}\label{3.22}
		\,{\rm d}(X_{1,h}(t)-\mathcal{R}_h X_{1}(t))=(X_{2,h}(t)-\mathcal{R}_h X_{2}(t))\,{\rm d}t & \qquad\forall\, t\in (0,T],\\
		\,{\rm d}(X_{2,h}(t)-\Pi_h X_{2}(t))= [\Delta_h(X_{1,h}(t)-\mathcal{R}_h X_{1}(t))] \,{\rm d}t\\\qquad\qquad\qquad + [\Pi_h (X_{1,h}(t)-\Pi_h X_{1}(t))+(\mathcal{R}_h\sigma(t)-\Pi_h \sigma(t))] \,{\rm d}W(t)& \qquad \forall\, t \in (0,T], \\
		X_{1,h}(0)-\mathcal{R}_hX_1(0) = 0,\\
		X_{2,h}(0)-\Pi_h X_{2}(0)=(\mathcal{R}_h-\Pi_h) X_{2,0}.
	\end{cases}
\end{equation}
We apply It\^o's formula to $(X_1, X_{1,h})\to \|\nabla\big(\mathcal{R}_h X_1-X_{1,h})\|_{\mathbb{L}^2_x}^2$ and $(X_2, X_{2,h})\to \|\Pi_h X_2-X_{2,h}\|_{\mathbb{L}^2_x}^2$ to get $\mathbb{P}$-almost surely, for all $t\in[0,T]$,
	\begin{align}\label{1212}
		\|\nabla\big(X_{1,h}(t)-\mathcal{R}_h X_{1}(t)\big)\|_{\mathbb{L}^2_x}^2=	& 2\int_{0}^{t}\left\langle \nabla(X_{2,h}(t)-\mathcal{R}_h X_{2}(t)), \nabla(X_{1,h}(t)-\mathcal{R}_hX_{1}(t))\right\rangle\,{\rm d}t,
	\end{align}
and
\begin{align}\label{1313}
	\|X_{2,h}(t)-\Pi_h X_2(t)\|_{\mathbb{L}^2_x}^2&=\|X_{2,h}(0)-\Pi_h X_2(0)\|_{\mathbb{L}^2_x}^2\notag\\&-2\int_0^t\left\langle \nabla(X_{1,h}(t)-\mathcal{R}_hX_1(t)),\nabla(X_{2,h}(t)-\Pi_hX_2(t))\right\rangle\,{\rm d}t\notag\\&+2\int_0^T\left\langle \gamma(X_{1,h}(t)-\Pi_h X_1(t))+(\mathcal{R}_h\sigma(t)-\Pi_h\sigma(t)),(X_{2,h}(t)-\Pi_hX_2(t))\right\rangle\,{\rm d}W(t)\notag\\&+\int_{0}^t\|\gamma(X_{1,h}(t)-\Pi_h X_1(t))+(\mathcal{R}_h\sigma(t)-\Pi_h\sigma(t))\|_{\mathbb{L}^2_x}^2\,{\rm d}t.
\end{align}
By adding \eqref{1212}-\eqref{1313} and taking expectation, we obtain for all $t\in[0,T],$
\begin{align*}
&\mathbb{E}\big[\|\nabla\big(X_{1,h}(t)-\mathcal{R}_h X_{1}(t)\big)\|_{\mathbb{L}^2_x}^2+\|X_{2,h}(t)-\Pi_h X_2(t)\|_{\mathbb{L}^2_x}^2\big]=\mathbb{E}\bigg[\|\big(X_{2,h}(0)-\Pi_h X_{2}(0)\big)\|_{\mathbb{L}^2_x}^2\notag\\&\qquad+ 2\int_{0}^{t}\left\langle \nabla(\mathcal{R}_h X_{2}(t)-\Pi_h X_{2}(t)), \nabla(X_{1,h}(t)-\mathcal{R}_hX_{1}(t))\right\rangle\,{\rm d}t+\|X_{2,h}(0)-\Pi_h X_2(0)\|_{\mathbb{L}^2_x}^2\notag\\&\qquad+\int_{0}^t\|\gamma(X_{1,h}(t)-\Pi_h X_1(t))+(\mathcal{R}_h\sigma(t)-\Pi_h\sigma(t))\|_{\mathbb{L}^2_x}^2\,{\rm d}t\bigg].
\end{align*}
It implies that
	\begin{align*}
		\mathbb{E}\big[\|\nabla\mathcal{R}_h X_1(t)-\nabla X_{1,h}(t)\|_{\mathbb{L}^2_x}^2\big] &+\mathbb{E}\big[\|\Pi_h X_2(t)-X_{2,h}(t)\|_{\mathbb{L}^2_x}^2\big]\le\mathbb{E}\big[\|\Pi_h X_{2,0}-\mathcal{R}_h X_{2,0}\|_{\mathbb{L}^2_x}^2\big]\\& +\mathbb{E}\bigg[\int_0^t\|\nabla (X_{1,h}(t)-\mathcal{R}_h X_1(t))\|_{\mathbb{L}^2_x}^2\,{\rm d}t\bigg]\\&+\mathbb{E}\bigg[\int_{0}^t\|\nabla(\mathcal{R}_h X_2(t)-\Pi_h X_{2}(t))\|\,{\rm d}t\bigg]\\&+ C\int_0^t\bigg(\mathbb{E}\big[\|\Pi_h X_1(t)-X_{1,h}(t)\|_{\mathbb{L}^2_x}^2\big] +\mathbb{E}\big[\|\Pi_h \sigma(t)-\mathcal{R}_h\sigma(t)\|_{\mathbb{L}^2_x}^2\big]\bigg)\,{\rm d}t.
	\end{align*} 
	By using estimates \eqref{eq:L2proj-L2}-\eqref{eq:Ritz-L2}, we have for all $t\in[0,T]$,
	\begin{align*}
		&\mathbb{E}\big[\|\nabla\Pi_h X_1(t)-\nabla X_{1,h}(t)\|_{\mathbb{L}^2_x}^2\big]+\mathbb{E}\big[\|\Pi_h X_2(t)-X_{2,h}(t)\|_{\mathbb{L}^2_x}^2\big]\\&\le C h^4\|X_{2,0}\|_{\mathbb{H}^2}^2+ C h^4 \|\sigma\|_{\mathbb{L}^2_t\mathbb{H}^2}+ C \mathbb{E}\bigg[\int_0^t\|\nabla(\Pi_h X_{2}(t)-\mathcal{R}_h X_2(t))\|_{\mathbb{L}^2_x}^2\,{\rm d}t\bigg]\\&\qquad\qquad\qquad\qquad + \mathbb{E}\bigg[\int_0^t\|\Pi_h X_1(t)-\mathcal{R}_h X_1(t)\|_{\mathbb{L}^2_x}^2\,{\rm d}t\bigg]\\
		&\le\, C h^4\|X_{2,0}\|_{\mathbb{H}^2}^2+ C h^4 \|\sigma\|_{\mathbb{L}^2_t\mathbb{H}^2}+C h^2\mathbb{E}\big[\|\mathcal{X}_2[\Pi_h U^*]\|_{\mathbb{L}^2_t\mathbb{H}^2_x}^2\big] +C\,h^4\mathbb{E}\big[\|\mathcal{X}_1[\Pi_h U^*]\|_{\mathbb{L}^2_t\mathbb{H}_x^2}^2\big]\\&\le Ch^2\big(\|X_{1,0}\|_{\mathbb{H}_x^3}^2+\|X_{2,0}\|_{\mathbb{H}_x^2}^2+ \|\widetilde{X}\|_{C_t\mathbb{H}_0^1}^2+\mathbb{E}\big[\|\sigma\|_{\mathbb{L}^2_t\mathbb{H}_x^2}^2\big]\big),
        \end{align*}
where in the last inequality \eqref{space-regularity for otimal state 21} and \eqref{space-regularity for otimal state 2} are used. With the help of estimates \eqref{eq:L2proj-L2}-\eqref{eq:Ritz-H1}, it implies that for all $t\in[0,T]$,
	\begin{align*}
		&\mathbb{E}\big[\|\nabla \mathcal{X}_1[\Pi_h U^*](t)-\nabla \mathcal{X}_{1,h}[\Pi_h U^*](t)\|_{\mathbb{L}^2_{x}}^2\big]+\mathbb{E}\big[\|X_2(t)-X_{2,h}(t)\|_{\mathbb{L}^2_x}^2\big]\notag\\&\le\,C\,h^2\,\big(\|X_{1,0}\|_{\mathbb{H}_x^3}^2+\|X_{2,0}\|_{\mathbb{H}_x^2}^2 + \|\widetilde{X}\|_{C_t\mathbb{H}_0^1}^2+\mathbb{E}\big[\|\sigma\|_{\mathbb{L}^2_t\mathbb{H}_x^2}^2\big]\big).
	\end{align*}
    This completes the proof.
\end{proof}
In the following, we establish a rate of convergence for the semi-discrete optimal control tuple $(X_{1,h}^*, X_{2,h}^*, U_h^*)$ of the {\bf SLQ}$_h$ problem \eqref{3.1}-\eqref{3.2} towards the unique optimal control tuple $(X_1^*, X_{2}^*, U^*)$ of the continuous {\bf SLQ} problem \eqref{1.3}-\eqref{1.4}. The proof relies on the identities \eqref{today0001} and \eqref{today0000}, along with the stability estimates \eqref{today0002}, \eqref{space regularity H-2 bound for adjoint Y}, and \eqref{bound for optimal control}.
\begin{thm}\label{Thm 3.8}\textit{Let Assumption~\ref{BB} hold. Let $(X^*_1, X_2^*, U^*)$ and $(X_{1,h}^*, X_2^*, U_h^*)$ solve problems {\bf SLQ} \eqref{1.3}-\eqref{1.4} and {\bf SLQ}$_h$ \eqref{3.1}-\eqref{3.2}, respectively. Then there exists a constant $C > 0$ such that}
\begin{equation} 
	\mathbb{E}[\|U^* - U_h^*\|_{\mathbb{L}^2_{t, x}}^2] + \mathbb{E}[\|X_{1}^* - X_{1,h}^*\|_{\mathbb{L}^2_{t, x}}^2] \leq C h^2\,\big(\|X_{1,0}\|_{\mathbb{H}_x^3}^2+\|X_{2,0}\|_{\mathbb{H}_x^2}^2 + \|\widetilde{X}\|_{C_t\mathbb{H}_0^1}^2+\mathbb{E}\big[\|\sigma\|_{\mathbb{L}^2_t\mathbb{H}_x^2}^2\big]\big).
\end{equation}
\end{thm}
\begin{proof} First we observe that
\begin{align*}
	&\alpha \mathbb{E}[\|U^* - U_h^*\|_{\mathbb{L}^2_{t, x}}^2]= \mathbb{E}\bigg[ \int_0^T \left\langle \alpha U^*(t), U^*(t) - U_h^*(t)\right\rangle  dt - \int_0^T \left\langle  \alpha U_h^*(t), \Pi_h U^*(t) - U_h^*(t) \right\rangle  dt \bigg] \\
	&= \mathbb{E} \bigg[ \int_0^T \left\langle  X_1^*(t)-\widetilde{X}(t), \mathcal{X}_1^0[U_h^*](t) - \mathcal{X}_1^0[U^*](t) \right\rangle  \,{\rm d}t+\beta\left\langle  X_1^*(T)-\widetilde{X}(T), \mathcal{X}_1^0[U_h^*](T) - \mathcal{X}_1^0[U^*](T) \right\rangle  \\&\qquad + \int_0^T \left\langle X_{1,h}^*(t)-\widetilde{X}_{1,h}(t), \mathcal{X}_{1,h}^0[\Pi_h U^* - U_h^*](t) \right\rangle  \,{\rm d}t +\beta\left\langle  X_{1,h}^*(T)-\widetilde{X}_{1,h}(T), \mathcal{X}_{1,h}^0[\Pi_h U^* - U_h^*](T) \right\rangle  \bigg],
\end{align*}
where in the last equality we used integral identities \eqref{today0001} and \eqref{today0000}. From the equality above we further derive (by inserting some intermediate terms)
\begin{align*}
	\alpha \mathbb{E}[\|U^* - U_h^*\|_{\mathbb{L}^2_{t, x}}^2] = &- \mathbb{E} \bigg[ \int_0^T \big\langle X_1^*(t) - \mathcal{X}_{1,h}[U_h^*](t), X_{1}^*(t) - \mathcal{X}_{1,h}[U_h^*](t) \big\rangle \,{\rm d}t \\
	&-\mathbb{E} \big[ \big\langle X_1^*(T) - \mathcal{X}_{1,h}[U_h^*](T), X_1^*(t) - \mathcal{X}_{1,h}[U_h^*](T) \big\rangle\big] + \sum_{i=1}^6I_i,
\end{align*}
which in turn gives
\begin{align}\label{today41}
	\alpha \mathbb{E}[\|U^* - U_h^*\|_{\mathbb{L}^2_{t, x}}^2] +  \mathbb{E}[\|X_1^* - X_{1,h}^*\|_{\mathbb{L}^2_t\mathbb{L}^2_x}^2]+\beta\mathbb{E}[\|X_1^*(T) - X_{1,h}^*(T)\|_{\mathbb{L}^2_x}^2]=\sum_{i=1}^6 I_i,
\end{align}
where
\begin{align*}
	I_1&= -\mathbb{E} \bigg[ \int_0^T \big\langle X_1^*(t)-\widetilde{X}(t), \mathcal{X}_1[U^* - U_h^*](t) - \mathcal{X}_{1,h}[\Pi_h U^* - U_h^*](t) \big\rangle \,{\rm d}t \bigg],\\
	I_2&=\beta\mathbb{E}\big[\big\langle X_1^*(T)-\widetilde{X}(T), \mathcal{X}_1[U^* - U_h^*](T) - \mathcal{X}_{1,h}[\Pi_h U^* - U_h^*](T) \big\rangle\big],\\
	I_3&=\mathbb{E}\bigg[\int_0^T \big\langle X_1^*(t) - \mathcal{X}_{1,h}[U_h^*](t), X_1^*(t) - \mathcal{X}_{1,h}[\Pi_h U^*](t) \big\rangle \,{\rm d}t\bigg],\\
	I_4&=\beta\mathbb{E}\big[\big\langle X_1^*(T) - \mathcal{X}_{1,h}[U_h^*](T), X_1^*(T) - \mathcal{X}_{1,h}[\Pi_h U^*](T) \big\rangle\big],\\
	I_5&=\mathbb{E}\bigg[\int_0^T\left\langle \widetilde{X}(t)-\widetilde{X}_{1,h}(t), \mathcal{X}_{1,h}^0[U_h^*-\Pi_h U^*](t)\right\rangle\,{\rm d}t\bigg],\\
	I_6&=\beta\mathbb{E}\big[\left\langle \widetilde{X}(T)-\widetilde{X}_{1,h}(T), \mathcal{X}_{1,h}^0[U_h^*-\Pi_h U^*](T)\right\rangle\big],
\end{align*}
here we used the facts (see equation~\eqref{Goud} and \eqref{goud})
\begin{align}\label{mkt}
\mathcal{X}_1[U^*]-\mathcal{X}_1[U_h^*]=\mathcal{X}_1^0[U^*-U_h^*] \qquad\text{and}\qquad \mathcal{X}_{1,h}[U_h^*]-\mathcal{X}_{1,h}[\Pi_h U^*]=\mathcal{X}_{1,h}^0[U_h^*-\Pi_h U^*].
\end{align}

\noindent
	\textbf{Step 1.} In this step, we split the term $I$ as follows:
\begin{align*}
	I_1+I_2&=:I_{11} + I_{21},			
\end{align*}
where
\begin{align*}
	I_{11}&=\mathbb{E}\bigg[\int_0^T \big\langle X_1^*(t)-\widetilde{X}(t), \mathcal{X}_1^0[U^*-U_h^*](t) \big\rangle \,{\rm d}t +\beta \big\langle X_1^*(T)-\widetilde{X}(T), \mathcal{X}_1^0[U^*-U_h^*](T) \big\rangle\bigg],\\
	I_{21}&= \mathbb{E}\bigg[\int_0^T\big\langle X_1^*(t)-\widetilde{X}(t),  \mathcal{X}_{1,h}^0[U_h^*-\Pi_h U^*](t) \big\rangle\,{\rm d}t +\beta\big\langle X_1^*(T)-\widetilde{X}(T),  \mathcal{X}_{1,h}^0[U_h^*-\Pi_h U^*](T) \big\rangle\bigg].\\
	\end{align*} 
\noindent
\textbf{Step 1(a).} As done in the proof of Pontryagin's maximum principle ({\em i.e.,} Theorem \ref{Pontryagin's Maximum Principle}) and from identity \eqref{today000}, we have
\begin{align*}
	I_{11}=\mathbb{E}\bigg[\int_0^T\left\langle  Y_2(t), U^*(t)-U_h^*(t)\right\rangle\,{\rm d}t\bigg].
\end{align*}
\textbf{Step 1(b).} For term $I_{21}$, we can follow similar lines as in the proof of Pontryagin's maximum principle ({\em i.e.,} Theorem~\ref{Pontryagin's Maximum Principle}) to conclude that
\begin{align*}
	I_{21}=\mathbb{E}\bigg[\int_0^T\left\langle  \nabla (\Pi_h Y_2- \mathcal{R}_hY_2), \mathcal{X}_{1,h}[U_h^*-\Pi_h U^*](t)\right\rangle\,{\rm d}t\bigg]+ \mathbb{E}\bigg[\int_0^T\left\langle  \Pi_h Y_2(t), U_h^*(t)-\Pi_h U^*(t) \right\rangle\,{\rm d}t\bigg].
\end{align*}
\textbf{Step 1(c):} From the last two substeps, we conclude that 
\begin{align*}
	I_1+I_2= I_{31}+ I_{41} + I_{51},
\end{align*}
where
\begin{align*}
	I_{31}&=\mathbb{E}\bigg[\int_0^T\left\langle \nabla (\Pi_h Y_2- \mathcal{R}_hY_2), \nabla \mathcal{X}_{1,h}^0[U_h^*-\Pi_h U^*](t)\right\rangle\,{\rm d}t\bigg],\\
	I_{41}&=\mathbb{E}\bigg[\int_0^T\left\langle \Pi_h Y_2(t)-Y_2, U_h^*(t)-\Pi_h U^*(t)\right\rangle\,{\rm d}t\bigg],\\
	I_{51}&=\mathbb{E}\bigg[\int_0^T\left\langle Y_2(t), U^*(t)-\Pi_hU^*(t)\right\rangle\,{\rm d}t\bigg].
\end{align*}
\textbf{Step 1(d)}:
For term $I_{31}$, we have for any $\delta>0$
\begin{align*}
	|I_{31}|\le C_\delta\mathbb{E}\big[\|\Pi_h Y_2-\mathcal{R}_h Y_2\|_{\mathbb{L}^2_{t}\mathbb{H}_0^1}^2\big]+ \delta\mathbb{E}\big[\|\nabla \mathcal{X}_{1,h}[U_h^*-\Pi_h U^*]\|_{\mathbb{L}^2_{t, x}}^2\big].
\end{align*}
By using stability estimate \eqref{today0002}, we obtain
\begin{align*}
 \mathbb{E}\big[\|\nabla \mathcal{X}_{1,h}^0[U_h^*-\Pi_h U^*]\|_{\mathbb{L}^2_{t, x}}^2\big]\le C\,\mathbb{E}\big[[U_h^*-\Pi_h U^*\|_{\mathbb{L}^2_{t, x}}^2\big] \le C\,\big(\mathbb{E}\big[[U^*-\Pi_h U^*\|_{\mathbb{L}^2_{t, x}}^2\big]+ \mathbb{E}\big[[U_h^*- U^*\|_{\mathbb{L}^2_{t, x}}^2\big]\big).
\end{align*}
By using \eqref{eq:L2proj-H1} and \eqref{eq:Ritz-H1}, we conclude that
\begin{align*}
 \mathbb{E}\big[\|\Pi_h Y_2-\mathcal{R}_h Y_2\|_{\mathbb{L}^2_{t}\mathbb{H}_0^1}^2\big]\le C\,h^2\mathbb{E}[\|Y_2\|_{\mathbb{L}^2_t\mathbb{H}_x^2}^2].
\end{align*}
With the help of \eqref{space regularity H-2 bound for adjoint Y} and 
\eqref{bound for optimal control}, we obtain 
\begin{align*}
\mathbb{E}\big[\|\Pi_h Y_2-\mathcal{R}_h Y_2\|_{\mathbb{L}^2_{t}\mathbb{H}_0^1}^2\big]\le C\,h^2\big(\|X_{1,0}\|_{\mathbb{H}_0^1}^2+ \|X_{2,0}\|_{\mathbb{L}^2_x}^2+\|\widetilde{X}\|_{C_t\mathbb{H}_0^1}^2+\mathbb{E}\big[\|\sigma\|_{\mathbb{L}^2_{t,x}}^2\big]\big).
\end{align*}
and 
by choosing small enough $\delta>0$, we yield
\begin{align*}
	|I_{31}|\le Ch^2\big(\|X_{1,0}\|_{\mathbb{H}_0^1}^2+ \|X_{2,0}\|_{\mathbb{L}^2_x}^2+ +\|\widetilde{X}\|_{C_t\mathbb{H}_0^1}^2+\mathbb{E}\big[\|\sigma\|_{\mathbb{L}^2_{t,x}}^2\big]\big)+ \frac{\alpha}{8}\mathbb{E}\big[[U_h^*- U^*\|_{\mathbb{L}^2_{t, x}}^2\big].
\end{align*}
\textbf{Step 1(f)} Similarly as in previous the substep, we get
\begin{align*} 
	|I_{41}|\le Ch^2\big(\|X_{1,0}\|_{\mathbb{H}_0^1}^2+ \|X_{2,0}\|_{\mathbb{L}^2_x}^2 +\|\widetilde{X}\|_{C_t\mathbb{H}_0^1}^2+\mathbb{E}\big[\|\sigma\|_{\mathbb{L}^2_{t,x}}^2\big]\big) + \frac{\alpha}{8}\mathbb{E}\big[\|U_h^*- U^*\|_{\mathbb{L}^2_{t, x}}^2\big].
\end{align*}
\textbf{Step 1(h)} By using the orthogonality of the projection $\Pi_h$, we have
\begin{align*}
	I_{51}=\mathbb{E}\bigg[\int_0^T\left\langle Y_2(t)-\Pi_h Y_2(t), U^*(t)-\Pi_hU^*(t)\right\rangle\,{\rm d}t\bigg].
\end{align*}
By using \eqref{eq:L2proj-L2},\eqref{1.6}, \eqref{space regularity H-2 bound for adjoint Y} and \eqref{space-regularity for otimal state 21}, it implies that
\begin{align*}
	|I_{51}|&\le\,C\mathbb{E}\big[\|Y_2-\Pi_h Y_{2}\|_{\mathbb{L}^2_{t, x}}^2\big]+\big[\|U^*- \Pi_h U^*\|_{\mathbb{L}^2_{t, x}}^2\big] \\&\le Ch^4\mathbb{E}\big[\|Y_2\|_{\mathbb{L}^2_{t}\mathbb{H}_x^2}^2\big] + Ch^4\mathbb{E}\big[\|U^*\|_{\mathbb{L}^2_{t}\mathbb{H}^2_x}^2\big]\\&\le Ch^4\mathbb{E}\big[\|Y_2\|_{\mathbb{L}^2_{t}\mathbb{H}_x^2}^2\big]\\
	&\le C\,h^4\big(\|X_{1,0}\|_{\mathbb{H}_0^1}^2+ \|X_{2,0}\|_{\mathbb{L}^2_x}^2+\|\widetilde{X}\|_{C_t\mathbb{H}_0^1}^2+\mathbb{E}\big[\|\sigma\|_{\mathbb{L}^2_{t,x}}^2\big]\big).
\end{align*}
\noindent
\textbf{Step 1(i):} From previous sub-steps, we conclude that
\begin{align}\label{today46}
	|I_1+I_2|\le Ch^2\big(\|X_{1,0}\|_{\mathbb{H}_0^1}^2+ \|X_{2,0}\|_{\mathbb{L}^2_x}^2+ +\|\widetilde{X}\|_{C_t\mathbb{H}_0^1}^2+\mathbb{E}\big[\|\sigma\|_{\mathbb{L}^2_{t,x}}^2\big]\big) + \frac{\alpha}{4}\mathbb{E}\big[[U_h^*- U^*\|_{\mathbb{L}^2_{t, x}}^2\big].
\end{align}
\textbf{Step 2:} In this step, we estimate the term $I_3$. We obtain that 
\begin{align}\label{today43}
	\mathbb{E}[\|X_1^* - X_h[\Pi_h U^*]\|_{\mathbb{L}_{t, x}}^2]\le  \mathbb{E}[\|X_1^* - \mathcal{X}_1[\Pi_h U^*]\|_{\mathbb{L}^2_{t, x}}^2]+ \mathbb{E}[\|\mathcal{X}_1[\Pi_h U^*] - \mathcal{X}_{1,h}[\Pi_h U^*]\|_{\mathbb{L}^2_{t, x}}^2].
\end{align}
By using the identity \eqref{mkt} and the estimate \eqref{today0002}, we obtain
\begin{align*}
	\mathbb{E}[\|X_1^* - \mathcal{X}_{1,h}[\Pi_h U^*]\|_{\mathbb{L}_{t, x}}^2]&\le 	\mathbb{E}[\|\mathcal{X}_1[U^*] - \mathcal{X}_1[\Pi_h U^*]\|_{\mathbb{L}_{t, x}}^2]+\mathbb{E}[\|\mathcal{X}_1[\Pi_h U^*] - \mathcal{X}_{1,h}[\Pi_h U^*]\|_{\mathbb{L}_{t, x}}^2]\\&\le 	\mathbb{E}[\|\mathcal{X}_1^0[U^* -\Pi_h U^*]\|_{\mathbb{L}_{t, x}}^2]+\mathbb{E}[\|\mathcal{X}_1[\Pi_h U^*] - \mathcal{X}_{1,h}[\Pi_h U^*]\|_{\mathbb{L}_{t, x}}^2]\\&\le C\mathbb{E}\big[\|U^*-\Pi_h U^*\|_{\mathbb{L}^2_{t, x}}^2\big] + \mathbb{E}[\|\mathcal{X}_1[\Pi_h U^*] - \mathcal{X}_{1,h}[\Pi_h U^*]\|_{\mathbb{L}_{t, x}}^2].
    \end{align*}
By using \eqref{eq:L2proj-L2}, \eqref{today44},  and \eqref{space-regularity for optimal control 1}, we get
\begin{align*}
\mathbb{E}[\|X_1^* - \mathcal{X}_{1,h}[\Pi_h U^*]\|_{\mathbb{L}_{t, x}}^2]	&\le C h^2\mathbb{E}\big[\|U^*\|_{C_t\mathbb{H}_0^1}^2\big]+C h^2\big(\|X_{1,0}\|_{\mathbb{H}_x^3}^2+\|X_{2,0}\|_{\mathbb{H}_x^2}^2 + \|\widetilde{X}\|_{C_t\mathbb{H}_0^1}^2+\mathbb{E}\big[\|\sigma\|_{\mathbb{L}^2_t\mathbb{H}_x^2}^2\big]\big)
    \\
	&\le Ch^2\big(\|X_{1,0}\|_{\mathbb{H}_x^3}^2+\|X_{2,0}\|_{\mathbb{H}_x^2}^2 + \|\widetilde{X}\|_{C_t\mathbb{H}_0^1}^2+\mathbb{E}\big[\|\sigma\|_{\mathbb{L}^2_t\mathbb{H}_x^2}^2\big]\big).
\end{align*}
It implies that
\begin{align}\label{today45}
	|I_3|&\le \frac{1}{2}\mathbb{E}\big[\|X_1^*-X_{1,h}^*\|_{\mathbb{L}^2_{t, x}}^2\big] + C\mathbb{E}\big[\|X_1^* - \mathcal{X}_{1,h}[\Pi_h U^*]\|_{\mathbb{L}_{t, x}}^2\big]\notag\\
	&\le \frac{1}{2}\mathbb{E}\big[\|X_1^*-X_{1,h}^*\|_{\mathbb{L}^2_{t, x}}^2\big]+ C\,h^2\big(\|X_{1,0}\|_{\mathbb{H}_x^2}^2+ \|X_{2,0}\|_{\mathbb{H}_0^1}^2+ \|\widetilde{X}\|_{C_t\mathbb{H}_0^1}^2+ \mathbb{E}\big[\|\sigma\|_{\mathbb{L}^2_t\mathbb{H}_0^1}^2\big]\big).
\end{align}
\textbf{Step 3:}
We can follow similar lines of Step 1 and Step 2 to conclude that
\begin{align}\label{yksk}
	|I_4|\le C\,h^2\big(\|X_{1,0}\|_{\mathbb{H}_x^2}^2+ \|X_{2,0}\|_{\mathbb{H}_0^1}^2+ \mathbb{E}\big[\|\sigma\|_{\mathbb{L}^2_t\mathbb{H}_0^1}^2\big]\big) + \frac{\beta}{2}\mathbb{E}\big[\|X^*(T)-X_{1,h}^*(T)\|_{\mathbb{L}^2_x}^2\big].
\end{align}
\textbf{Step 4:} By using Young's inequality, \eqref{eq:L2proj-L2}, \eqref{today0002} and \eqref{space-regularity for optimal control 1}, we get $(\delta>0)$
\begin{align*}
	|I_5|&\le C_\delta \mathbb{E}\big[\|\widetilde{X}-\widetilde{X}_{1,h}\|_{\mathbb{L}^2_{t, x}}^2\big]+\delta\mathbb{E}\big[\|\mathcal{X}_{1,h}^0[U_h^*-\Pi_h U^*]\|_{\mathbb{L}^2_{t, x}}^2\big]\\
	&\le C_\delta h^2 \|\widetilde{X}\|_{C_t\mathbb{H}_0^1}^2+ \delta C\mathbb{E}\big[\|U_h^*-\Pi_h U^*\|_{\mathbb{L}^2_{t, x}}^2\big]\\
	&\le C_\delta h^2\|\widetilde{X}\|_{C_t\mathbb{H}_0^1}^2 + C\delta\mathbb{E}\big[\|U_h^*-U^*\|_{\mathbb{L}^2_{t, x}}^2\big]+ C\,\delta\mathbb{E}\big[\|U^*-\Pi_h U^*\|_{\mathbb{L}^2_{t, x}}^2\big]\\
	&\le C_\delta h^2 \big(\|X_{1,0}\|_{\mathbb{H}_0^1}^2+ \|X_{2,0}\|_{\mathbb{L}_x^2}^2+ \|\widetilde{X}\|_{C_t\mathbb{H}_0^1}^2+\mathbb{E}\big[\|\sigma\|_{\mathbb{L}^2_{t,x}}^2\big]\big)+ C\delta  \mathbb{E}\big[\|U_h^*-U^*\|_{\mathbb{L}^2_{t, x}}^2\big].
\end{align*}
\textbf{Step 5:} Similarly to Step $4$, we conclude that
\begin{align}\label{today46}
	I_6\le C_\delta h^2\big(\|X_{1,0}\|_{\mathbb{H}_0^1}^2+ \|X_{2,0}\|_{\mathbb{L}_x^2}^2+ \|\widetilde{X}\|_{C_t\mathbb{H}_0^1}^2+\mathbb{E}\big[\|\sigma\|_{\mathbb{L}^2_{t,x}}^2\big]\big) + C\delta  \mathbb{E}\big[\|U_h^*-U^*\|_{\mathbb{L}^2_{t, x}}^2\big].
\end{align}
\textbf{Step 6:} In this final step, from \eqref{today41}--\eqref{today46} and by choosing small $\delta>0$, we obtain
\begin{align*}
		&\mathbb{E}[\|U^* - U_h^*\|_{\mathbb{L}^2_{t, x}}^2] + \mathbb{E}[\|X_1^* - X_{1,h}^*\|_{\mathbb{L}^2_{t, x}}^2]+\beta\mathbb{E}[\|X_1^*(T) - X_h^*(T)\|_{\mathbb{L}^2_x}^2]\\&\qquad \leq C h^2\big(\|X_{1,0}\|_{\mathbb{H}_x^3}^2+ \|X_{2,0}\|_{\mathbb{H}_x^2}^2+\|\widetilde{X}\|_{C_t\mathbb{H}_0^1}^2+ \mathbb{E}\big[\|\sigma\|_{\mathbb{L}^2_t\mathbb{H}_x^2}^2\big]\big).
\end{align*}
This completes the proof.
\end{proof}
The following theorem presents the main result of this section, establishing the rate of convergence in the energy norm.
\begin{thm}[Final result of this section]\label{thm3.5}\textit{Let Assumption~\ref{BB} hold. Let $(X^*_1, X^*_2, U^*)$ and $(X_{1,h}^*, X_{2,h}^*,  U_h^*)$ solve {\bf SLQ} \eqref{1.3}-\eqref{1.5} and {\bf SLQ}$_h$\eqref{3.1}-\eqref{3.2} problems, respectively. Then there exists a constant $C > 0$ such that} for all $t\in[0,T]$,
	\begin{align}
		&\mathbb{E}[\|U^* - U_h^*\|_{\mathbb{L}^2_{t, x}}^2] + \mathbb{E}[\|\nabla(X_{1}^*(t) - X_{1,h}^*(t))\|_{\mathbb{L}^2_{x}}^2]+\mathbb{E}[\|X_{2}^*(t) - X_{2,h}^*(t)\|_{\mathbb{L}^2_{ x}}^2] \notag\\ &\qquad\leq C h^2\big(\|X_{1,0}\|_{\mathbb{H}_x^3}^2+ \|X_{2,0}\|_{\mathbb{H}_x^2}^2+\|\widetilde{X}\|_{C_t\mathbb{H}_0^1}^2+ \mathbb{E}\big[\|\sigma\|_{\mathbb{L}^2_t\mathbb{H}_x^2}^2\big]\big).
	\end{align}
\end{thm}
\begin{proof}
	For the proof, one can follow similar lines as in the proof of Proposition \ref{Proposition 3.7}. It is a consequence of the error bound on the additional term $\mathbb{E}[\|U^* - U_h^*\|_{\mathbb{L}^2_{t, x}}^2]$, which is established in Theorem \ref{Thm 3.8}.
\end{proof}
\section{Time discretization}\label{Section 4}
We denote by $I_{\tau} = \{t_n\}_{n=0}^{N} \subset [0,T]$ a time mesh with maximum step size $\tau := \max\{t_{n+1} - t_n : n = 0,1, \cdots, N-1\}$, and $\Delta_n W := W(t_n) - W(t_{n-1})$ for all $n = 1, \cdots, N$. Throughout, we assume that $\tau < 1$. For simplicity, we choose a uniform partition, {\em i.e.,} $\tau = T/N$, but the results in this work still hold for quasi-uniform partitions. 
We propose a temporal discretization of problem {\bf SLQ}$_h$ which will be analyzed in Section \ref{Section 3}. For this purpose, we use a mesh $I_\tau$ covering $[0,T]$, and consider step size processes $(X_{h\tau}, U_{h\tau}) \in \mathbb{X}_{h\tau} \times \mathbb{U}_{h\tau} \subset \mathbb{L}^2_{\mathbb{F}}\mathbb{L}^2_{t}(\mathbb{V}_h\times\mathbb{V}_h)$, where
\begin{align*}
	\mathbb{X}_{h\tau} &:= \left\{ X_{h\tau} \in \mathbb{L}^2_{\mathbb{F}}\mathbb{L}^2_{t}\mathbb{V}_h: X_{h\tau}(t) = X_{h\tau}(t_n) \ \forall\, t \in [t_n,t_{n+1}), \ n = 0, 1,\cdots, N \right\}, \\
	\mathbb{U}_{h\tau} &:= \left\{ U_{h\tau} \in \mathbb{L}^2_{\mathbb{F}}\mathbb{L}^2_{t}\mathbb{V}_h: U_{h\tau}(t) = U_{h\tau}(t_n) \ \forall\, t \in [t_n,t_{n+1}), \ n = 0,1,\cdots, N-1 \right\}.
\end{align*}
We also define for any $f\in \mathbb{L}^2(0,T)$,
\begin{align}\label{jhs}
\widehat{f}(t):=\frac{1}{\tau}\int_{t_{n}}^{t_{n+1}}f(\tau)\,{\rm d}\tau \qquad\forall\, s\in (t_n,t_{n+1}],\qquad n=0,....,N-1.\quad\text{and}\quad  \hat{Y}(0)=Y(0).
\end{align} 
We define a projection $\Pi_\tau:C([0,T];\mathbb{K})\to \mathbb{L}^2_t\mathbb{K}$ as follows: for all $X\in C([0,T];\mathbb{K})$,
\begin{align*}
	\Pi_\tau X(t):=X(t_n)\qquad\forall\,t\in[t_n,t_{n+1}),\qquad n=0,1,....,N-1.
\end{align*}
For simplicity, we also define $\widetilde{X}_{h\tau}=\Pi_\tau \widetilde{X}_{h}$.
\subsection{Space-time discretization of SLQ problem \eqref{1.3}-\eqref{1.4}}
Problem {\bf SLQ}$_{h\tau}$ then reads as follows: find an optimal tuple $\big(X_{1,h\tau}^*, X_{2,h\tau}^*, U_{h\tau}^*\big) \in \mathbb{X}_{h\tau} \times \mathbb{U}_{h\tau}$ that minimizes the following quadratic cost functional
\begin{equation}\label{fully discrete cost functional}
	\mathcal{J}_{h\tau}(X_{1, h\tau}, U_{h\tau}) = \frac{1}{2} \mathbb{E}\left[\|X_{1,h\tau}-\widetilde{X}_{h\tau}\|^2_{\mathbb{L}_{t,x}^2} +\alpha \|U_{h\tau}\|^2_{\mathbb{L}_{t,x}^2} + \beta\mathbb{E}\big[\|X_{1,h\tau}(T)-\widetilde{X}_{h\tau}(T)\|_{\mathbb{L}^2_x}^2\big]\right]
\end{equation}
subject to the following forward difference equations; for all $n=0,1,...,N-1,$
\begin{equation}\label{discrete state equation}
	\begin{cases}
		X_{1,h\tau}(t_{n+1})-X_{1,h\tau}(t_n)=\frac{\tau}{2} \big(X_{2,h\tau}(t_{n+1})+X_{2,h\tau}(t_n)\big),\\
		X_{2,h\tau}(t_{n+1}) - X_{2,h\tau}(t_n) = \frac{\tau}{2} \Delta_h \big(X_{1,h\tau}(t_{n+1})+X_{1,h\tau}(t_n)\big) + \tau U_{h\tau}(t_n)+ \left[\gamma X_{1,h\tau}(t_n)+\mathcal{R}_h \sigma(t_n) \right] \Delta_{n+1} W, \\
		X_{1,h\tau}(0) =\mathcal{R}_h X_{1,0},\\
		X_{2,h\tau}(0)=\mathcal{R}_h X_{2,0}.
	\end{cases}
\end{equation}


For given $U_{h\tau}\in\mathbb{U}_{h\tau}$, the tuple $(\mathcal{X}_{1,h\tau}^0[U_{h\tau}], \mathcal{X}_{2,h\tau}^0[U_{h\tau}])\equiv(X_{1,h\tau}^0, X_{2,h\tau}^0)\in \mathbb{X}_{h\tau}^2$ is the unique solution to the following auxiliary random difference equation for $n=0,1,..., N-1$,
\begin{align}
	\label{Boundary_SPDE_ht_zero initial data}
	\begin{cases}
		X_{1,h\tau}^0(t_{n+1})-X_{1,h}^0(t_n)=\frac{\tau}{2}\big(X_{2,h\tau}^0(t_{n+1})+ X_{2,h\tau}^0(t_n)\big),\\
		X_{2,h\tau}^0(t_{n+1}) - X_{2,h\tau}^0(t_n) = \frac{\tau}{2} \left[ \Delta_h \big(X_{1,h\tau}^0(t_{n+1})+X_{1,h\tau}^0(t_n)\big) \right] + \tau U_{h\tau}(t_n) + \gamma X_{1,h\tau}^0(t_n) \Delta_{n+1} W, \\
		X_{1, h\tau}^0(0) = 0,\\
		X_{2,h\tau}^0(0)= 0,
	\end{cases}
\end{align}
which is the space--time discretization of \eqref{1.91}.

In the following, we derive stability estimates for the fully discrete state \((X_{1, h\tau}^0, X_{2, h\tau}^0)\) associated with the equation \eqref{Boundary_SPDE_ht_zero initial data}.
\begin{proposition}[Stability bound]\label{stability bound for discrete state} Let $U_{h\tau}\in \mathbb{U}_{h\tau}$. Then there exists $C>0$ such that
	\begin{align}\label{today0003}
	\sup_{t\in[0,T]}\mathbb{E}[\|\nabla \mathcal{X}_{1,h\tau}^0[U_{h\tau}](t)\|_{\mathbb{L}^2_x}^2+\|\mathcal{X}_{2,h\tau}^0[U_{h\tau}](t)\|_{\mathbb{L}^2_x}^2]\le C\mathbb{E}\big[\|U_{h\tau}\|_{\mathbb{L}^2_{t, x}}^2\big].
	\end{align}
\end{proposition}
\begin{proof}
For the proof, we refer to Appendix~\ref{appendix2}.
\end{proof}
The following lemma gives the stability estimate for the fully discrete state $(X_{1,h\tau}, X_{2,h\tau})$ to the equation \eqref{discrete state equation}.
\begin{lem} Let Assumption~\ref{CC} hold. Then there exists a $C>0$ such that
{\begin{align*}
			&\sup_{t\in[0,T]}\mathbb{E}[\|\nabla \mathcal{X}_{1,h\tau}[U_{h\tau}](t)\|_{\mathbb{L}^2_x}^2+\|\mathcal{X}_{2,h\tau}[U_{h\tau}](t)\|_{\mathbb{L}^2_x}^2]\notag\\&\qquad\le C\big(\|X_{2}(0)\|_{\mathbb{L}^2_x}^2+\|\nabla X_{1}(0)\|_{\mathbb{L}^2_x}^2+\mathbb{E}\big[\|U_{h\tau}\|_{\mathbb{L}^2_{t, x}}^2\big] + \sup_{t\in[0,T]}\mathbb{E}\big[\|\sigma(t)\|_{\mathbb{L}^2_x}^2\big]\big).
	\end{align*}}
\end{lem}
\begin{proof}
	The proof follows similar lines as the one for Proposition \ref{stability bound for discrete state}. 
\end{proof}
\begin{remark}[Solution operator] We define the solution operator $\mathcal{X}_{h\tau}[\cdot]:\mathbb{U}_{h\tau}\to \mathbb{X}_{h\tau}^2$ as follows:
	\begin{align*}
		\mathcal{X}_{h\tau}[U_{h\tau}]=(\mathcal{X}_{1, h\tau}[U_{h\tau}],	\mathcal{X}_{2,h\tau}[U_{h\tau}]),
	\end{align*}
where $(\mathcal{X}_{1, h\tau}[U_{h\tau}],	\mathcal{X}_{2,h\tau}[U_{h\tau}])$ is the unique solution of the forward difference equations \eqref{discrete state equation} with control $U_{h\tau}\in\mathbb{U}_{h\tau}$.
\end{remark}
\begin{remark}[Reduced cost functional]The discrete reduced cost functional is defined as follows: for all $U_{h\tau}\in\mathbb{U}_{h\tau},$
	\begin{align*}
		\hat{\mathcal{J}}_{h\tau}(U_{h\tau}):&=	\mathcal{J}_{h\tau}(\mathcal{X}_{1,h\tau}[U_{h\tau}], U_{h\tau})\\& = \frac{1}{2} \left[\|\mathcal{X}_{1,h\tau}[U_{h\tau}]-\widetilde{X}_{h\tau}\|^2_{\mathbb{L}_{t,x}^2} +\alpha \|U_{h\tau}\|^2_{\mathbb{L}_{t,x}^2} + \beta\mathbb{E}\big[\|\mathcal{X}_{1,h\tau}[U_{h\tau}](T)-\widetilde{X}_{h\tau}(T)\|_{\mathbb{L}^2_x}^2\big]\right]
	\end{align*}
\end{remark}
The following lemma provide an integral identity that will be useful for the proof of the convergence rate below (see Theorem~\ref{strong rate of convergence for time discretization}).
  \begin{lem}[Existence and uniqueness of a discrete optimal control] Let Assumption~\ref{CC} hold. Then there exists a unique optimal tuple $(X_{1,h\tau}^*, X_{2,h\tau}^*, U_{h\tau}^*)$ to the ${\bf SLQ}_{h\tau}$ problem \eqref{fully discrete cost functional}-\eqref{discrete state equation} and the following uniform bound holds:
	\begin{align}\label{uniform bound for discrete optimal pair}
		\sup_{1\le n\le N}\mathbb{E}\big[\|\nabla X_{1,h\tau}^*(t_n)\|_{\mathbb{L}^2_x}^2 + \|U_{h\tau}^*\|_{\mathbb{L}^2_{t, x}}^2\big]\le C(\|X_{1,0}\|_{\mathbb{H}_0^1}^2+ \|X_{2,0}\|_{\mathbb{L}^2_x}^2+\|\widetilde{X}\|_{C_t\mathbb{L}_x^2}^2+ \|\sigma\|_{C_t\mathbb{L}^2_x}^2).
	\end{align} Moreover, the following integral identity holds: for all $V_{h\tau}\in \mathbb{U}_{h\tau}$,
\begin{align}\label{today0005}
		&\left\langle \mathcal{D}_{U}\hat{\mathcal{J}}_{h\tau}(U_{h\tau}^*),V_{h\tau}\right\rangle_{\mathbb{L}^2_\mathbb{F}\mathbb{L}^2_{t, x}}\notag\\&=\mathbb{E}\bigg[\int_0^T\left\langle \mathcal{X}_{1, h\tau}[U_{h\tau}^*](t)-\widetilde{X}_{h\tau}(t),\mathcal{X}_{1,h\tau}^0[V_{h\tau}](t)\right\rangle\,{\rm d}t +\alpha\int_0^T\left\langle U_{h\tau}^*(t), V_{h\tau}(t)\right\rangle\,{\rm d}t\notag \\&\qquad\qquad+\beta\left\langle \mathcal{X}_{1,h\tau}[U_{h\tau}^*](T), \mathcal{X}^0_{1,h\tau}[V_{h\tau}^*](T)\right\rangle\bigg]\notag\\&=0.
	\end{align}
\end{lem}
\begin{proof}
For the existence and uniqueness of the optimal control tuple $(X_{1,h\tau}^*, X_{2,h\tau}^*, U_{h\tau}^*)$, one can follow similar arguments as in the proof of \cite[Theorem 1.43]{Hinze et al. book}; for more details see \cite{LuZhang2021}. The proof of identity~\eqref{today0005} is similar to that of the identity \eqref{today0001}, and we leave its proof to the interested reader.
\end{proof}
\begin{remark}[Fréchet derivative of the reduced cost functional]  
We can compute the Fréchet derivative of the reduced cost functional in variational form.  
For all \( U_{h\tau}, V_{h\tau} \in \mathbb{L}^2_{\mathbb{F}} \mathbb{L}^2_{t}\mathbb{V}_h \), we have  
\begin{align}\label{j2}
&\big\langle \mathcal{D}_{U}\hat{\mathcal{J}}_{h\tau}(U_{h\tau}), V_{h\tau} \big\rangle_{\mathbb{L}^2_{\mathbb{F}}\mathbb{L}^2_{t,x}} \notag\\
&\quad = \mathbb{E} \bigg[ 
    \int_0^T \big\langle \mathcal{X}_{1,h\tau}[U_{h\tau}](t) - \widetilde{X}_{h\tau}(t), \,\mathcal{X}_{1,h\tau}^0[V_{h\tau}](t) \big\rangle \, \,{\rm d}t  
    + \alpha \int_0^T \big\langle U_{h\tau}(t), \, V_{h\tau}(t) \big\rangle \, \,{\rm d}t \notag\\
&\qquad\quad + \beta \big\langle \mathcal{X}_{1,h\tau}[U_{h\tau}](T), \, \mathcal{X}^0_{1,h\tau}[V_{h\tau}](T) \big\rangle 
\bigg].
\end{align}
\end{remark}
The following proposition constitutes a crucial step in avoiding the use of {\em Malliavin calculus} in the subsequent error analysis. 
\begin{proposition}\label{Proposition02} Let Assumption~\ref{CC} hold. Then the following identity holds
	\begin{align}
	&\mathbb{E}\bigg[\int_{0}^T\left\langle X_{1,h}^*(t)-\widetilde{X}_{h\tau}(t), \mathcal{X}_{1, h\tau}^0[U_{h\tau}](t)\right\rangle\,{\rm d}t+ \beta\left\langle X_{1,h}^*(T)-\widetilde{X}_{h}(T), \mathcal{X}_{1,h\tau}^0[U_{h\tau}](T)\right\rangle\bigg]\notag\\&=I_1+I_2+ I_3+I_4+ I_5\label{j1},
\end{align}
where
\begin{align*}
I_1&=\frac{\tau}{2}\sum_{n=0}^{N-1}\mathbb{E}\bigg[\left\langle \big(\mathcal{X}_{2,h\tau}^0[U_{h\tau}](t_{n+1})+\mathcal{X}_{2,h\tau}^0[U_{h\tau}](t_{n})\big), Y_{1,h}(t_{n+1})\right\rangle\bigg]-\sum_{n=0}^{N-1}\mathbb{E}\bigg[\int_{t_{n}}^{t_{n+1}}\left\langle Y_{1,h}(t), X_{2,h\tau}^0(t_{n})\right\rangle\,{\rm d}t\bigg],\\
I_2&=\sum_{n=0}^{N-1}\mathbb{E}\bigg[\int_{t_{n}}^{t_{n+1}}\left\langle \nabla Y_{2,h}(t), \nabla \mathcal{X}_{1,h\tau}^0[U_{h\tau}](t_{n})\right\rangle\,{\rm d}t\bigg]\\&\qquad-\sum_{n=0}^{N-1}\mathbb{E}\bigg[\frac{\tau}{2}\left\langle \nabla \big(\mathcal{X}_{1,h\tau}^0[U_{h\tau}](t_{n+1})+\mathcal{X}_{1,h\tau}^0[U_{h\tau}](t_{n})\big), \nabla Y_{2,h}(t_{n+1})\right\rangle\bigg],\\
I_3&=\tau\sum_{n=0}^{N-1}\mathbb{E}\bigg[\left\langle U_{h\tau}(t_{n}), Y_{2,h}(t_{n+1})\right\rangle\bigg],\\
I_4&=-\sum_{n=0}^{N-1}\mathbb{E}\bigg[\left\langle \int_{t_{n}}^{t_{n+1}}Y_{1,h}(t)\,{\rm d}t,\mathcal{X}_{1,h\tau}^0[U_{h\tau}](t_n)\Delta_{n+1}W\right\rangle\bigg].
\end{align*}
\end{proposition}
\begin{proof} For convenience, we set $(\mathcal{X}_{1,h\tau}^0[U_{h\tau}], \mathcal{X}_{2,h\tau}^0[U_{h\tau}])\equiv (X_{1,h\tau}, X_{2,h\tau})$. We give the proof in several steps as follows:

\noindent
\textbf{Step 1.} By testing $\eqref{Boundary_SPDE_ht_zero initial data}_1$ with $Y_{1,h}(t_{n+1})$ and \eqref{discrete adjoint ode} with $X_{1,h\tau}^0(t_{n})$, we obtain
\begin{align}\label{b1}
	\left\langle X_{1,h\tau}^0(t_{n+1}), Y_{1,h}(t_{n+1})\right\rangle-\left\langle  X_{1,h\tau}^0(t_{n}),Y_{1,h}(t_{n+1})\right\rangle=\frac{\tau}{2} \left\langle \big(X_{2,h\tau}^0(t_{n+1})+X_{2,h\tau}^0(t_n)\big), Y_{1,h}(t_{n+1})\right\rangle
\end{align}
and 
\begin{align}\label{b2}
	 &\left\langle  Y_{1,h}(t_{n+1}), X_{1,h\tau}^0(t_{n})\right\rangle-\left\langle Y_{1,h\tau}(t_n), X_{1,h\tau}^0(t_{n})\right\rangle\notag\\&=\int_{t_{n}}^{t_{n+1}}\left\langle \nabla Y_{2,h}, \nabla X_{1,h\tau}^0(t_{n})\right\rangle\,{\rm d}t -\int_{t_{n}}^{t_{n+1}}\left\langle \gamma Z_{2,h}(t), X_{h\tau}^0(t_{n})\right\rangle\,{\rm d}t\notag\\&\qquad - \int_{t_{n}}^{t_{n+1}}\left\langle X^*_{1,h}(t)-\widetilde{X}_{h\tau}(t),X_{h\tau}^0(t_n)\right\rangle\,{\rm d}t+\int_{t_{n}}^{t_{n+1}}\left\langle Z_{1,h}(t),X_{1,h\tau}(t_{n})\right\rangle\,{\rm d}W(t).
\end{align}
Add identities \eqref{b1}--\eqref{b2} and apply expectations to get
\begin{align*}
&	\mathbb{E}\big[\left\langle  Y_{1,h}(t_{n+1}), X_{1,h\tau}^0(t_{n+1})\right\rangle-\left\langle Y_{1,h\tau}(t_n), X_{1,h\tau}^0(t_{n})\right\rangle\big]\notag\\&=\frac{\tau}{2}\mathbb{E}\bigg[ \left\langle \big(X_{2,h\tau}^0(t_{n+1})+X_{2,h\tau}^0(t_{n})\big), Y_{1,h}(t_{n+1})\right\rangle+\int_{t_{n}}^{t_{n+1}}\left\langle \nabla Y_{2,h}, \nabla X_{1,h\tau}^0(t_{n})\right\rangle\,{\rm d}t\notag\\&\qquad -\int_{t_{n}}^{t_{n+1}}\left\langle \gamma Z_{2,h}(t), X_{1, h\tau}^0(t_{n})\right\rangle\,{\rm d}t - \int_{t_{n}}^{t_{n+1}}\left\langle X^*_{1,h}(t)-\widetilde{X}_{h\tau}(t), X_{1,h\tau}^0(t_n)\right\rangle\,{\rm d}t\bigg].
\end{align*}
After summing over index $n$, it gives
\begin{align*}
    &\sum_{n=0}^{N-1}\int_{t_{n}}^{t_{n+1}}\left\langle X^*_{1,h}(t)-\widetilde{X}_{h\tau}(t), X_{1,h\tau}^0(t_n)\right\rangle\,{\rm d}t+\mathbb{E}\big[\left\langle  Y_{1,h}(t_{N}), X_{1,h\tau}^0(t_{N})\right\rangle-\left\langle Y_{1,h\tau}(t_0), X_{1,h\tau}^0(t_{0})\right\rangle\big]\notag\\&
    =\frac{\tau}{2}\sum_{n=0}^{N-1}\mathbb{E}\bigg[ \left\langle \big(X_{2,h\tau}^0(t_{n+1})+X_{2,h\tau}^0(t_{n})\big), Y_{1,h}(t_{n+1})\right\rangle+\int_{t_{n}}^{t_{n+1}}\left\langle \nabla Y_{2,h}, \nabla X_{1,h\tau}^0(t_{n})\right\rangle\,{\rm d}t\notag\\&\qquad -\int_{t_{n}}^{t_{n+1}}\left\langle \gamma Z_{2,h}(t), X_{1, h\tau}^0(t_{n})\right\rangle\,{\rm d}t\bigg].
\end{align*}
 By  using the facts $Y_{1,h}(t_N)=\beta(X^*_{1,h}(T)-\widetilde{X}_{h\tau}(T))$ and $X_{1,h\tau}^0(t_0)=0$, we obtain
 \begin{align}\label{c1}
    &\sum_{n=0}^{N-1}\int_{t_{n}}^{t_{n+1}}\left\langle X^*_{1,h}(t)-\widetilde{X}_{h\tau}(t), X_{1,h\tau}^0(t_n)\right\rangle\,{\rm d}t+\beta\mathbb{E}\big[\left\langle  (X^*_{1,h}(T)-\widetilde{X}_h(T)), X_{1,h\tau}^0(t_{N})\right\rangle\notag\\&
    =\frac{\tau}{2}\sum_{n=0}^{N-1}\mathbb{E}\bigg[ \left\langle \big(X_{2,h\tau}^0(t_{n+1})+X_{2,h\tau}^0(t_{n})\big), Y_{1,h}(t_{n+1})\right\rangle+\int_{t_{n}}^{t_{n+1}}\left\langle \nabla Y_{2,h}, \nabla X_{1,h\tau}^0(t_{n})\right\rangle\,{\rm d}t\notag\\&\qquad -\int_{t_{n}}^{t_{n+1}}\left\langle \gamma Z_{2,h}(t), X_{1, h\tau}^0(t_{n})\right\rangle\,{\rm d}t\bigg].
\end{align}
\textbf{Step 2.} We test $\eqref{Boundary_SPDE_ht_zero initial data}_2$ with $Y_{2,h}(t_{n+1})$, to get
\begin{align}\label{b3}
	&\left\langle X_{2,h\tau}^0(t_{n+1}), Y_{2,h}(t_{n+1})\right\rangle-\left\langle X_{2,h\tau}^0(t_n),Y_{2,h}(t_{n+1})\right\rangle\notag\\&=\frac{-\tau}{2}\left\langle \nabla \big(X_{1,h\tau}^0(t_{n+1})+X_{1,h\tau}^0(t_{n+1})\big),\nabla Y_{2,h}(t_{n+1})\right\rangle+ \tau\left\langle U_{h\tau}(t_{n}), Y_{2,h}(t_{n+1})\right\rangle\notag\\&\qquad+ \left\langle Y_{2,h}(t_{n+1}),\gamma X_{1,h\tau}^0(t_n)\Delta_{n+1}W\right\rangle.
\end{align} 
\textbf{Step 2(a).} For the last term of r.h.s of the equation \eqref{b3}, by testing \eqref{discrete adjoint ode} with $\gamma X_{1,h\tau}^0(t_n)\Delta_{n+1}W$ , we compute
\begin{align}\label{b4}
	&\left\langle Y_{2,h}(t_{n+1}),\gamma X_{1,h\tau}^0(t_n)\Delta_{n+1}W\right\rangle\notag\\&= \left\langle Y_{2,h}(t_{n}), \gamma X_{1,h\tau}^0(t_n)\Delta_{n+1}W \right\rangle-\left\langle \int_{t_{n+1}}^{t_{n}}Y_{1,h}(t)\,{\rm d}t,\gamma X_{1,h\tau}^0(t_n)\Delta_{n+1}W\right\rangle\notag\\&\qquad+\left\langle  \int_{t_{n}}^{t_{n+1}}Z_{2,h}(t)\,{\rm d}W(t),\gamma X_{1,h\tau}^0(t_n)\Delta_{n+1}W\right\rangle.
\end{align} 
For the first term of the right hand side of the equation \eqref{b4}, by using the independence of Wiener process and covariance of It\^o integral, we conclude that
 \begin{align}\label{b5}
	\mathbb{E}\big[\left\langle Y_{2,h}(t_{n}), \gamma X_{1,h\tau}^0(t_n)\Delta_{n+1}W \right\rangle\big]=0,
\end{align}
and
\begin{align}\label{b6}
	\mathbb{E}\bigg[\left\langle  \int_{t_{n}}^{t_{n+1}}Z_{2,h}({s})\,{\rm d}W(t),\gamma X_{1,h\tau}^0(t_n)\Delta_{n+1}W\right\rangle\bigg]=\mathbb{E}\bigg[\int_{t_{n}}^{t_{n+1}}\left\langle \gamma Z_{2,h}(t), X_{1,h\tau}^0(t_n)\right\rangle\,{\rm d}t\bigg].
\end{align}
From identities \eqref{b3}-\eqref{b6}, we obtain
\begin{align}\label{b7}
&\mathbb{E}\bigg[\left\langle X_{2,h\tau}^0(t_{n+1}), Y_{2,h}(t_{n+1})\right\rangle-\left\langle X_{2,h\tau}^0(t_n),Y_{2,h}(t_{n+1})\right\rangle\bigg]\notag\\&=\mathbb{E}\bigg[\frac{\tau}{2}\left\langle \nabla\big( X_{1,h\tau}^0(t_{n})+X_{1,h\tau}^0(t_{n+1})\big), \nabla Y_{2,h}(t_{n+1})\right\rangle+ \tau\left\langle U_{h\tau}(t_{n}), Y_{2,h}(t_{n+1})\right\rangle\bigg]\notag\\
&\qquad-\mathbb{E}\bigg[\left\langle \int_{t_{n+1}}^{t_{n}}Y_{1,h}(t)\,{\rm d}t, \gamma X_{1,h\tau}^0(t_n)\Delta_{n+1}W\right\rangle\bigg]+\mathbb{E}\bigg[\int_{t_{n}}^{t_{n+1}}\left\langle \gamma Z_{2,h}(t), X_{1,h\tau}^0(t_n)\right\rangle\,{\rm d}t\bigg].
\end{align}
\textbf{Step 2(b).} On the other hand, by testing \eqref{discrete adjoint ode} with $X_{2,h\tau}^0(t_{n})$, we obtain
\begin{align}\label{b8}
	\left\langle Y_{2,h}(t_{n+1}), X_{2,h\tau}^0(t_{n})\right\rangle-\left\langle Y_{2,h}(t_n),X_{2,h\tau}^0(t_{n})\right\rangle&=-\int_{t_{n}}^{t_{n+1}}\left\langle Y_{1,h}(t), X_{2,h\tau}^0(t_{n})\right\rangle\notag\\&
\qquad +\int_{t_{n}}^{t_{n+1}}\left\langle \gamma Z_{2,h}(t),X_{2,h\tau}^0(t_n)\right\rangle\,{\rm d}W(t).
\end{align}
\textbf{Step 2(c).} Adding identities \eqref{b7}--\eqref{b8} then give
\begin{align}\label{c2}
&	\mathbb{E}\bigg[\left\langle Y_{2,h}(t_{n+1}),X_{2,h\tau}^0(t_{n+1})\right\rangle-\left\langle Y_{2,h}(t_n),X_{2,h\tau}^0(t_{n})\right\rangle\bigg]\notag\\&=-\mathbb{E}\bigg[\int_{t_{n}}^{t_{n+1}}\left\langle Y_{1,h}(t), X_{2,h\tau}^0(t_{n})\right\rangle\,{\rm d}t\bigg]+\mathbb{E}\bigg[\frac{\tau}{2}\bigg \langle \nabla\big( X_{1,h\tau}^0(t_{n+1})+X_{1,h\tau}^0(t_{n})\big), Y_{2,h}(t_{n+1})\bigg\rangle\notag\\&\qquad+ \tau\left\langle U_{h\tau}(t_{n}), Y_{2,h}(t_{n+1})\right\rangle\bigg]-\mathbb{E}\bigg[\left\langle \int_{t_{n+1}}^{t_{n}}Y_{1,h}(t)\,{\rm d}t,X_{1,h\tau}^0(t_n)\Delta_{n+1}W\right\rangle\bigg]\notag\\&\qquad+\mathbb{E}\bigg[\int_{t_{n}}^{t_{n+1}}\left\langle \gamma Z_{2,h}(t), X_{1,h\tau}^0(t_n)\right\rangle\,{\rm d}t\bigg].
\end{align}
After summing over index $n$, we get
\begin{align*}
    &\mathbb{E}\bigg[\left\langle Y_{2,h}(t_{N}),X_{2,h\tau}^0(t_{N})\right\rangle-\left\langle Y_{2,h}(t_0),X_{2,h\tau}^0(t_{0})\right\rangle\bigg]\notag\\&=-\sum_{n=0}^{N-1}\mathbb{E}\bigg[\int_{t_{n}}^{t_{n+1}}\left\langle Y_{1,h}(t), X_{2,h\tau}^0(t_{n})\right\rangle\,{\rm d}t + \frac{\tau}{2}\left\langle \nabla\big( X_{1,h\tau}^0(t_{n+1})+X_{1,h\tau}^0(t_{n})\big), Y_{2,h}(t_{n+1})\right\rangle\notag\\&\qquad+\tau\left\langle U_{h\tau}(t_{n}), Y_{2,h}(t_{n+1})\right\rangle-\left\langle \int_{t_{n+1}}^{t_{n}}Y_{1,h}(t)\,{\rm d}t,X_{1,h\tau}^0(t_n)\Delta_{n+1}W\right\rangle+\int_{t_{n}}^{t_{n+1}}\left\langle \gamma Z_{2,h}(t), X_{1,h\tau}^0(t_n)\right\rangle\,{\rm d}t\bigg].
\end{align*}
By using the fact $Y_{2,h}(t_N)=X_{2,h\tau}^0(t_0)=0$, we get
\begin{align}\label{c2}
    -&\mathbb{E}\bigg[\int_{t_{n}}^{t_{n+1}}\left\langle \gamma Z_{2,h}(t), X_{1,h\tau}^0(t_n)\right\rangle\,{\rm d}t\bigg]\notag\\&=-\sum_{n=0}^{N-1}\mathbb{E}\bigg[\int_{t_{n}}^{t_{n+1}}\left\langle Y_{1,h}(t), X_{2,h\tau}^0(t_{n})\right\rangle\,{\rm d}t + \frac{\tau}{2}\left\langle \nabla\big( X_{1,h\tau}^0(t_{n+1})+X_{1,h\tau}^0(t_{n})\big), Y_{2,h}(t_{n+1})\right\rangle\notag\\&\qquad+\tau\left\langle U_{h\tau}(t_{n}), Y_{2,h}(t_{n+1})\right\rangle-\left\langle \int_{t_{n+1}}^{t_{n}}Y_{1,h}(t)\,{\rm d}t,X_{1,h\tau}^0(t_n)\Delta_{n+1}W\right\rangle\bigg].
\end{align}
\textbf{Step 3.} By adding \eqref{c1} and \eqref{c2}, we conclude that
\begin{align*}
	 &\mathbb{E}\bigg[\int_{0}^T\left\langle X_{1,h}^*(t)-\widetilde{X}_{h\tau}, {X}_{1, h\tau}^0(t)\right\rangle\,{\rm d}t+ \beta\left\langle X_{1,h}^*(T)-\widetilde{X}_{h\tau}(T), {X}_{1,h\tau}^0(T)\right\rangle\bigg]\\&=I_1+I_2+ I_3+I_4+ I_5.
\end{align*}
This completes the proof.
\end{proof}
\begin{remark}[On avoiding {\em Malliavin calculus} in the present error analysis]\label{Remark 4.3}
	Previous works on strong error estimates for discretizations of the stochastic optimal control problems, such as those for stochastic heat equations with multiplicative noise ({\em e.g.,} see~\cite{Prohl&Wang1, ProhlWang2021}), relied on {\em Malliavin calculus} to prove time-discretization error estimates in \cite[Lemmas 3.11--3.13]{ProhlWang2021}. This was necessary due to $Z_h$’s role in the drift term, requiring extensive technical machinery ({\em e.g.,} see~\cite[Sec.~3.3, pg. 3401 to pg. 3421]{ProhlWang2021}).
    
    In contrast, our error analysis bypasses {\em Malliavin calculus} by reformulating the Fréchet derivative of the discrete cost functional, $\mathcal{D}_{U_{h\tau}}\hat{\mathcal{J}}_{h\tau}(U_{h\tau}^* )$ at the fully discrete optimal control $U_{h\tau}^*$ without involving the drift term $Z_h = (Z_{1,h}, Z_{2,h})$; see~equations~\eqref{j2} and \eqref{j1}. This enables us to derive all temporal regularity estimates within a variational framework, with the key error terms provided by a single proposition; see Proposition~\ref{Proposition02}.
\end{remark}
\subsection{Error analysis for space-time discretization} In this subsection, we estimate the error between the fully discrete optimal tuple $(X_{1, h\tau}^*, X_{2,h\tau}^*, U_{h\tau}^*)$ and the semi-discrete optimal tuple $(X_{1,h}^*, X_{2,h}^*, U_h^*)$ in suitable norms. To this end, we introduce several technical propositions and lemmas. Moreover, Assumption~\ref{BB} give that there exists a constant \(C>0\), independent of the discretization parameters \(h\) and \(\tau\), such that
\begin{align}\label{assumption on noise coefficitents}
	\|\mathcal{R}_h\sigma-\Pi_\tau\mathcal{R}_h\sigma\|_{\mathbb{L}^2_{t,x}}^2
	+ \|\widetilde{X}_{h}-\widetilde{X}_{h\tau}\|_{\mathbb{L}^2_{t,x}}^2
	\le C\,\tau \big(\|\widetilde{X}\|_{C_t^{1/2}\mathbb{H}_0^1}^2+\|\sigma\|_{\mathbb{L}^2_{\mathbb{F}}C^{1/2}_t\mathbb{H}_0^1}^2\big).
\end{align} 
We now state the following proposition, which provides the error estimate between the semi-disctere state $\mathcal{X}_{1,h}[U_{h\tau}]$ and the fully-discrete state $\mathcal{X}_{1,h\tau}[U_{h\tau}]$ corresponding to the same semi-discrete control $U_{h\tau}$. This result will be useful in the proof of Theorem~\ref{strong rate of convergence for time discretization}.
\begin{proposition}[Error estimate]\label{Proposition 4.7} Let $U_{h\tau}\in \mathbb{U}_{h\tau}$ and Assumption~\ref{BB} hold. Then there exists a $C>0$ such that for all $t\in[0,T]$,
	\begin{align}\label{today0004}
	&\mathbb{E}\bigg[\|\nabla \big(\mathcal{X}_{1,h}[U_{h\tau}](t)-\mathcal{X}_{1,h\tau}[U_{h\tau}](t)\|_{\mathbb{L}^2_{x}}^2\bigg]+\mathbb{E}\bigg[\| \mathcal{X}_{2,h}[U_{h\tau}](t)-\mathcal{X}_{2,h\tau}[U_{h\tau}](t)\|_{\mathbb{L}^2_{x}}^2\big]\bigg]\notag\\&\qquad\le C \tau 
    \bigg(\|X_{1,0}\|_{\mathbb{H}_x^3}^2+ \|X_{2,0}\|_{\mathbb{H}^2_x}^2+\mathbb{E}\big[ \|\sigma\|_{\mathbb{L}^2_t\mathbb{H}^2_x\cap C_t^{1/2}\mathbb{H}_0^1}^2+ \|\nabla\Delta_h U_{h\tau}\|_{\mathbb{L}_{t, x}^2}^2\big]\bigg).
	\end{align}
\begin{proof} For convenience, we set $$(X_{1,h},X_{2,h})\equiv(\mathcal{X}_{1,h}[U_{h\tau}], \mathcal{X}_{2,h}[U_{h\tau}]),\qquad (X_{1,h\tau},X_{2,h\tau})\equiv(\mathcal{X}_{1,h\tau}[U_{h\tau}], \mathcal{X}_{2,h\tau}[U_{h\tau}]).$$
We have  for all $n\in\{0,1,2,...,N-1\}$
	\begin{align}
		X_{1,h}(t_{n+1})-X_{1,h}(t_n)&=\int_{t_n}^{t_{n+1}}X_{2,h}(t)\,{\rm d}t,\label{toaday2}\\	
	X_{2,h}(t_{n+1})-X_{2,h}(t_n)&=\int_{t_n}^{t_{n+1}}\Delta_h X_{1,h}(t)\,{\rm d}t +\int_{t_n}^{t_{n+1}}U_{h\tau}(t)\,{\rm d}t + \int_{t_n}^{t_{n+1}}\big(\mathcal{R}_h\sigma(t)+\gamma X_{1,h}(t)\big)\,{\rm d}W(t)\label{toaday21}.
		\end{align}
We define for all $n=0,1,...,N,$
\begin{align*}
	e_{n}^1= X_{1,h}(t_n)-X_{1,h\tau}(t_n), \qquad e_{n}^{2}= X_{2,h}(t_n)-X_{2,h\tau}(t_n).
\end{align*}
From~\eqref{discrete state equation} and \eqref{toaday2}-\eqref{toaday21}, we conclude that 
\begin{align}e_{n+1}^1-e_{n}^1&=\frac{\tau}{2} (e_{n+1}^2+e_{n}^2) +\frac{1}{2} \int_{t_n}^{t_{n+1}}\big(X_{2,h}(t)-X_{2,h}(t_{n+1})\big)\,{\rm d}t+\frac{1}{2} \int_{t_n}^{t_{n+1}}\big(X_{2,h}(t)-X_{2,h}(t_{n})\big)\,{\rm d}t,\label{andreas1}\\
	e_{n+1}^2-e_{n}^2&=\frac{\tau}{2}\Delta_h e_{n+1}^1 +\frac{\tau}{2}\Delta_h e_{n}^1+ \gamma e_{n}^1\Delta_{n+1}W+ \frac{1}{2}\int_{t_n}^{t_{n+1}} \Delta_h(X_{1,h}(t)-X_{1,h}(t_{n+1}))\,{\rm d}t\notag \\&\qquad+\frac{1}{2}\int_{t_n}^{t_{n+1}} \Delta_h(X_{1,h}(t)-X_{1,h}(t_{n}))\,{\rm d}t\notag\\&\qquad+ \int_{t_{n}}^{t_{n+1}}\big[(\mathcal{R}_h\sigma(t)-\mathcal{R}_h\sigma(t_n))+\gamma(X_{1,h}(t)-X_{1,h}(t_n)) \big]\,{\rm d}W(t).\label{andreas}
\end{align}
We test \eqref{andreas} with $e_{n+1}^2+e_{n}^2$ write to arrive at
\begin{align}\label{a1}
	\left\langle e_{n+1}^2-e_{n}^2, e_{n+1}^2+ e_{n}^2\right\rangle&=I_1+ I_2+ I_3+ I_4,
\end{align}
where
\begin{align*}
I_1(n)&=	-\tau\left\langle \nabla (e_{n+1}^2+e_{n}^2), \nabla (e_{n+1}^1+e_{n}^1)\right\rangle, \\
I_2(n)&=\int_{t_n}^{t_{n+1}}\left\langle \nabla (e_{n+1}^2+e_{n}^2), \nabla X_{1,h}(t)-\nabla \frac{1}{2}\big(X_{1,h}(t_{n+1})+X_{1,h}(t_n)\big) \right\rangle\,{\rm d}t,\\
I_3(n)&=\gamma\left\langle e_n^1, e_{n+1}^2+ e_{n}^2\right\rangle\Delta_{n+1}W,\\
I_4(n)&=\left\langle  e_{n+1}^2+e_{n}^2, \int_{t_{n+1}}^{t_{n}}\bigg[(\mathcal{R}_h\sigma(t)-\mathcal{R}_h\sigma(t_n))+\gamma(X_{1,h}(t)-X_{1,h}(t_n))\bigg]\,{\rm d}W(t)\right\rangle.
\end{align*}
We estimate each term separately.

\noindent
\textbf{Step 1.} We start with the term $I_1$. For this purpose, we test \eqref{andreas} with $\Delta_h(e_{n+1}^1+e_{n}^1)$ to conclude with the help of \eqref{andreas1} that
\begin{align*}
	I_1(n)&=  -\|\nabla e_{n+1}^1\|^2+\|\nabla e_{n}^1\|^2-\frac{1}{2}\int_{t_n}^{{n+1}}\left\langle \nabla\big(X_{2,h}(t)-X_{2,h}(t_{t_{n+1}})\big), \nabla\big(e_{n+1}^1+e_{n}^1\big)\right\rangle\,{\rm d}t\\&\qquad-\frac{1}{2}\int_{t_n}^{{n+1}}\left\langle \nabla\big(X_{2,h}(t)-X_{2,h}(t_{t_{n}})\big),\nabla(e_{n+1}^1+e_{n}^1)\right\rangle\,{\rm d}t.
\end{align*} 
After summation and by using Young's inequality, we obtain that  $(\delta>0)$
\begin{align*}
	\mathbb{E}\big[I_{1}(n)\big]&\le -\mathbb{E}\big[\|\nabla e_{n+1}^1\|^2\big]+\mathbb{E}\big[\|\nabla e_{n}^1\|^2\big]+\delta\tau\mathbb{E}\big[\|\nabla e_{n+1}^1\|_{\mathbb{L}^2_{x}}^2\big]+\delta\tau \mathbb{E}\big[\|\nabla e_{n}^1\|_{\mathbb{L}^2_x}^2\big]\\&\qquad + C_\delta\mathbb{E}\bigg[\int_{t_{n}}^{t_{n+1}}\|\nabla(X_{2,h}(t)-X_{2,h}(t_{n+1}))\|_{\mathbb{L}^2_x}^2\,{\rm d}t + \int_{t_{n}}^{t_{n+1}}\|\nabla(X_{2,h}(t)-X_{2,h}(t_n))\|_{\mathbb{L}^2_{x}}^2\,{\rm d}t\bigg].
\end{align*}
Using the estimate \eqref{time-regularity for semi-discrete state-2}, we obtain
\begin{align*}
&\sum_{n=0}^{k-1}\mathbb{E}\big[I_{1}(n)\big]\le-\mathbb{E}\big[\|\nabla e_{k}^1\|^2_{\mathbb{L}^2_x}\big]+\tau \delta \mathbb{E}\big[\|\nabla e_{k}^1\|^2_{\mathbb{L}^2_x}\big] + C_\delta\tau\sum_{n=0}^{k-1}\mathbb{E}\big[\|\nabla e_{n}^1\|^2_{\mathbb{L}^2_x}\big]\\&\qquad+ C_\delta \tau \big(\|X_{1,0}\|_{\mathbb{H}_x^3}^2+ \|X_{2,0}\|_{\mathbb{H}^2_x}^2+\mathbb{E}\big[ \|\sigma\|_{\mathbb{L}^2_t\mathbb{H}^2_x}^2+ \|\nabla\Delta_hU_{h}\|_{\mathbb{L}_{t, x}^2}^2\big]\big).
\end{align*}
\textbf{Step 2.} We consider the term $I_2$. With the help of Young's inequality, we obtain $(\delta>0)$
\begin{align*}
	I_2(n)&\le \delta\tau (\|e_{n+1}^2\|_{\mathbb{L}^2_x}^2+\|e_{n}^2\|_{\mathbb{L}^2_x}^2) +C_\delta\int_{t_{n}}^{t_{n+1}}\big\|\Delta_h (X_{1,h}(t)-X_{1,h}(t_{n+1}))\big\|_{\mathbb{L}^2_x}^2\,{\rm d}t.
\end{align*}
With the help of \eqref{time-regularity for semi-discrete state-1}, we obtain $(\delta>0)$
\begin{align}\label{a3}
\mathbb{E}\big[I_2(n)\big]\le \tau\delta\mathbb{E}\big[ \|e_{n+1}^2\|_{\mathbb{L}^2_x}^2\big]+\tau\mathbb{E}\big[ \|e_{n}^2\|_{\mathbb{L}^2_x}^2\big] +C_\delta\,\tau\big(\|X_{1,0}\|_{\mathbb{H}_x^3}^2+ \|X_{2,0}\|_{\mathbb{H}^2_x}^2+\mathbb{E}\big[ \|\sigma\|_{\mathbb{L}^2_t\mathbb{H}^2_x}^2+ \|\nabla\Delta_hU_{h}\|_{\mathbb{L}_{t, x}^2}^2\big]\big).
\end{align}
It implies that
\begin{align*}
	&\sum_{n=0}^{k-1}\mathbb{E}\big[I_{2}(n)\big]\le \delta\tau\mathbb{E}\big[\|e_{k}^2\|_{\mathbb{L}^2_x}^2\big]+\tau C_{\delta}\sum_{n=0}^{k-1}\mathbb{E}\big[\|e_{n}^2\|_{\mathbb{L}^2_x}^2\big]\\&\qquad+ C_\delta \tau \big(\|X_{1,0}\|_{\mathbb{H}_x^3}^2+ \|X_{2,0}\|_{\mathbb{H}^2_x}^2+\mathbb{E}\big[ \|\sigma\|_{\mathbb{L}^2_t\mathbb{H}^2_x}^2+ \|\nabla\Delta_hU_{h}\|_{\mathbb{L}_{t, x}^2}^2\big]\big).
	\end{align*}
\textbf{Step 3.} In this step, we estimate the term $I_3$. By independence of Wiener process, we have
\begin{align*}
\mathbb{E}\big[I_3(n)\big]& =\mathbb{E}\big[\big\langle e_n^1,\, \gamma e_{n}^1 \big\rangle \Delta_{n+1}W\big] + \mathbb{E}\big[\big\langle e_{n+1}^2,\, \gamma e_{n}^1 \big\rangle \Delta_{n+1}W\big]\\
&=\mathbb{E}\big[\big\langle e_{n+1}^2,\, \gamma e_{n}^1 \big\rangle \Delta_{n+1}W\big].
\end{align*}
Since $e_{n+1}^2$ is not $\mathcal{F}_{t_n}$-measurable, we expand $e_{n+1}^2$ using the recursion \eqref{andreas}.
In order to now estimate $I_3(n)$, we test \eqref{andreas} with $\gamma e_{n}^1\Delta_{n+1}W$ to obtain
\begin{align*}
	&I_3(n) = \gamma\mathbb{E}\bigg[\big\langle e_n^1,\, e_n^2 \big\rangle \Delta_{n+1}W  - \big\langle \nabla e_n^1,\, \frac{\tau}{2}\nabla e_{n+1}^1 + \frac{\tau}{2}\nabla e_n^1 \big\rangle \Delta_{n+1}W \\
	&- \big\langle \nabla e_n^1,\, \frac{1}{2}\int_{t_n}^{t_{n+1}} \nabla(X_{1,h}(t) - X_{1,h}(t_{n+1}))\,{\rm d}t \big\rangle \Delta_{n+1}W - \big\langle \nabla e_n^1,\, \frac{1}{2}\int_{t_n}^{t_{n+1}} \nabla(X_{1,h}(t) - X_{1,h}(t_n))\,{\rm d}t \big\rangle \Delta_{n+1}W \\
	&+ \big\langle \gamma e_n^1,\, e_n^1 \Delta_{n+1}W \big\rangle \Delta_{n+1}W  + \big\langle e_n^1,\, \int_{t_n}^{t_{n+1}} [(\Pi_h\sigma(t) - \Pi_h\sigma(t_n)) +\gamma (X_{1,h}(t) - X_{1,h}(t_n))] \,{\rm d}W(t)] \big\rangle \Delta_{n+1}W\bigg].
\end{align*}
\textbf{(a):} For the first term, since $\big\langle e_n^1 , e_n^2 \big\rangle$, is $\mathcal{F}_{t_n}$-measurable, we arrive at
\[
\mathbb{E}\left[ \sum_n \big\langle e_n^1,\, e_n^2 \big\rangle \Delta_{n+1}W \right]=0.\]
\textbf{(b):} For the second term, we use Young's inequality, independence of random variables, and It\^o isometry to get $(\delta>0)$
\begin{align*}
	\mathbb{E}\big[\left\langle \nabla e_{n}^1,\frac{\tau}{2}\nabla(e_{n+1}^1+e_{n}^1)\Delta_{n+1}W\right\rangle\big]&=\mathbb{E}\big[\left\langle \nabla e_{n}^1,\frac{\tau}{2}\nabla e_{n+1}^1\Delta_{n+1}W\right\rangle\big]\\
	&\le \tau^2\delta \mathbb{E}\big[\|\nabla e_{n+1}^1\|^2_{\mathbb{L}^2_x}\big]+ C_{\delta}\tau \mathbb{E}\big[\|\nabla e_{n}^1\|_{\mathbb{L}^2_x}^2\big].
\end{align*}
This implies that for any $N-1\ge k\ge 1$,
\begin{align*}
	\sum_{n=0}^{k-1}\mathbb{E}\big[\left\langle \nabla e_{n}^1,\frac{\tau}{2}\nabla(e_{n+1}^1+e_{n}^1)\Delta_{n+1}W\right\rangle\big]\le \delta \mathbb{E}\big[\|\nabla e_{k}^1\|_{\mathbb{L}^2_x}^2\big]+ C_{\delta}\tau\sum_{n=0}^{k-1}\mathbb{E}\big[\|\nabla e_{n}^1\|_{\mathbb{L}^2_x}^2\big].
\end{align*}
\textbf{(c):} For the third term, with the help of the estimate  \eqref{time-regularity for semi-discrete state-1}, we obtain that $(\delta>0)$
\begin{align*}
	\sum_{n=0}^{k-1}\mathbb{E}\bigg[\big\langle \nabla e_n^1,\, \frac{1}{2}\int_{t_n}^{t_{n+1}} \nabla(X_{1,h}(t) - X_{1,h}(t_{n+1}))\,{\rm d}t\big\rangle \Delta_{n+1}W \bigg]&\le \delta \mathbb{E}\big[\|\nabla e_{k}^1\|_{\mathbb{L}^2_x}^2\big]+ C_{\delta}\tau\sum_{n=0}^{k-1} \mathbb{E}\big[\|\nabla e_{n}^1\|_{\mathbb{L}^2_x}^2\big]\\&+ C_{\delta}\tau\mathbb{E}\big[\|\nabla\Delta_h U_{h\tau}\|_{\mathbb{L}^2_x}^2\big].
\end{align*}
\textbf{(d):} For the fifth term, with the help of the estimate \eqref{time-regularity for semi-discrete state-1}, we obtain that $(\delta>0)$
\begin{align*}
	\sum_{n=0}^{N-1}\mathbb{E}\bigg[\big\langle \nabla e_n^1,\, \frac{1}{2}\int_{t_n}^{t_{n+1}} \nabla(X_{1,h}(t) - X_{1,h}(t_n))\,{\rm d} \big\rangle \Delta_{n+1}W \bigg]&\le \delta \mathbb{E}\big[\|\nabla e_{k}^1\|_{\mathbb{L}^2_x}^2\big]+C_{\delta}\tau\sum_{n=0}^{k-1}\mathbb{E}\big[\|\nabla e_{n}^1\|^2_{\mathbb{L}^2_x}\big]\\&+ C_{\delta}\tau\mathbb{E}\big[\|\nabla\Delta_h U_{h\tau}\|_{\mathbb{L}^2_{t, x}}^2\big]. 
\end{align*}
\textbf{(e):} By independence and It\^o isometry and Poincar\'e inequality, the quadratic term $\big\langle e_n^1, e_n^1 \Delta_{n+1} W \big\rangle \Delta_{n+1} W$ is handled as
\[
\mathbb{E}\big[\big\langle e_n^1, e_n^1 \Delta_{n+1} W \big\rangle \Delta_{n+1} W \big]= \tau\mathbb{E}\big[\|e_{n}^1\|_{\mathbb{L}^2_x}^2\big]\le C\tau\mathbb{E}\big[\|\nabla e_{n}^1\|_{\mathbb{L}^2_x}^2\big].
\]
\textbf{(f):} For the final term, we use the estimates \eqref{assumption on noise coefficitents} and \eqref{time-regularity for semi-discrete state-1} to get that
\begin{align*}
&\sum_{n=0}^{k-1}\mathbb{E}\bigg[\big\langle e_n^1,\, \int_{t_n}^{t_{n+1}} \big[\mathcal{R}_h\sigma(t) - \mathcal{R}_h\sigma(t_n)) + (X_{1,h}(t) - X_{1,h}(t_n))\big] \,{\rm d}W(t) \big\rangle \Delta_{n+1}W\bigg]\\&\le \sum_{n=0}^{k-1}\tau\mathbb{E}\bigg[\|e_{n}^1\|_{\mathbb{L}^2_x}^2\bigg]+C\tau\big(\|X_{1,0}\|_{\mathbb{H}_x^3}^2+ \|X_{2,0}\|_{\mathbb{H}^2_x}^2+\mathbb{E}\big[ \|\sigma\|_{\mathbb{L}^2_t\mathbb{H}^2_x\cap C^{1/2}_t\mathbb{H}_0^1}^2+ \|\nabla\Delta_hU_{h}\|_{\mathbb{L}_{t, x}^2}^2\big]\big).
\end{align*}
Thus, finally, we get $(\delta>0)$
\begin{align*}
&	\sum_{n=0}^{k-1}\mathbb{E}\big[I_3(n)\big]\le \delta\mathbb{E}\big[\|\nabla e_{k}^1\|^2_{\mathbb{L}^2_{x}}\big]+C_{\delta}\tau\sum_{n=0}^{k-1}\mathbb{E}\big[\|\nabla e_{n}^1\|_{\mathbb{L}^2_{x}}^2\big]\\&\qquad+ C_\delta \tau \bigg(\|X_{1,0}\|_{\mathbb{H}_x^3}^2+ \|X_{2,0}\|_{\mathbb{H}^2_x}^2+ \mathbb{E}\big[\|\sigma\|_{\mathbb{L}^2_t\mathbb{H}_x^2\cap C^{1/2}_t\mathbb{H}_0^1}^2+ \|\nabla\Delta_hU_{h}\|_{\mathbb{L}_{t, x}^2}^2\big]\bigg).
\end{align*}
\noindent
\textbf{Step 4.} In this step, with the help of the estimate \eqref{assumption on noise coefficitents}, we can estimate the term $I_4$ in a similar way as in Step 3, we yield
\begin{align*}
 \sum_{n=0}^{k-1}\mathbb{E}\big[I_{4}(n)\big]\le &\delta \mathbb{E}\big[\|\nabla e_{k}^1\|_{\mathbb{L}^2_x}^2\big]+ C_{\delta}\tau\sum_{n=0}^{k-1}\mathbb{E}\big[\|\nabla e_{n}^1\|_{\mathbb{L}^2_x}^2\big]\\&+C_\delta \tau\big(\|X_{1,0}\|_{\mathbb{H}_x^3}^2+ \|X_{2,0}\|_{\mathbb{H}^2_x}^2+\mathbb{E}\big[ \|\sigma\|_{\mathbb{L}^2_t\mathbb{H}^2_xC^{1/2}_t\mathbb{H}_0^1}^2+ \|\nabla\Delta_hU_{h}\|_{\mathbb{L}_{t, x}^2}^2\big]\big).
\end{align*}
\textbf{Step 5.} From the last steps and choosing small enough $\delta>0$, we get for any $1\le k\le N-1$,
\begin{align*}
	&\mathbb{E}\big[\|\nabla e_{k}^1\|_{\mathbb{L}^2_x}^2+ \|e_{k}^2\|_{\mathbb{L}^2_x}^2\big]\le C\tau \sum_{n=0}^{k-1}\mathbb{E}\big[\|\nabla e_{n}^1\|_{\mathbb{L}^2_x}^2+\|e_{n}^2\|_{\mathbb{L}_x^2}^2\big]+ \\&\qquad+ C \tau \big(\|X_{1,0}\|_{\mathbb{H}_x^3}^2+ \|X_{2,0}\|_{\mathbb{H}^2_x}^2+\mathbb{E}\big[ \|\sigma\|_{\mathbb{L}^2_t\mathbb{H}^2_x\cap C^{1/2}_t\mathbb{H}_0^1}^2+ \|\nabla\Delta_hU_{h}\|_{\mathbb{L}_{t, x}^2}^2\big]\big).
\end{align*}
We use discrete Gronwall's inequality to conclude the estimate \eqref{today0004}.
\end{proof}
\end{proposition}
In the proof of Proposition~\ref{Proposition 4.7}, it is clear that estimating \(I_1(n)\) requires H\"older time regularity of \(X_{2,h} = \partial_t X_{1,h}\), which is limited up to \(1/2\) (see, equation~\eqref{time-regularity for semi-discrete state-2}). Consequently, this limitation results in a convergence rate of order \(1/2\) in the proposition.

The following theorem establishes the rate of convergence of \(\mathbf{SLQ}_{h\tau}\) problem \eqref{fully discrete cost functional}-\eqref{discrete state equation} to \(\mathbf{SLQ}_h\) problem \eqref{3.1}-\eqref{3.2}. 
\begin{thm}\label{strong rate of convergence for time discretization}Let Assumption~\ref{BB} hold. Let $(X_h^*, U_h^*)$ and $(X_{1,h\tau}^*, U_{h\tau}^*)$ be solve {\bf SLQ}$_h$ problem \eqref{3.1}-\eqref{3.2} and ${\bf SLQ}_{h\tau}$ problem \eqref{fully discrete cost functional}-\eqref{discrete state equation}, respectively. Then there exists a positive constant $C$ such that
	\begin{align*}
		&\mathbb{E}\big[\|U_h^*-U_{h\tau}^*\|_{\mathbb{L}^2_{t, x}}^2\big]+\mathbb{E}\big[\|X^*_{1,h}- X_{1,h\tau}^*\|_{\mathbb{L}^2_{t, x}}^2\big]+ \beta\mathbb{E}\big[\|X^*_h(T)- X_{1,h\tau}^*(T)\|_{\mathbb{L}^2_x}^2\big] \\&\le C\tau\big(\|X_{1,0}\|_{\mathbb{H}_x^3}^2+ \|X_{2,0}\|_{\mathbb{H}^2_x}^2+\|\widetilde{X}\|_{\mathbb{L}^2_t\mathbb{H}_x^2\cap C_t^{1/2}\mathbb{H}_0^1}^2+\mathbb{E}\big[ \|\sigma\|_{\mathbb{L}^2_t\mathbb{H}^2_x\cap C^{1/2}_t\mathbb{H}_0^1}^2\big]\big).
	\end{align*} 
	
\end{thm}
\begin{proof} We will complete the proof in several steps as follows.
	
	\noindent
	\textbf{Step 1.} We have
    \begin{align*}
        &\mathbb{E}\bigg[\int_0^T\alpha\big\langle U_h^*(t)- U_{h\tau}^*(t), V_{h\tau}(t)\big\rangle \,{\rm d}t\bigg]=\mathbb{E}\bigg[\int_0^T\alpha\big\langle U_h^*(t), V_{h\tau}(t)\big\rangle \,{\rm d}t\bigg]-\mathbb{E}\bigg[\int_0^T\alpha\big\langle U_{h\tau}^*(t), V_{h\tau}(t)\big\rangle \,{\rm d}t\bigg].
    \end{align*}
    We use the integral identities \eqref{today0000} and \eqref{today0005} to conclude that for all $V_{h\tau}\in\mathbb{U}_{h\tau}$,
	\begin{align*}
		&\mathbb{E}\bigg[\int_0^T\alpha\big\langle U_h^*(t)- U_{h\tau}^*(t), V_{h\tau}(t)\big\rangle \,{\rm d}t\bigg]\\&=\mathbb{E}\bigg[\int_0^T\big\langle \widetilde{X}_h(t)-X_{1,h}^*(t), \mathcal{X}_{1,h}^0[ V_{h\tau}](t)\big\rangle \,{\rm d}t\bigg]+ \beta\mathbb{E}\big[\big\langle \widetilde{X}_h(T)-X_{1,h}^*(T), \mathcal{X}_{1,h}^0[ V_{h\tau}](T)\big\rangle\big]\\&\qquad-\mathbb{E}\bigg[\int_0^T\big\langle \widetilde{X}_{h\tau}(t)- X_{1,h\tau}^*(t), \mathcal{X}_{1,h\tau}^0[V_{h\tau}](t)\big\rangle \,{\rm d}t\bigg]-\beta\mathbb{E}\big[\big\langle \widetilde{X}_{h\tau}(T)- X_{1,h\tau}^*(T), \mathcal{X}_{1,h\tau}^0[V_{h\tau}](T)\big\rangle\big]\\&=-\bigg\{\mathbb{E}\bigg[\int_0^T\big\langle X_{1,h}^*(t)-  X_{1,h\tau}^*(t), \mathcal{X}_{1,h\tau}^0[V_{h\tau}](t)\big\rangle \,{\rm d}t\bigg]+\mathbb{E}\bigg[\int_0^T\big\langle X_{1,h}^*(t)-\widetilde{X}_h(t), \mathcal{X}_{1,h}^0[V_{h\tau}](t)-\mathcal{X}_{1,h\tau}^0[V_{h\tau}](t)\big\rangle \,{\rm d}t\bigg]\\&\qquad-\mathbb{E}\bigg[\int_0^T\big\langle \widetilde{X}_h(t)-\widetilde{X}_{h\tau}(t), \mathcal{X}_{1,h\tau}^0[V_{h\tau}](t)\big\rangle \,{\rm d}t\bigg]\bigg\}\\&\qquad- \beta\bigg\{\mathbb{E}\big[\big\langle X_{1,h}^*(T)- X_{1,h\tau}^*(T), \mathcal{X}_{1,h\tau}^0[V_{h\tau}](T)\big\rangle\big]+\mathbb{E}\big[\big\langle X_{1,h}^*(T)-\widetilde{X}_h(T), \mathcal{X}_{1,h}^0[V_{h\tau}](T)-\mathcal{X}_{1,h\tau}^0[V_{h\tau}](T)\big\rangle \big]\\&\qquad-\mathbb{E}\big[\big\langle \widetilde{X}_h(T)-\widetilde{X}_{h\tau}(T), \mathcal{X}_{1,h\tau}^0[V_{h\tau}](T)\big\rangle\big]\bigg\},
	\end{align*}
    where inserting some intermediate terms are added and subtracted. In the above equality we take $V_{h\tau}=\Pi_\tau U_h^*-U_{h\tau}^*$ and use the facts $\mathcal{X}_{1,h}[U^*_h]-\mathcal{X}_{1,h}[U_{h\tau}^*]=\mathcal{X}_{1,h}^0[U_h^*-U_{h\tau}^*]$ and $\mathcal{X}_{1,h\tau}[U_{h\tau}^*]-\mathcal{X}_{1,h\tau}[\Pi_\tau U_h^*]=\mathcal{X}_{1,h\tau}^0[U_{h\tau}^*-\Pi_\tau U_h^*]$ to conclude that (by inserting some intermediate terms)
    $$-\mathbb{E}\bigg[\int_0^T\big\langle X_{1,h}^*(t)-  X_{1,h\tau}^*(t), \mathcal{X}_{1,h\tau}^0[V_{h\tau}](t)\big\rangle \,{\rm d}t\bigg]=\sum_{i=1}^3I_i,$$
    and 
    $$-\mathbb{E}\big[\big\langle X_{1,h}^*(T)- X_{1,h\tau}^*(T), \mathcal{X}_{1,h\tau}^0[V_{h\tau}](T)\big\rangle\big]=\sum_{i=1}^3I_i'$$
    Finally, we deduce that  
	\begin{align}\label{today 31}
		\mathbb{E}\bigg[\int_0^T\alpha\big\langle U_h^*(t)- U_{h\tau}^*(t), \Pi_\tau U_h^*(t)-U_{h\tau}^*(t)\big\rangle \,{\rm d}t\bigg]=:\sum_{i=1}^5 I_i +\sum_{i=1}^4 I_i',
	\end{align}
	where
	\begin{align*}
		I_1&=-\mathbb{E}\bigg[\int_0^T \big\langle X_{1,h}^*(t)-X_{1,h\tau}^*(t), \mathcal{X}^0_{1,h}[\Pi_\tau U^*_h- U_{h\tau}^*](t) \big\rangle \,{\rm d}t\bigg],\\
        I_2&=\mathbb{E}\bigg[\int_0^T \big\langle X^*_{1,h}(t)- X_{1,h\tau}^*(t),\, \mathcal{X}_{1,h}[\Pi_\tau U^*_h](t)- \mathcal{X}_{1,h\tau}[\Pi_{\tau} U_h^*](t)\big\rangle \,{\rm d}t\bigg],\\
		I_3&=-\mathbb{E}\bigg[\int_0^T\big\langle X_{1,h}^*(t)-X_{1,h\tau}^*(t), X_{1,h}^*(t)-X^*_{1,h\tau}(t)\big\rangle \,{\rm d}t\bigg],\\
        I_4&=\mathbb{E}\bigg[\int_0^T\big\langle \widetilde{X}_h(t)-\widetilde{X}_{1,h\tau}(t), \mathcal{X}_{1,h\tau}^0[\Pi_\tau U_h^*-U_{h\tau}^*](t)\big\rangle \,{\rm d}t\bigg],\\
		I_5&=-\mathbb{E}\bigg[\int_0^T\big\langle X^*_{1,h}(t)-\widetilde{X}_h(t), (\mathcal{X}^0_{1,h}-\mathcal{X}^0_{1,h\tau})[\Pi_\tau U_h^*- U_{h\tau}^*](t) \big\rangle \,{\rm d}t\bigg]\\&\qquad\qquad-\beta\mathbb{E}\big[\big\langle X^*_{1,h}(T)-\widetilde{X}_h(T), (\mathcal{X}^0_{1,h}-\mathcal{X}^0_{1,h\tau})[\Pi_\tau U_h^*- U_{h\tau}^*](T) \big\rangle\big],\\
        I_1'&=\beta\mathbb{E}\big[\big\langle X^*_{1,h}(T)- X_{1,h\tau}^*(T),\, \mathcal{X}_{1,h}[\Pi_\tau U^*_h](T)- \mathcal{X}_{1,h\tau}[\Pi_{\tau} U_h^*](T)\big\rangle \big],\\
        I_2'&=-\beta\mathbb{E}\big[\big\langle X_{1,h}^*(T)-X_{1,h\tau}^*(T), \mathcal{X}^0_{1,h}[\Pi_\tau U^*_h- U_h^*](T) \big\rangle \big],\\
        I_3'&=-\beta\mathbb{E}\big[\big\langle X_{1,h}^*(T)-X_{1,h\tau}^*(T), X_{1,h}^*(T)-X^*_{1,h\tau}(T)\big\rangle\big],\\
		I_4'&=\beta\mathbb{E}\bigg[\big\langle \widetilde{X}_h(T)-\widetilde{X}_{h\tau}(T), \mathcal{X}_{1,h\tau}^0[\Pi_\tau U_h^*-U_{h\tau}](T)\big\rangle \bigg].
	\end{align*}
We will estimate these terms separately in the following sub-steps:

	\noindent
	\textbf{Step 1(a).} For term $I_1$, we can conclude that
	\begin{align*}
		I_1&\le\,\frac{1}{4}\mathbb{E}\big[\|X^*_{1,h}- X_{1,h\tau}^*\|_{\mathbb{L}^2_{t, x}}^2\big]+ C\,\mathbb{E}\big[\| \mathcal{X}^0_{1,h}[\Pi_\tau U^*_h- U_h^*]\|_{\mathbb{L}^2_{t, x}}^2\big]\\&\le\,\frac{1}{4}\mathbb{E}\big[\| X^*_{1,h}-X_{1,h\tau}^*\|_{\mathbb{L}^2_{t, x}}^2\big]+ C\,\mathbb{E}\big[\|\Pi_\tau U_h^*- U_h^*\|_{\mathbb{L}^2_{t, x}}^2\big]\\&\le\,\frac{1}{4}\mathbb{E}\big[\| X^*_{1,h} - X_{1,h\tau}^*\|_{\mathbb{L}^2_{t, x}}^2\big]+ C\,\tau\big(\|X_{1,0}\|_{\mathbb{H}_x^3}^2+ \|X_{2,0}\|_{\mathbb{H}^2_x}^2+\|\widetilde{X}\|_{C_t\mathbb{H}_x^2}^2+\mathbb{E}\big[ \|\sigma\|_{\mathbb{L}^2_t\mathbb{H}^2_x}^2\big]\big),
	\end{align*}
	where in the second inequality the estimate \eqref{today0002} is used, while in the last inequality the estimate \eqref{time-regularity for semi-discrete control} is used. Similarly, we obtain
	\begin{align*}
		I_1'\le \frac{\beta}{4}\mathbb{E}\big[\|X_{1,h}^*(T)-X_{1,h\tau}^*(T)\|_{\mathbb{L}^2_x}^2\big]+ C\tau\big(\|X_{1,0}\|_{\mathbb{H}_x^3}^2+ \|X_{2,0}\|_{\mathbb{H}^2_x}^2+\|\widetilde{X}\|_{C_t\mathbb{H}_x^2}^2+\mathbb{E}\big[ \|\sigma\|_{\mathbb{L}^2_t\mathbb{H}^2_x}^2\big]\big).
	\end{align*}
	\noindent
	\textbf{Step 1(b).} For term $I_2$, we use the estimates \eqref{today0004} and \eqref{434343} to conclude that
	\begin{align*}
		I_2&\le\,\frac{1}{4}\mathbb{E}\big[\|X^*_{1,h}- X_{1,h\tau}^*\|_{\mathbb{L}^2_{t, x}}^2\big] +C\,\mathbb{E}\big[\| \mathcal{X}_{1,h}[\Pi_\tau U^*_h]- \mathcal{X}_{1,h\tau}[\Pi_{\tau} U_h^*]\|_{\mathbb{L}^2_{t, x}}^2\big]\\&\le\frac{1}{4}\mathbb{E}\big[\| X^*_{1,h}- X_{1,h\tau}^*\|_{\mathbb{L}^2_{t, x}}^2\big] + C\tau\big(\|X_{1,0}\|_{\mathbb{H}_x^3}^2+ \|X_{2,0}\|_{\mathbb{H}^2_x}^2+\|\widetilde{X}\|_{C_t\mathbb{H}_x^2\cap C^{1/2}_t\mathbb{H}_0^1}^2+\mathbb{E}\big[ \|\sigma\|_{\mathbb{L}^2_t\mathbb{H}^2_x\cap C^{1/2}_t\mathbb{H}_0^1}^2\big]\big).
	\end{align*}
Similarly, we obtain
\begin{align*}
	I_2'\le  \frac{\beta}{4}\mathbb{E}\big[\|X_h^*(T)-X_{1,h\tau}^*(T)\|_{\mathbb{L}^2_x}^2\big]+ C\tau\big(\|X_{1,0}\|_{\mathbb{H}_x^3}^2+ \|X_{2,0}\|_{\mathbb{H}^2_x}^2+\|\widetilde{X}\|_{C_t\mathbb{H}_x^2\cap C^{1/2}_t\mathbb{H}_0^1}^2+\mathbb{E}\big[ \|\sigma\|_{\mathbb{L}^2_t\mathbb{H}^2_x\cap C^{1/2}_t\mathbb{H}_0^1}^2\big]\big).
\end{align*}

%

\noindent
\textbf{Step 1(f).}  As previous sub-steps, by using \eqref{assumption on noise coefficitents} and \eqref{today0003}, we conclude that
\begin{align*}
	I_4 + I_4'\le C\tau\big(\|X_{1,0}\|_{\mathbb{H}_x^3}^2+ \|X_{2,0}\|_{\mathbb{H}^2_x}^2+\|\widetilde{X}\|_{C_t\mathbb{H}_x^2\cap C^{1/2}_t\mathbb{H}_0^1}^2+\mathbb{E}\big[ \|\sigma\|_{\mathbb{L}^2_t\mathbb{H}^2_x\cap C^{1/2}_t\mathbb{H}_0^1}^2\big]\big)+ \frac{\alpha}{4}\mathbb{E}\big[\| U_h^*-U_{h\tau}^*\|_{\mathbb{L}^2_{t, x}}^2\big].
\end{align*}
	\noindent
	\textbf{Step 1(d).} For the term $I_5$, as an application of It\^o formula as done in the proof of Pontryagin's maximum principle (see identity \eqref{today000}), we get 
\begin{align*}
	&\mathbb{E}\bigg[\int_0^T\big\langle X^*_{1,h}(t)-\widetilde{X}_h(t),\mathcal{X}^0_{1,h}[\Pi_\tau U_h^*- U_{h\tau}^*](t) \big\rangle \,{\rm d}t\bigg]+\beta\mathbb{E}\big[\big\langle X^*_{1,h}(T)-\widetilde{X}_h(T), \mathcal{X}^0_{1,h}[\Pi_\tau U_h^*- U_{h\tau}^*](T) \big\rangle \big]\\&=\mathbb{E}\bigg[\int_0^T\left\langle Y_{2,h},\Pi_\tau U_h^*- U_{h\tau}^*\right\rangle\,{\rm d}t\bigg].
\end{align*}
From Proposition~\ref{Proposition02}, we get
	\begin{align*}
		&\mathbb{E}\bigg[\int_{0}^T\left\langle X^*_{1,h}(t)-\widetilde{X}_{h}(t), \mathcal{X}_{1,h\tau}^0[\Pi_\tau U_h^*-U_{h\tau}^*](t)\right\rangle\,{\rm d}t\bigg]+\beta \mathbb{E}\big[\left\langle X^*_{1,h}(T)-\widetilde{X}_{h}(T), \mathcal{X}_{1,h\tau}^0[\Pi_\tau U_h^*-U_{h\tau}^*](T)\right\rangle\big]\\&=I_{11}+I_{12}+ I_{13}+I_{14},
	\end{align*}
	where
	\begin{align*}
		I_{11}&=\frac{\tau}{2}\sum_{n=0}^{N-1}\mathbb{E}\bigg[\left\langle \big(\mathcal{X}_{2,h\tau}^0[V_{h\tau}](t_{n+1})+\mathcal{X}_{2,h\tau}^0[V_{h\tau}](t_{n})\big), Y_{1,h}(t_{n+1})\right\rangle\bigg]\\&\qquad-\sum_{n=0}^{N-1}\mathbb{E}\bigg[\int_{t_{n}}^{t_{n+1}}\left\langle Y_{1,h}(t), \mathcal{X}_{2,h\tau}^0[V_{h\tau}](t_{n})\right\rangle\bigg],\\
		I_{12}&=\sum_{n=0}^{N-1}\mathbb{E}\bigg[\int_{t_{n}}^{t_{n+1}}\left\langle \nabla Y_{2,h}, \nabla \mathcal{X}_{1,h\tau}^0[V_{h\tau}](t_{n})\right\rangle\,{\rm d}t\bigg]\\&\qquad\qquad-\frac{\tau}{2}\sum_{n=0}^{N-1}\mathbb{E}\bigg[\left\langle \nabla\big( \mathcal{X}_{1,h\tau}^0[V_{h\tau}](t_{n+1})+\mathcal{X}_{1,h\tau}^0[V_{h\tau}](t_{n})\big), \nabla Y_{2,h}(t_{n+1})\right\rangle\bigg],\\
		I_{13}&=\tau\sum_{n=0}^{N-1}\mathbb{E}\bigg[\left\langle V_{h\tau}(t_{n}), Y_{2,h}(t_{n+1})\right\rangle\bigg],\\
		I_{14}&=-\sum_{n=0}^{N-1}\mathbb{E}\bigg[\left\langle \int_{t_{n+1}}^{t_{n}}Y_{1}(t)\,{\rm d}t,\gamma \mathcal{X}_{1,h\tau}^0[V_{h\tau}](t_n)\Delta_{n+1}W\right\rangle\bigg],\\
		V_{h\tau}&=\Pi_\tau U_h^*-U_{h\tau}^*.
	\end{align*}
It implies that
\begin{align*}
	I_5= I_{11}+ I_{12}+ I_{13}'+ I_{14},
\end{align*}
where
\begin{align*}
	I_{13}'=\sum_{n=0}^{N-1}\mathbb{E}\bigg[\int_{t_{n}}^{t_{n+1}}\left\langle V_{h\tau}(t_n), Y_{2,h}(t_{n+1})-Y_{2,h}(t)\right\rangle\,{\rm d}t\bigg].
\end{align*}
We estimate each $ I$'s terms separately.

\noindent
\textbf{Step 1(d)(a).} In this step, we estimate the term $I_{12}$ as follows:
\begin{align}\label{d1}
	I_{12}=&-\frac{1}{2}\sum_{n=0}^{N-1}\mathbb{E}\bigg[\int_{t_{n}}^{t_{n+1}}\left\langle \nabla\big(\mathcal{X}_{1,h\tau}^0[V_{h\tau}](t_{n+1})+\mathcal{X}_{1,h\tau}^0[V_{h\tau}](t_{n})\big), \nabla Y_{2,h}(t_{n+1})-\nabla Y_{2,h}(t)\right\rangle\,{\rm d}t\bigg]\notag\\&+\frac{1}{2}\sum_{n=0}^{N-1}\mathbb{E}\bigg[\int_{t_{n}}^{t_{n+1}}\left\langle \nabla Y_{2,h}(t), \nabla \mathcal{X}_{1,h\tau}^0[V_{h\tau}](t_n)-\nabla \mathcal{X}_{1,h\tau}^0[V_{h\tau}](t_{n+1})\right\rangle\,{\rm d}t\bigg],
\end{align}
By using the discrete integration by parts formula, the identity~\eqref{jhs} and facts $Y_{2,h}(t_N)=X_{1,h\tau}^0(0)=0$, we obtain
\begin{align*}
	&\sum_{n=0}^{N-1}\mathbb{E}\bigg[\int_{t_{n}}^{t_{n+1}}\left\langle \nabla Y_{2,h}(t),\nabla \mathcal{X}_{1,h\tau}^0[V_{h\tau}](t_n)-\nabla \mathcal{X}^0_{1,h\tau}[V_{h\tau}](t_{n+1})\right\rangle\,{\rm d}t\bigg]\\&=\sum_{n=0}^{N-1}\mathbb{E}\bigg[\int_{t_{n}}^{t_{n+1}}\left\langle (\nabla\hat{Y}_{2,h}(t_{n+1})-\nabla \hat{Y}_{2,h}(t_{n})),\nabla \mathcal{X}^0_{1,h\tau}[V_{h\tau}](t_n)\right\rangle\,{\rm d}t\bigg].
	\end{align*}
We use of Young's inequality, \eqref{today0002}, \eqref{time-regularity for semi-discrete control} and \eqref{today48} to conclude that 
\begin{align}
	&\sum_{n=0}^{N-1}\mathbb{E}\bigg[\int_{t_{n}}^{t_{n+1}}\left\langle \nabla Y_{2,h}(t),\nabla \mathcal{X}_{1,h\tau}^0[V_{h\tau}](t_n)-\nabla \mathcal{X}^0_{1,h\tau}[V_{h\tau}](t_{n+1})\right\rangle\,{\rm d}t\bigg]\notag\\
	&\le C_\delta\tau\sum_{n=0}^{N-1}\mathbb{E}\bigg[\|\nabla\hat{Y}_{2,h}(t_{n+1})-\nabla\hat{Y}_{2,h}(t_n)\|_{\mathbb{L}^2_x}^2\bigg]+\tau \delta \sum_{n=0}^{N-1}\mathbb{E}\big[\|\nabla \mathcal{X}^0_{1,h\tau}[V_{h\tau}](t_n)\|_{\mathbb{L}^2_x}^2\big]\notag\\
	&\le C_\delta\tau(\|X_{1,0}\|_{\mathbb{H}_0^1}^2+ \|X_{2,0}\|_{\mathbb{L}^2_x}^2+\|\widetilde{X}\|_{C_t\mathbb{H}_0^1}^2+ \mathbb{E}\big[\|\sigma\|_{\mathbb{L}^2_{t,x}}^2\big])+ \delta \mathbb{E}\big[\|\Pi_\tau U_h^*-U_{h\tau}^*\|_{\mathbb{L}^2_{t, x}}^2\big]\notag\\
	&\le C_{\delta}\tau \big(\|X_{2,0}\|_{\mathbb{H}^1_0}^2+\|X_{1,0}\|_{\mathbb{H}^2_x}^2+\|\widetilde{X}\|_{C_t\mathbb{H}_0^1}^2+\mathbb{E}\big[\|\sigma\|_{\mathbb{L}^2_{t}\mathbb{H}^1_0}^2 \big]\big)+ \delta\mathbb{E}\big[\|U_h^*-U_{h\tau}^*\|_{\mathbb{L}_{t, x}^2}^2\big].\label{d4}
\end{align}
Similarly, we can use the estimates \eqref{today03}-\eqref{today47} to obtain
\begin{align}\label{d2}
&\frac{1}{2}\sum_{n=0}^{N-1}\mathbb{E}\bigg[\int_{t_{n}}^{t_{n+1}}\left\langle \nabla \big(\mathcal{X}^0_{1,h\tau}[V_{h\tau}](t_{n+1})+\mathcal{X}^0_{1,h\tau}[V_{h\tau}](t_{n})), \nabla Y_{2,h}(t_{n+1})-\nabla Y_{2,h}(t)\right\rangle\,{\rm d}t\bigg]\notag\\&\le C_\delta \tau\big(\|X_{2,0}\|_{\mathbb{H}^1_0}^2+\|X_{1,0}\|_{\mathbb{H}^2_x}^2+\|\widetilde{X}\|_{C_t\mathbb{H}_0^1}^2+\mathbb{E}\big[\|\sigma\|_{\mathbb{L}^2_{t}\mathbb{H}^1_0}^2 \big]\big) +\delta \mathbb{E}\big[\|U_h^*-U_{h\tau}^*\|_{\mathbb{L}^2_x}^2\big].
\end{align}
From the previous estimates \eqref{d1}-\eqref{d2}, we conclude that
\begin{align}\label{f1}
	I_{12}\le C_\delta \tau\big(\|X_{2,0}\|_{\mathbb{H}^1_0}^2+\|X_{1,0}\|_{\mathbb{H}^2_x}^2+\|\widetilde{X}\|_{C_t\mathbb{H}_0^1}^2+\mathbb{E}\big[\|\sigma\|_{\mathbb{L}^2_{t}\mathbb{H}^1_0}^2 \big]\big) +\delta \mathbb{E}\big[\|U_h^*-U_{h\tau}^*\|_{\mathbb{L}^2_x}^2\big].
\end{align}
\textbf{Step 1(d)(b).}
To estimate the term $I_{11}$, we can follows similar lines as used to estimate the term $I_{12}$. The term $I_{13}'$ can be easily estimate by using Young's inequality. For terms $I_{11}$ and $I_{13}'$, we can conclude that
\begin{align}\label{f2}
	|I_{11}|+|I_{13}'|\le C_{\delta}\tau\big(\|X_{2,0}\|_{\mathbb{H}^1_0}^2+\|X_{1,0}\|_{\mathbb{H}^2_x}^2+\|\widetilde{X}\|_{C_t\mathbb{H}_0^1}^2+\mathbb{E}\big[\|\sigma\|_{\mathbb{L}^2_{t}\mathbb{H}^1_0}^2 \big]\big)+ \delta \mathbb{E}\big[\|U_h^*-U_{h\tau}^*\|_{\mathbb{L}^2_x}^2\big].
\end{align}
\textbf{Step 1(d)(c).}
For term $I_{14}$, we obtain
\begin{align*}
	|I_{14}|\le C_\delta\sum_{n=0}^{N-1}\mathbb{E}\bigg[\bigg\|\int_{t_{n}}^{t_{n+1}}Y_{1,h}(t)\,{\rm d}t\bigg\|_{\mathbb{L}^2_x}^2\bigg]+ \delta\sum_{n=0}^{N-1}\mathbb{E}\bigg[\bigg\|\mathcal{X}^0_{1,h\tau}[V_{h\tau}](t_n)\Delta_{n+1}W\bigg\|_{\mathbb{L}^2_x}^2\bigg].
\end{align*} By using H\"older's inequality and It\^o isometry we yield
\begin{align}\label{f3}
	|I_{14}|&\le C_\delta \tau \sum_{n=0}^{N-1}\mathbb{E}\bigg[\int_{t_{n}}^{t_{n+1}}\|Y_{1,h}(t) \|_{\mathbb{L}^2_x}^2\,{\rm d}t\bigg] + \delta \tau \sum_{n=0}^{N-1}\mathbb{E}\bigg[\|\mathcal{X}^0_{1,h\tau}[V_{h\tau}](t_n) \|_{\mathbb{L}^2_x}^2\bigg]\notag\\
	&\le C_{\delta} \tau\big(\|X_{1,0}\|_{\mathbb{H}_0^1}^2+ \|X_{2,0}\|_{\mathbb{L}^2_x}^2+\|\widetilde{X}\|_{C_t\mathbb{L}_x^2}^2+\mathbb{E}\big[ \|\sigma\|_{\mathbb{L}^2_t\mathbb{L}^2}^2\big]\big) + \delta  \mathbb{E}\big[\|V_{h\tau}\|_{\mathbb{L}^2_{t, x}}^2\big]\notag\\&\le\,C_{\delta} \tau\big(\|X_{1,0}\|_{\mathbb{H}_0^1}^2+ \|X_{2,0}\|_{\mathbb{L}^2_x}^2+\|\widetilde{X}\|_{C_t\mathbb{L}_x^2}^2+\mathbb{E}\big[ \|\sigma\|_{\mathbb{L}^2_t\mathbb{L}^2}^2\big]\big) + \delta  \mathbb{E}\big[\|U_h^*-U_{h\tau}^*\|_{\mathbb{L}^2_{t, x}}^2\big]+ \delta  \mathbb{E}\big[\|U_h^*-\Pi_hU_{h}^*\|_{\mathbb{L}^2_{t, x}}^2\big]\notag\\&\le C_{\delta}\tau\big(\|X_{2,0}\|_{\mathbb{H}^1_0}^2+\|X_{1,0}\|_{\mathbb{H}^2_x}^2+\|\widetilde{X}\|_{C_t\mathbb{H}_0^1}^2+\mathbb{E}\big[\|\sigma\|_{\mathbb{L}^2_{t}\mathbb{H}^1_0}^2 \big]\big) + \delta \mathbb{E}\big[\|U_h^*-U_{h\tau}^*\|_{\mathbb{L}^2_{t, x}}^2\big],
\end{align}
where \eqref{today0003}, \eqref{estimate for discrete adjoint 1} and \eqref{time-regularity for semi-discrete control} are used.
Finally for the term $I_4$, we use \eqref{f1}-\eqref{f3} to obtain
\begin{align*}
	|I_5|\le C_\delta \tau \big(\|X_{2,0}\|_{\mathbb{H}^1_0}^2+\|X_{1,0}\|_{\mathbb{H}^2_x}^2+\|\widetilde{X}\|_{C_t\mathbb{H}_0^1}^2+\mathbb{E}\big[\|\sigma\|_{\mathbb{L}^2_{t}\mathbb{H}^1_0}^2 \big]\big) + \delta \mathbb{E}\big[\|U_h^*-U_{h\tau}^*\|_{\mathbb{L}^2_{t, x}}^2\big].
\end{align*}
Finally, with the estimates from terms $I's$ and \eqref{today 31}, we conclude that there exists a positive constant $C$ such that
	\begin{align}\label{today 32}
		\mathbb{E}\bigg[\int_0^T\alpha\big\langle U_h^*(t)&- U_{h\tau}^*(t), \Pi_\tau U_h^*(t)-U_{h\tau}^*(t)\big\rangle \,{\rm d}t\bigg]+\mathbb{E}\big[\|X^*_{1,h}- X_{1,h\tau}^*\|_{\mathbb{L}^2_{t, x}}^2\big]+\beta\mathbb{E}\big[\|X^*_{1,h}(T)- X_{1,h\tau}^*(T)\|_{\mathbb{L}^2_{x}}^2\big]\notag\\&\le\,C\,\tau\big(\|X_{1,0}\|_{\mathbb{H}^3_x}^2+ \|X_{2,0}\|_{\mathbb{H}^2_x}^2+\|\widetilde{X}\|_{C_t\mathbb{H}_x^2\cap C^{1/2}_t\mathbb{H}_0^1}^2+\mathbb{E}\big[ \|\sigma\|_{\mathbb{L}^2_t\mathbb{H}^2_x\cap C^{1/2}_t\mathbb{H}_0^1}^2\big]\big)+ \frac{\alpha}{4}\mathbb{E}\big[\| U_h^*-U_{h\tau}^*\|_{\mathbb{L}^2_{t, x}}^2\big].
	\end{align}

\noindent
\textbf{Step 2.} We have the following identity
	\begin{align*}
		\alpha\mathbb{E}\big[\|U_h^*-U_{h\tau}^*\|_{\mathbb{L}^2_{t, x}}^2\big]&=J_1 + J_2,
	\end{align*}
	where 
	\begin{align*}
		J_1&=\mathbb{E}\bigg[\int_0^T\alpha\big\langle U_h^*(t)- U_{h\tau}^*(t), \Pi_\tau U_h^*-U_{h\tau}^*(t)\big\rangle \,{\rm d}t\bigg],\\
		J_2&=\mathbb{E}\bigg[\int_0^T\alpha\big\langle U_h^*(t)- U_{h\tau}^*(t), U_h^*-\Pi_{\tau} U_{h}^*(t)\big\rangle \,{\rm d}t\bigg].
	\end{align*}
	For the term $J_1$, from \eqref{today 32}, we conclude that there exists a positive constant $C$ such that
	\begin{align}
		J_1-I_3-I_3'\le\,C\,\tau\big(\|X_{1,0}\|_{\mathbb{H}_x^3}^2+ \|X_{2,0}\|_{\mathbb{H}^2_x}^2+\|\widetilde{X}\|_{C_t\mathbb{H}_x^2\cap C^{1/2}_t\mathbb{H}_0^1}^2+\mathbb{E}\big[ \|\sigma\|_{\mathbb{L}^2_t\mathbb{H}^2_x\cap C^{1/2}_t\mathbb{H}_0^1}^2\big]\big) + \frac{\alpha}{4}\mathbb{E}\big[\| U_h^*-U_{h\tau}^*\|_{\mathbb{L}^2_{t, x}}^2\big]\label{j3}.
	\end{align}
	For the term $J_2$, we obtain by using the estimate \eqref{time-regularity for semi-discrete control} that there exists a positive constant $C$ such that 
	\begin{align}
		J_2\,&\le\,\frac{\alpha}{4}\mathbb{E}\big[\|U_h^*- U_{h\tau}^*\|_{\mathbb{L}^2_{t, x}}^2\big] + C\,\mathbb{E}\big[\|U_h^*-\Pi_{\tau} U_{h}^*\|_{\mathbb{L}^2_{t, x}}^2\big]\notag\\&\le\,\frac{\alpha}{4}\mathbb{E}\big[\|U_h^*- U_{h\tau}^*\|_{\mathbb{L}^2_{t, x}}^2\big] + C\,\tau\big(\|X_{1,0}\|_{\mathbb{H}_x^3}^2+ \|X_{2,0}\|_{\mathbb{H}^2_x}^2+\|\widetilde{X}\|_{C_t\mathbb{H}_x^2}^2+\mathbb{E}\big[ \|\sigma\|_{\mathbb{L}^2_t\mathbb{H}^2_x}^2\big]\big)\label{j4}.
	\end{align}
	From the estimates \eqref{j3}-\eqref{j4}, we conclude that there exists a positive constant $C$ such that
	\begin{align*}
		&\mathbb{E}\big[\|U_h^*-U_{h\tau}^*\|_{\mathbb{L}^2_{t, x}}^2\big]+\mathbb{E}\big[\|X^*_{1,h}- X_{1,h\tau}^*\|_{\mathbb{L}_{t, x}^2}^2\big]+\beta \mathbb{E}\big[\|X^*_{1,h}(T)- X_{1,h\tau}^*(T)\|_{\mathbb{L}_{x}^2}^2\big]\\&\le\,C\,\tau\big(\|X_{1,0}\|_{\mathbb{H}_x^3}^2+ \|X_{2,0}\|_{\mathbb{H}^2_x}^2+\|\widetilde{X}\|_{C_t\mathbb{H}_x^2\cap C^{1/2}_t\mathbb{H}_0^1}^2+\mathbb{E}\big[ \|\sigma\|_{\mathbb{L}^2_t\mathbb{H}^2_x\cap C^{1/2}_t\mathbb{H}_0^1}^2\big]\big).
	\end{align*}
	This completes the proof.
\end{proof}
\begin{remark}[An important point]\label{Remark 4.5}
	For the {\bf SLQ} problem with stochastic heat equation the methods in \cite{Prohl&Wang1, ProhlWang2021, B.Li} require time discretization of the $\textbf{\textbf{BSDE}}$ and employ different techniques to estimate error terms due to this discretization of \textbf{\textbf{BSDE}}; see the proof of \cite[Theorem 3.3, pg. 3422]{ProhlWang2021} and \cite[Section 4.3]{B.Li}. However, in our approach, time discretization of $\textbf{\textbf{BSDE}}_h$ \eqref{discrete adjoint ode} is not required for the error analysis.
\end{remark}
\begin{thm}\label{strong rate of convergence for time discretization1} Let Assumption~\ref{BB} hold. Let $(X_{1,h}^*, X_{2,h}^*, U_h^*)$ and $(X_{1,h\tau}^*, X_{2,h\tau}^*, U_{h\tau}^*)$ be solve {\bf SLQ}$_h$ problem \eqref{3.1}-\eqref{3.2} and ${\bf SLQ}_{h\tau}$ problem \eqref{fully discrete cost functional}-\eqref{discrete state equation}, respectively. Then there exists a positive constant $C$ such that for all $t\in[0,T]$,
	\begin{align*}
		&\mathbb{E}\big[\|U_h^*-U_{h\tau}^*\|_{\mathbb{L}^2_{t, x}}^2\big]+\mathbb{E}\big[\|\nabla(X^*_{1,h}(t)- X_{1,h\tau}^*(t))\|_{\mathbb{L}^2_{t, x}}^2\big]+\mathbb{E}\big[\|X_{2,h}^*(t)- X_{2,h\tau}^*(t)\|_{\mathbb{L}^2_x}^2\big] \\&\le C\tau\big(\|X_{1,0}\|_{\mathbb{H}_x^3}^2+ \|X_{2,0}\|_{\mathbb{H}^2_x}^2+\|\widetilde{X}\|_{C_t\mathbb{H}_x^2\cap C^{1/2}_t\mathbb{H}_0^1}^2+\mathbb{E}\big[ \|\sigma\|_{\mathbb{L}^2_t\mathbb{H}^2_x\cap C^{1/2}_t\mathbb{H}_0^1}^2\big]\big).
	\end{align*} 	
\end{thm}
\begin{proof}
For the proof, one can follow similar lines as in the proof of Proposition \ref{Proposition 4.7}. It is a consequence of the error bound on the additional term $\mathbb{E}[\|U_h^* - U_{h\tau}^*\|_{\mathbb{L}^2_{t, x}}^2]$, which is established in Theorem \ref{strong rate of convergence for time discretization}.    
\end{proof}
\begin{remark}[Rate of convergence]
	In the proof of Theorem \ref{strong rate of convergence for time discretization}, One needs error bound on \(\mathbb{E}\big[\|U_h^* - \Pi_\tau U_h^*\|_{\mathbb{L}_x^2}^2\big]^{1/2}\), but estimating \(\mathbb{E}\big[\|U_h^* - \Pi_\tau U_h^*\|_{\mathbb{L}_x^2}^2\big]^{1/2}\) relies on the time regularity of \(U_h^* = -\frac{1}{\alpha} Y_{2,h}\) (see Lemma~\ref{time-regularity for semi-discrete control}). As \(Y_{2,h}\), a solution component of the ${\bf BSPDE}_h$ \eqref{discrete adjoint ode}, has Hölder continuity up to \(1/2\), the convergence rate in Theorem \ref{strong rate of convergence for time discretization} is limited to order \(1/2\) (see Proposition~\ref{Time-regularity for semi-discrete optimal control}). Thus, improving this rate is challenging.
\end{remark}
\subsection{Main result of the error analysis for space-time discretization}
The following theorem gives the main result of this section, establishing the rate of convergence in the energy norm.
\begin{thm}[Final result of this section]\label{final result for full discrete scheme} Let Assumption~\ref{BB} hold. Let $(X^*_1, X_2^*, U^*)$ and $(X_{1,h\tau}^*, X_{2,h\tau}^*, U_{h\tau}^*)$ be solve {\bf SLQ} problem \eqref{1.3}-\eqref{1.4} and ${\bf SLQ}_{h\tau}$ problem \eqref{fully discrete cost functional}-\eqref{discrete state equation}, respectively. Then there exists a positive constant $C$ such that
	\begin{align*}
		&\mathbb{E}\big[\|U^*-U_{h\tau}^*\|_{\mathbb{L}^2_{t, x}}^2\big]+\sup_{t\in[0,T]}\bigg[\mathbb{E}\big[\|\nabla (X_1^*(t)- X_{1,h\tau}^*(t))\|_{\mathbb{L}^2_{x}}^2\big]+\mathbb{E}\big[\|X_2^*(t)- X_{2,h\tau}^*(t)\|_{\mathbb{L}^2_{x}}^2\big]\bigg]\\&\le\,C\ (\tau +h^2)\big(\|X_{1,0}\|_{\mathbb{H}_x^3}^2+ \|X_{2,0}\|_{\mathbb{H}^2_x}^2+\|\widetilde{X}\|_{C_t\mathbb{H}_x^2\cap C^{1/2}_t\mathbb{H}_0^1}^2+\mathbb{E}\big[ \|\sigma\|_{\mathbb{L}^2_t\mathbb{H}^2_x\cap C^{1/2}_t\mathbb{H}_0^1}^2\big]\big).
	\end{align*} 
\end{thm}
\begin{proof}
	This is a combined result of Theorems~\ref{thm3.5} and \ref{strong rate of convergence for time discretization1}.
	\end{proof}
\section{Fully discrete Pontryagin's Maximum Principle and gradient descent method}\label{Section 5} The fully discrete optimal tuple $(X_{1,h\tau}^*, X_{2,h\tau}^*,  U_{h\tau}^*)$ for the ${\bf SLQ}_{h\tau}$ problem \eqref{fully discrete cost functional}-\eqref{discrete state equation} exists but lacks an explicit, {\em implementable } form. Thus, we need to apply the fully discrete Pontryagin's maximum principle (see Proposition~\ref{fully discrete Pontryagin's maximum principle} below) to characterize it via a decoupled forward-backward system and an optimality condition for a practical implementation purpose. Hence, this section discusses the fully discrete Pontryagin's maximum principle. 
\subsection{Discrete Pontryagin's maximum principle} 
Let $U_{h\tau}\in \mathbb{U}_{h\tau}$. Then let the pair $(Y_{1,h\tau}, Y_{2,h\tau})\in \mathbb{X}_{h\tau}\times\mathbb{X}_{h\tau}$ solve the following backward difference equations: for all $n=N-1,...,0$,
\begin{align}\label{backward spde}
\begin{cases}
    Y_{1,h\tau}(t_{n})=\mathbb{E}\biggl[Y_{1,h\tau}(t_{n+1})+\frac{\tau}{2}\Delta_h\bigl[Y_{2,h\tau}(t_{n+1})+Y_{2,h\tau}(t_n)\bigr]+Y_{2,h\tau}(t_{n+1})\gamma\cdot\Delta_{n+1}W\bigg|\mathcal{F}_{t_n}\biggr]\\ \qquad\qquad\qquad+\tau\bigl(\widetilde{X}_h(t_{n})-\mathcal{X}_{1,h\tau}[U_{h\tau}](t_{n})\bigr),\\
    Y_{2,h\tau}(t_{n})=\mathbb{E}\biggl[Y_{2, h\tau}(t_{n+1})+\frac{\tau}{2}\bigl[Y_{1,h\tau}(t_{n+1})+Y_{1,h\tau}(t_n)\bigr]\bigg|\mathcal{F}_{t_n}\biggr],\\
    Y_{1,h\tau}(t_{N})=\frac{\tau}{2}\Delta_h Y_{2,h\tau}(t_N)+\beta(\widetilde{X}_{h\tau}(t_N)-X_{1,h\tau}(t_N)),\\
    Y_{2,h\tau}(t_N)=\frac{\tau}{2} Y_{1,h\tau}(t_N),
\end{cases}
\end{align}
For $i=1,2$, we define the operator $\mathcal{Y}_{i,h\tau}:\mathbb{U}_{h\tau}\to X_{h\tau}$ such that
\begin{align*}
    (\mathcal{Y}_{1,h\tau}[U_{h\tau}],\mathcal{Y}_{2,h\tau}[U_{h\tau}])=(Y_{1,h\tau}, Y_{2,h\tau})\in \mathbb{X}_{h\tau}\times\mathbb{X}_{h\tau},
\end{align*}
solve \eqref{backward spde}.
\begin{proposition}[Discrete Pontryagin's maximum principle]\label{fully discrete Pontryagin's maximum principle}
Let Assumption~\ref{CC} hold. The unique optimal tuple $(X_{1, h\tau}^*, X_{2,h\tau}^*, U_{h\tau}^*)\in [\mathbb{X}_{h\tau}]^2\times \mathbb{U}_{h\tau}$ to ${\bf SLQ}_{h\tau}$ problem \eqref{fully discrete cost functional}-\eqref{discrete state equation} if only if there exists the quadruple $(X^*_{1,h\tau}, X_{2,h\tau}^*, U_{h\tau}^*, Y_{2,h\tau})$ which satisfies the following conditions: 
\begin{itemize}
    \item [1.] \textbf{Forward state}: $(X^*_{1,h\tau}, X_{2,h\tau}^*)=(\mathcal{X}_1[U_{h\tau}^*],\mathcal{X}_2[U_{h\tau}^*]),$ 
    \item [2.]  {\textbf{Backward state:}} $(Y_{1,h\tau}, Y_{2,h\tau})=(\mathcal{Y}_{1,h\tau}[U_{h\tau}^*],\mathcal{Y}_{2,h\tau}[U_{h\tau}^*]),$
    \item [3.] {\textbf{Optimality condition:}} $\alpha U^*_{h\tau}(t_n)=\mathbb{E}\big[Y_{2,h\tau}(t_{n+1})\big|\mathcal{F}_{t_n}\big],$ for all $n=0,1...,N-1$.
\end{itemize}
\end{proposition}
\begin{proof}
For the proof, one can easily drive this discrete optimality system by defining discrete Lagrangian. For more details we refer to the proof of \cite[Prop.~2.1]{Bakan}.
\end{proof}
Note that items $1$ and $2$ in Proposition~\ref{fully discrete Pontryagin's maximum principle} are now decoupled: the first step requires solving a space-time discretization of {\bf SPDE} \eqref{1.1}, while the second requires solving the space-time discretization of the {\bf BSPDE} \eqref{1.5}. 
\begin{remark}[Frech\'et derivation of the fully discrete reduced cost functional] From the proof of Proposition~\ref{fully discrete Pontryagin's maximum principle}, one can easily conclude that for all $U_{h\tau}\in\mathbb{U}_{h\tau}$, for all $n=0,1..,N-1,$
	\begin{align}\label{Frachet derivatie}
		\mathcal{D}_{U}\hat{\mathcal{J}}_{h\tau}(U_{h\tau})(t_n):=-\mathbb{E}\big[\mathcal{Y}_{2,h\tau}[U_{h\tau}](t_{n+1})\big|\mathcal{F}_{t_n}\big]+\alpha U_{h\tau}.
	\end{align}
\end{remark}

\subsection{Gradient descent method}
By Proposition \ref{fully discrete Pontryagin's maximum principle}, solving the minimization ${\bf SLQ}_{h\tau}$ problem \eqref{fully discrete cost functional}-\eqref{discrete state equation} is equivalent to solving the system of {\em coupled} forward-backward difference equations with the optimality condition. By using the explicit expression of $\mathcal{D}_{U}\hat{\mathcal{J}}_{h\tau}$ from \eqref{Frachet derivatie}, we may exploit the variational character of ${\bf SLQ}_{h\tau}$ problem \eqref{fully discrete cost functional}-\eqref{discrete state equation} to construct a gradient descent method (for short, {\em i. e.}, ${\bf SLQ}_{h\tau}^{{\rm grad}}$) where approximate iterates of the optimal control $U^*_{h\tau}$ in the Hilbert space $\mathbb{U}_{h\tau}$ are obtained.  A similar approach has been chosen in \cite{ChaudharyProhl,ProhlWang2021, Prohl&Wang1} in a different setting.

\begin{algorithm}[H]
	\caption{Gradient descent method to compute control iterates$\{U_{h\tau}^{(\ell)}\}_{\ell\in\mathbb{N}}$}
	\label{tt1}
	\begin{algorithmic}[1]
		\State \textbf{Input:} Fix given $X_{1,0}, X_{2,0}\in \mathbb{H}_0^1$, $\widetilde{X}\in C_t\mathbb{H}_0^1$, noise coefficient $\sigma\in \mathbb{L}^2_{\mathbb{F}}C_t\mathbb{H}_0^1$, initial control iterate $U_{h\tau}^{(0)}\in\mathbb{U}_{h\tau}$, and fix $\kappa>0$.
		\State \textbf{Iterates:} For any $\ell\in \mathbb{N}\cup\{0\}$;
		\State \textbf{State iterates:} Compute the state iterates $ (X_{1, h\tau }^{(\ell)}, X_{1, h\tau }^{(\ell)})\in  \mathbb{X}_{h\tau}\times \mathbb{X}_{h\tau}$ such that $$(X_{1, h\tau}^{(\ell)},X_{2, h\tau}^{(\ell)}):=(\mathcal{X}_{1,h\tau}[U_{h\tau}^{(\ell)}], \mathcal{X}_{2,h\tau}[U_{h\tau}^{(\ell)}]).$$
		\State \textbf{Adjoint iterates:} Compute the adjoint iterates $({Y}_{1, h\tau}^{(\ell)}, {Y}_{2, h\tau}^{(\ell)})\in \mathbb{X}_{h\tau}\times \mathbb{X}_{h\tau}$ such that $$(Y_{1,h\tau}^{(\ell)}, Y_{2,h\tau}^{(\ell)}):=(\mathcal{Y}_{1,h\tau}[U_{h\tau}^{(\ell)}], \mathcal{Y}_{2,h\tau}[U_{h\tau}^{(\ell)}]).$$
		\State \textbf{Update iterates:}
		Update $U_{h}^{(\ell+1)}\in \mathbb{U}_{h\tau}$ by the following formula: for all $n=0,1..,N-1$
		\begin{align*}
			U_{h\tau}^{(\ell+1)}(t_n):=(1-\frac{\alpha}{\kappa})U_{h\tau}^{(\ell)}(t_n)+ \frac{1}{\kappa}\mathbb{E}\big[{Y}_{2, h\tau}^{(\ell)}(t_{n+1})\big|\mathcal{F}_{t_{n}}\big].
		\end{align*}
	\end{algorithmic}
\end{algorithm}
To find rate of convergence for the ${\bf SLQ}_{h\tau}^{\mbox{grad}}$, one needs the Lipschitz constant of $\mathcal{D}_U\hat{J}_{h\tau}(U_{h\tau})$ which can be find as follows: for all $U_{h\tau}, V_{h\tau}\in \mathbb{U}_{h\tau}$,
\begin{align*}
	\left\langle \mathcal{D}^2_{U}\hat{J}_{h\tau}(U_{h\tau})V_{h\tau},V_{h\tau}\right\rangle=&\mathbb{E}\bigg[\int_0^T\big[\left\langle\mathcal{X}_{1, h\tau}^0[V_{h\tau}](t),\mathcal{X}_{1, h\tau}^0[V_{h\tau}](t)\right\rangle+ \alpha\left\langle V_{h\tau}(t), V_{h\tau}(t)\right\rangle\big]\,{\rm d}t\\&\qquad+\beta\left\langle\mathcal{X}_{1,h\tau}^0[V_{h\tau}](T),\mathcal{X}_{1, h\tau}^0[V_{h\tau}](T)\right\rangle\bigg].
\end{align*}
It shows that for all $U_{h\tau}, V_{h\tau}\in\mathbb{U}_{h\tau}$,
\begin{align*}
	|\left\langle \mathcal{D}^2_{U}\hat{J}_{h\tau}(U_{h\tau})V_{h\tau},V_{h\tau}\right\rangle|&\le\mathbb{E}\big[\|\mathcal{X}_{1,h\tau}^0[V_{h\tau}]\|_{\mathbb{L}^2_{t, x}}^2]+\alpha\mathbb{E}\big[\|V_{h\tau}\|_{\mathbb{L}^2_{t, x}}^2\big]+\beta\mathbb{E}\big[\|\mathcal{X}_{h\tau}^0[V_{h\tau}](T)\|^2_{\mathbb{L}^2_x}\big]\\&\le (T+\beta)c_Pc_1e^{c_2T}\mathbb{E}\big[\|V_{h\tau}\|_{\mathbb{L}^2_{t, x}}^2\big]+\alpha\mathbb{E}\big[\|V_{h\tau}\|_{\mathbb{L}^2_{t, x}}^2\big]\\&=\big((T+\beta)c_Pc_1e^{c_2T}+\alpha\big)\mathbb{E}\big[\|V_{h\tau}\|_{\mathbb{L}^2_{t, x}}^2\big],
\end{align*}
where \eqref{today1414} is used and where \( c_1 = c_P \gamma^2 + \frac{\gamma^2 \tau}{4} (2c_P + 1) + 1 \), \( c_2 = 1 \), $c_P=\bigg(\mbox{diam}(D)/\pi\bigg)^2.$

\noindent
It shows that for all $U_{h\tau}\in\mathbb{U}_{h\tau}$,
\begin{align*}
	\|\mathcal{D}^2_U \hat{J}_{h\tau}(U_{h\tau})\|_{\mathcal{L}(\mathbb{U}_{h\tau};\mathbb{U}_{h\tau})}\le \big((T+\beta)c_Pc_1e^{c_2T}+\alpha\big).
\end{align*}
It gives the Lipschitz constant $K$ of $\mathcal{D}_U\hat{J}_{h\tau}(U_{h\tau})$ such that $$K=\|\mathcal{D}^2_U \hat{J}_{h\tau}(U_{h\tau})\|_{\mathcal{L}(\mathbb{U}_{h\tau};\mathbb{U}_{h\tau})}\le \big((T+\beta)c_Pc_1e^{c_2T}+\alpha\big).$$
\begin{proposition}[Error between $U_{h\tau}^{(\ell)}$ and ${U}_{h\tau}^*$]\label{error for gradient descent method} Let Assumption~\ref{CC} hold and $\kappa>K$. Then there exists a constant $C>0$ such that the following error estimates hold:
	\begin{align*}
		\mathbb{E}\big[\|{U}^*_{h\tau}-U_{h\tau}^{(\ell)}\|_{\mathbb{L}^2_{t, x}}^2\big]&\le C\bigg(1-\frac{\alpha}{\kappa}\bigg)^{\ell},\\
        \hat{\mathcal{J}}_{h\tau}(U_{h\tau}^{(\ell)})-\hat{\mathcal{J}}_{h\tau}(U_{h\tau}^*)&\le \frac{2\kappa\mathbb{E}\big[\|U_{h\tau}^*-U_{h\tau}^{(0)}\|_{\mathbb{L}^2_{t,x}}^2\big]}{\ell}.
	\end{align*}
\end{proposition}

\begin{proof}  The proof is a direct consequence of \cite[Theorem 1.2.4]{Ne} with Lipschitz constant $K$.
\end{proof}
\subsection{Final result of the error analysis}
\begin{thm}\label{convergence of gradient descent method} Let Assumption~\ref{BB} hold and $\kappa > K$. Let $(X^*_1, X_2^*, U^*)$  be solve problem {\bf SLQ}~\eqref{1.3}-\eqref{1.4} and $(X_{h\tau}^{(\ell)}, U_{h\tau}^{(\ell)})$ be computed by Algorithm~\ref{tt1}. Then there exists a positive constant $C$ such that for $\kappa>K$ and for all $t\in[0,T]$,
	\begin{align*}
		&\mathbb{E}\big[\|U^*-U_{h\tau}^{(\ell)}\|_{\mathbb{L}^2_{t, x}}^2\big]+\mathbb{E}\big[\|\nabla (X_1^*(t)- X_{1,h\tau}^{(\ell)}(t))\|_{\mathbb{L}^2_{x}}^2\big]+\mathbb{E}\big[\|X_2^*(t)- X_{2,h\tau}^{(\ell)}(t)\|_{\mathbb{L}^2_{x}}^2\big]\\&\le\,C\ \bigg(\tau +h^2+ \bigg(1-\frac{\alpha}{\kappa}\bigg)^{\ell}\bigg)\big(\|X_{1,0}\|_{\mathbb{H}_x^3}^2+ \|X_{2,0}\|_{\mathbb{H}^2_x}^2+\|\widetilde{X}\|_{C_t\mathbb{H}_x^2\cap C^{1/2}_t\mathbb{H}_0^1}^2+\mathbb{E}\big[ \|\sigma\|_{\mathbb{L}^2_t\mathbb{H}^2_x\cap C^{1/2}_t\mathbb{H}_0^1}^2\big]\big).
	\end{align*}
\end{thm}
\begin{proof}
	The proof is a direct consequence of Theorem~\ref{final result for full discrete scheme}, stability estimates~\eqref{today0003} and Proposition~\ref{error for gradient descent method}.
\end{proof}
\begin{remark}\label{5.1} In the gradient descent algorithm ({\em i.e.}, Algorithm~\ref{tt1}), computing the adjoint iterate ${Y}_{2,h\tau}^{(\ell)}$ requires the evaluation of a conditional expectation. Since these conditional expectations are generally not available in closed form, they must be approximated. One common approach is to estimate the conditional expectation using regression-based methods \cite{Lemor2006,BouchardTouzi2004, gobet2005,BenderDenk2007}, a statistical technique; see subsection~\ref {subsection 1.3} for more of its details.  In the presence of multiplicative noise (i.e., $\gamma\neq 0$), one may use the methodology of the random partition estimator method\cite{Duns&Prohl} to approximate (simulate) the conditional expectation in the adjoint iterates $Y_{h\tau}^{(\ell)}$--this method \cite{Duns&Prohl} is practical for limited higher dimension of state space. A comprehensive analysis of such methods lies beyond the scope of this paper. However, in the next subsection, we demonstrate that in the presence of only additive noise ({\em i.e.}, $\gamma=0$), the conditional expectation can be computed explicitly by the help of {\em artificial} gradient iterates.
\end{remark}
\subsection{Implementable scheme}\label{subsection 5.4}
\noindent
In the case of additive noise ({\em i.e.}, \(\gamma=0\)), the adjoint iterate \(Y_{2,h\tau}^{(\ell)}\) in Algorithm~\ref{tt1} can be computed using the new approach based on {\em artificial} gradient iterates, which eliminates the need of the approximation of conditional expectations. Therefore, in this subsection, we restrict our analysis to the case of additive noise.
\subsubsection{Artificial iterates for gradient descent method:}\label{5.4.1}
For all $\ell\in \mathbb{N}\cup\{0\}$ we introduce the concept of {\em artificial} control iterate, {\em artificial} state iterate and {\em artificial} adjoint iterate to compute adjoint iterate $Y_{2, h\tau}^{(\ell)}$ in Algorithm~\ref{tt1} with $\gamma=0$ as follows: 
\begin{itemize}
 \item [1.] \textbf{Artificial control iterate:} For $m\in\{0,...,N-1\}$, let $\mathfrak{U}_{m}^{(\ell)}\in \mathbb{U}_{h\tau}$ such that for all $n=0,...,N-1$,
 \begin{align}\label{aicontrol}
     \mathfrak{U}_{m}^{(\ell)}(t_n):=\mathbb{E}\big[{U}_{h\tau}^{(\ell)}(t_n)\big|\mathcal{F}_{t_m}\big].
 \end{align}

    \item [2.] \textbf{Artificial state iterate:} For $m\in\{0,...,N-1\}$ and $i=1,2$, let $\frak{X}_{i,m}\in \mathbb{X}_{h\tau}$ such that for all $n=0,...,N$
    \begin{align}\label{aistate}
    \mathfrak{X}_{i,m}^{(\ell)}(t_n):=\mathbb{E}\big[\mathcal{X}_{i,h\tau}^{(\ell)}(t_n)\big|\mathcal{F}_{t_m}\big].
\end{align}
Then by using the tower property of conditional expectation the {\em artificial} state iterate $\bigl(\mathfrak{X}_{1,m}^{(\ell)}, \mathfrak{X}_{2,m}^{(\ell)}\bigr)\in\mathbb{X}_{h\tau}\times\mathbb{X}_{h\tau}$ solves the following {\em artificial} state equations for all $n\in \{1,...,N-1\}$, 
    \begin{equation}\label{i}
	\begin{cases}
		\mathfrak{X}_{1,m}^{(\ell)}(t_{n+1})-\mathfrak{X}^{(\ell)}_{1,m}(t_n)=\frac{\tau}{2} \big(\mathfrak{X}_{2,m}^{(\ell)}(t_{n+1})+\mathfrak{X}_{2,m}^{(\ell)}(t_n)\big),\\
		\mathfrak{X}_{2,m}^{(\ell)}(t_{n+1}) - \mathfrak{X}_{2,m}^{(\ell)}(t_n) = \frac{\tau}{2} \Delta_h \big(\mathfrak{X}_{1,m}^{(\ell)}(t_{n+1})+\mathfrak{X}_{1,m}^{(\ell)}(t_n)\big) + \tau \mathfrak{U}_{m}^{(\ell)}(t_n)+\mathfrak{W}_{m}(t_n), \\
		\mathfrak{X}_{1,m}^{(\ell)}(0) =\mathcal{R}_h X_{1,0},\\
		\mathfrak{X}_{2,m}^{(\ell)}(0)=\mathcal{R}_h X_{2,0},
	\end{cases}
\end{equation}
where $\mathfrak{W}_{m}(t_n):=\mathbb{E}\bigl[\sigma(t_n)\Delta_{n+1}W\big|\mathcal{F}_{t_m}\bigr]=\begin{cases}
    0,\qquad& n+1>m,\\
    \sigma(t_n)\Delta_{n+1}W,\qquad& n+1\le m.
\end{cases}$
 \item [3.] \textbf{Artificial adjoint iterate:} For $m\in\{0,...,N-1\}$ and $i=1,2$, let $\mathfrak{Y}_{i,m}^{(\ell)}\in \mathbb{X}_{h\tau}$ such that for all $n=0,...,N$,
 \begin{align}\label{aiadjoint}
     \mathfrak{Y}_{i,m}^{(\ell)}(t_n):=\mathbb{E}\big[\mathcal{Y}_{i,h\tau}^{(\ell)}(t_n)\big|\mathcal{F}_{t_m}\big].
 \end{align}
Then by using the tower property of conditional expectation the {\em artificial} adjoint state $\bigl(\mathfrak{Y}_{1,m}^{(\ell)}, \mathfrak{Y}_{2,m}^{(\ell)}\bigr)\in \mathbb{X}_{h\tau}\times\mathbb{X}_{h\tau}$ solves the following {\em artificial} backward equations: for all $0\le m\le n$, 
\begin{align}\label{ii}
\begin{cases}
    \mathfrak{Y}_{1,m}^{(\ell)}(t_{n})=\mathfrak{Y}_{1,m}^{(\ell)}(t_{n+1})+\frac{\tau}{2}\Delta_h\bigl[\mathfrak{Y}_{2,m}^{(\ell)}(t_{n+1})+\mathfrak{Y}_{2,m}^{(\ell)}(t_n)\bigr]+\tau\bigl(\widetilde{X}_h(t_{n})-\mathcal{X}_{1,m}^{(\ell)}(t_{n})\bigr),\\
    \mathfrak{Y}_{2,m}^{(\ell)}(t_{n})=\mathfrak{Y}_{2,m}^{(\ell)}(t_{n+1})+\frac{\tau}{2}\bigl[\mathfrak{Y}_{1,m}^{(\ell)}(t_{n+1})+\mathfrak{Y}_{1,m}^{(\ell)}(t_n)\bigr],\\
    \mathfrak{Y}_{1,m}^{(\ell)}(t_{N})=\frac{\tau}{2}\Delta_h\mathfrak{Y}_{2,m}^{(\ell)}(t_N)+\beta(\widetilde{X}_{h\tau}(t_N)-\mathfrak{X}^{(\ell)}_{1,m}(t_N)),\\
    \mathfrak{Y}_{2,m}^{(\ell)}(t_N)=\frac{\tau}{2}\mathfrak{Y}_{1,m}^{(\ell)}(t_N),
\end{cases}
\end{align}
and for all $m>n$, $\bigl(\mathfrak{Y}_{1,m}^{(\ell)}(t_n), \mathfrak{Y}_{2,m}^{(\ell)}(t_n)\bigr):=\bigl(\mathfrak{Y}_{1,n}^{(\ell)}(t_n), \mathfrak{Y}_{2,n}^{(\ell)}(t_n)\bigr)$.

\item [4.] \textbf{Artificial updated control iterate:}
The {\em artificial} update control  $\mathfrak{U}_{m}^{(\ell)}\in \mathbb{U}_{h\tau}$ satisfies the following formula: for all $n=0,1...,N-1$,
\begin{align}\label{updatecontrol}
    \mathfrak{U}_{m}^{(\ell+1)}(t_n):=(1-\frac{\alpha}{\kappa})\mathfrak{U}_{m}^{(\ell)}(t_n)+\frac{1}{\kappa}\mathfrak{Y}_{2,m}^{(\ell)}(t_{n+1}).
\end{align}
\end{itemize}
\subsubsection{Computation of gradient iterates}
\noindent
From items~(1)--(4), it is evident that the computation of these {\em artificial} iterates does not involve any direct evaluation of conditional expectations. By employing these {\em artificial} iterates, we can efficiently compute the state iterate \(X_{1,h\tau}^{(\ell)}\), the adjoint iterate \(Y_{1,h\tau}^{(\ell)}\), and the control iterate \(U_{h\tau}^{(\ell)}\) of Algorithm~\ref{tt1} as follows:

\begin{itemize}
\item [A.] \textbf{Gradient control, state and adjoint iterates:} By the help of \eqref{aicontrol}, \eqref{aistate} and \eqref{aiadjoint}, for $i=1,2,$ the control iterate $U_{h\tau}^{(\ell)}\in\mathbb{U}_{h\tau}$, the state iterate $X_{i, h\tau}^{(\ell)}\in \mathbb{X}_{h\tau}$ and the adjoint iterate $Y_{i,h\tau}^{(\ell)}\in\mathbb{X}_{h\tau}$ of Algorithm~\ref{tt1} are then computed by the following  relation: for all $n=0,1,...,N-1$,
\begin{align}\label{relation}
U_{h\tau}^{(\ell)}(t_n)=\mathfrak{U}_{n}^{(\ell)}(t_{n}),\qquad X_{i, h\tau}^{(\ell)}(t_{n+1})=\mathfrak{X}_{i,n+1}^{(\ell)}(t_{n+1}),\qquad Y_{i,h\tau}^{(\ell)}(t_{n+1})=\mathfrak{Y}_{i,n+1}^{(\ell)}(t_{n+1}).
\end{align}
\end{itemize}
\noindent
\noindent
Consequently, Algorithm~\ref{tt1} with \(\gamma=0\) can be reformulated into the following {\em implementable} algorithm.

\begin{algorithm}[ht]\label{tt3}
	\caption{{\em Implementable} algorithm to compute control iterates $\{U_{h\tau}^{(\ell)}\}_{\ell\in \mathbb{N}}$ of Algorithm~\ref{tt1} with $\gamma=0$}
	\label{tt3}
    \begin{itemize}
        \item [1.] \textbf{Input:} Fix given $X_{1,0}, X_{2,0}\in \mathbb{H}^1_0$, $\widetilde{X}\in C_t\mathbb{H}_0^1$, noise coefficient $\sigma\in \mathbb{L}^2_{\mathbb{F}}C_t\mathbb{H}_0^1$, initial guess ${U}_{h\tau}^{(0)}\equiv 0$, and fix $\kappa>K$, total time steps $N$, total space steps $M$, $\tau=1/N$, and $h=1/M$.
        \item [2.]
        \textbf{Gradient iterates:} For all $\ell\in \mathbb{N}\cup\{0\}$; 
        \begin{itemize}
        \item [2(i).]\textbf{Artificial iterate:} For all $m\in\{0,...,N-1\}$, 
        \begin{itemize}
            \item [a.] \textbf{Initial control iterate} For all $n\in\{0,...,N\}$, $\mathfrak{U}_{m}^{(0)}(t_n)\equiv 0$.
            \item [b.]  \textbf{Artificial state iterates:} Compute $(\mathfrak{X}_{1,m}^{(\ell)},\mathfrak{X}^{(\ell)}_{2,m})\in \mathbb{X}_{h\tau}\times\mathbb{X}_{h\tau}$ by  \eqref{i}. 
\item [c.] \textbf{Artificial adjoint iterates:}  Compute $(\mathfrak{Y}_{1,m}^{(\ell)},\mathfrak{Y}_{2,m}^{(\ell)})\in \mathbb{X}_{h\tau}\times\mathbb{X}_{h\tau}$ by \eqref{ii}.
\item [d.]  \textbf{Artificial update control iterates:} Update the {\em artificial} control $\mathfrak{U}_m^{(\ell+1)}\in\mathbb{U}_{h\tau}$ by~\eqref{updatecontrol}.
        \end{itemize}
        \item [2(ii).]  \textbf{Gradient control iterates:} Compute the control iterate $U_{h\tau}^{(\ell+1)}\in \mathbb{U}_{h\tau}$ by~\eqref{updatecontrol} and~\eqref{relation}.
        \end{itemize}
    \end{itemize}
\end{algorithm}
\section{Conclusion}\label{Remark 6.2}
This work proposes convergence with rates for an {\em implementable} scheme to solve the {\bf SLQ} roblem \eqref{1.1}—\eqref{1.2}. From a methodological viewpoint, it contains two main novelties. First, we introduce a {\em new} proposition (Proposition~\ref{Proposition02}) that circumvents the lengthy {\em Malliavin calculus} arguments  in the error analysis for the optimal pair $(X^*,U^*)$ to {\bf SLQ} problem\eqref{1.1}-\eqref{1.2} as discussed in Remarks~\ref{Remark 4.3} and~\ref{Remark 4.5}. Second, we {\em eliminate} the costly approximation of the conditional expectations that typically arise in the computation of the adjoint state \((Y_{1, h\tau}, Y_{2,h\tau})\) in Pontryagin's maximum principle ({\em cf.}\ Proposition \ref{fully discrete Pontryagin's maximum principle} and Remark~\ref{5.1}) by introducing a new concept of {\em artificial} gradient iterates; see Section~\ref{5.4.1} . Computational studies supporting efficiency are reported in Section~\ref{example 1}.

\appendix
\section{Technical Results}
In this section, we state bounds in stronger norms for {\bf SLQ} problem \eqref{1.3}-\eqref{1.4}. These results rest on the stronger data Assumptions~\ref{BB} as stated in Section~\ref{sec assumption}.
\begin{lem}[Spatial regularity of optimal control]\label{Spatial regularity of optimal control}Let Assumption~\ref{BB} hold. Let $(X_1^*, X_2^*, U^*)$ be the unique optimal control tuple for {\bf SLQ} problem \eqref{1.3}-\eqref{1.4}. Then there exists a $C>0$ such that the following estimates hold:
    \begin{align}\label{space-regularity for optimal control 1}
		\mathbb{E}\big[\sup_{t\in[0,T]}\|U^*(t)\|^2_{\mathbb{H}^1_0}\big]&\le C(\| X_{1,0}\|_{\mathbb{H}^1_0}^2+ \|X_{2,0}\|_{\mathbb{L}^2_x}^2+ \|\widetilde{X}\|_{\mathbb{L}^2_{t,x}}^2+\mathbb{E}\big[\|\sigma\|_{\mathbb{L}^2_{t,x}}^2\big]),
	\end{align}
    \begin{equation}\label{space-regularity for otimal state 21}
		\mathbb{E} \left[ \sup_{t\in[0,T]} \left( \|X_1^*(t)\|_{\mathbb{H}^2_x}^{2} + \|X_2^*(t)\|_{\mathbb{H}^1_0}^{2} \right) \right] \leq C(\|X_{1,0}\|_{\mathbb{H}^2_x}^2 +\|X_{2,0}\|_{\mathbb{H}^1_0}^2 +\|\widetilde{X}\|_{C_t\mathbb{H}_0^1}^2+\mathbb{E}\big[\|\sigma\|_{\mathbb{L}^2_t\mathbb{H}^1_0}^2\big]),
	\end{equation}
    \begin{align}\label{00985}
         \mathbb{E}\big[\sup_{t\in[0,T]}\|U^*(t)\|_{\mathbb{H}_x^2}^2\big]&\le C(\|X_{1,0}\|_{\mathbb{H}^2_x}^2+ \|X_{2,0}\|_{\mathbb{H}^1_0}^2+ \|\widetilde{X}\|_{C_t\mathbb{H}_0^1}^2+\mathbb{E}\big[\|\sigma\|_{\mathbb{L}^2_{t}\mathbb{H}_0^1}^2\big]),
     \end{align}
    \begin{equation}\label{space-regularity for otimal state 2}
		\mathbb{E}[ \sup_{0 \leq t \leq T} \big( \|X_1^*(t)\|_{\mathbb{H}^3_x}^{2} + \|X_2^*(t)\|_{\mathbb{H}^2_x}^{2}] \leq C\big(\|X_{1,0}\|_{\mathbb{H}^3_x}^2 +\|X_{2,0}\|_{\mathbb{H}^2_x}^2 +\|\widetilde{X}\|_{C_t\mathbb{H}_0^1}^2+\mathbb{E}\big[\|\sigma\|_{\mathbb{L}^2_t\mathbb{H}^2_x}^2\big]\big).
	\end{equation}
\end{lem}
\begin{proof}
The asserted regularity estimates follow directly from the optimality condition \eqref{1.6} together with Lemmas~\ref{Lemma 2.1} and \ref{Lemma 2.3}. More precisely, \eqref{space-regularity for optimal control 1} is obtained from the optimality condition \eqref{1.6} combined with \eqref{today020} and \eqref{001}. Then \eqref{space-regularity for otimal state 21} follows by combining \eqref{002} with \eqref{space-regularity for optimal control 1}. Estimate \eqref{00985} is a consequence of \eqref{1.6}, \eqref{space regularity H-2 bound for adjoint Y} and \eqref{space-regularity for otimal state 21}. Finally, \eqref{space-regularity for otimal state 2} follows from \eqref{0002} together with \eqref{00985}. The intermediate computations are routine and are left to the reader.
\end{proof}
        The following proposition gathers stability bounds in stronger norms for the semi discretization ${\bf SLQ}_h$ \eqref{3.1}—\eqref{3.2}.
\begin{proposition}\label{Proposition 3.3} Let Assumption~\ref{BB} hold. Let $(X_{1,h}^*, X_{2,h}^*, U_{h}^*)$ be the unique optimal tuple to  {\bf SLQ}$_h$ problem \eqref{3.1}-\eqref{3.2}. Then the following estimates hold:
	{	\begin{align}
			&\mathbb{E}\big[\sup_{s\in[t,T]}\big(\|X_{2,h}^*(t)\|_{\mathbb{L}^2_x}^2+\|\nabla X_{1,h}^*(t)\|_{\mathbb{L}^2_x}^2\big]\big]\notag\\&\qquad\qquad\le C \big[\|X_{2,h}(0)\|_{\mathbb{L}^2_x}^2+\|\nabla X_{1,h}(0)\|_{\mathbb{L}^2_x}^2+\mathbb{E}\big[\|U_{h}^*\|_{\mathbb{L}^2_{t, x}}^2+\|\sigma\|_{\mathbb{L}^2_{t, x}}^2\big]\big),
		\end{align}
		and
		\begin{align}
			&\mathbb{E}\big[\sup_{s\in[t,T]}\big[\|\nabla X_{2,h}^*(t)\|_{\mathbb{L}^2_x}^2+\|\Delta_h X_{1,h}^*(t)\|_{\mathbb{L}^2_x}^2\big]\big]\notag\\
			&\qquad\qquad\le C \big(\|\nabla X_{2,0}(0)\|_{\mathbb{L}^2_x}^2+\|\Delta_h X_{1,h}(0)\|_{\mathbb{L}^2_x}^2+\mathbb{E}\big[\|\nabla U_{h}^*\|_{\mathbb{L}^2_{t, x}}^2+\|\nabla \sigma\|_{\mathbb{L}^2_{t, x}}^2  \big]\big).
	\end{align}}
\end{proposition}
\begin{proof}
	For the proof, one can use similar arguments as used in the proof of \cite[Lemma 3.2]{FengPandaProhl2024}. It is a direct consequence of It\^o formula. 
\end{proof}
\begin{proposition}[Higher regularity estimate]\label{Proposition 3.4} Let Assumption~\ref{BB} hold.
	Let the quadruple $(Y_{1,h}, Y_{2,h}, Z_{1,h}, Z_{2, h})$ be the unique solution to ${\bf BSPDE}_h$ \eqref{discrete adjoint ode}, then there exists $C>0$ such that 
	{\begin{align}
			&\mathbb{E}\big[\sup_{t\in[0,T]}\big[\|Y_{1,h}(t)\|_{\mathbb{L}^2_x}^2+\|\nabla Y_{2,h}(t)\|_{\mathbb{L}^2_x}^2\big]\big]+\mathbb{E}\bigg[\int_0^{T}\|Z_{1,h}(t)\|_{\mathbb{L}^2_x}^2\,{\rm d}t +\int_0^T\|\nabla Z_{2,h}(t)\|^2_{\mathbb{L}_x^2}\,{\rm d}t\bigg]\notag\\&\le C \mathbb{E}\big[\|X_{1,h}^*-\widetilde{X}_h\|_{\mathbb{L}^2_{t, x}}^2+ \beta^2\|X_{1,h}^*(T)-\widetilde{X}_h(T)\|_{\mathbb{L}^2_x}^2\big],\label{estimate for discrete adjoint 1}
		\end{align}
		\begin{align}
			&\mathbb{E}\big[\sup_{t\in[0,T]}\big[\|\nabla Y_{1,h}(t)\|_{\mathbb{L}^2_x}^2+\|\Delta_hY_{2,h}(t)\|_{\mathbb{L}^2_x}^2\big]\big]+\mathbb{E}\bigg[\int_0^{T}\|\nabla Z_{1,h}(t)\|_{\mathbb{L}^2_x}^2\,{\rm d}t +\int_0^T\|\Delta_h Z_{2,h}(t)\|^2_{\mathbb{L}_x^2}\,{\rm d}t\bigg]\notag\\&\le C \mathbb{E}\big[\|\nabla \big(X_{1,h}^*-\widetilde{X}_h\big)\|_{\mathbb{L}^2_{t, x}}^2+\|\nabla \big(X_{1,h}^*(T)-\widetilde{X}_h(T)\big)\|_{\mathbb{L}^2_x}^2\big], \label{estimate for discrete adjoint 2}
		\end{align}
		and
		\begin{align}
			&\mathbb{E}\big[\sup_{t\in[0,T]}\|\Delta_h Y_{1,h}(t)\|_{\mathbb{L}^2_x}^2+\|\nabla\Delta_hY_{2,h}(t)\|_{\mathbb{L}^2_x}^2\big]\big]+\mathbb{E}\bigg[\int_0^{T}\|\Delta_h Z_{1,h}(t)\|_{\mathbb{L}^2_x}^2\,{\rm d}t +\int_0^T\|\nabla\Delta_h Z_{2,h}(t)\|^2_{\mathbb{L}_x^2}\,{\rm d}t\bigg]\notag\\&\le C \mathbb{E}\big[\|\Delta_h \big(X_{1,h}^*-\widetilde{X}_h\big)\|_{\mathbb{L}^2_{t, x}}^2+\|\Delta_h \big(X_{1,h}^*(T)-\widetilde{X}_h(T)\big)\|_{\mathbb{L}^2_x}^2\big]. \label{estimate for discrete adjoint 3}
	\end{align}}
\end{proposition}
\begin{proof}
	For the proof, we can follow similar lines as used in the proof of Lemma \ref{Lemma 2.3}; it is a direct consequence of It\^o formula. 
\end{proof}
\begin{lem}[Higher stability estimate] Let Assumption~\ref{BB} hold. Let $U_h^*$ be the unique optimal control to {\bf SLQ}$_h$ problem \eqref{3.1}-\eqref{3.2}. Then the following estimates hold:
	\begin{align}\label{434343}
		\mathbb{E}\bigg[\sup_{t\in[0,T]}\|\nabla\Delta_h U^*_h(t)\|_{\mathbb{L}^2_x}^2\bigg]\le C\big(\|X_{2,0}\|_{\mathbb{H}_x^1}^2+\| X_{1,0}\|_{\mathbb{H}_x^2}^2+\|\widetilde{X}\|_{C_t\mathbb{H}_x^2}^2+ \mathbb{E}\big[\|\sigma\|_{\mathbb{L}^2_t\mathbb{H}_0^1}^2\big]\big).
	\end{align}
\end{lem}
\begin{proof}
	The proof is a direct consequence of the semi-discrete optimality condition \eqref{3.3} and Propositions \ref{Proposition 3.3}-\ref{Proposition 3.4}.
\end{proof}
\begin{proposition}[Time regularity estimate] Let Assumption~\ref{BB} hold. Let $(X_{1,h}, X_{2,h})$ be the unique solution to  ${\bf SPDE}_h$ \eqref{3.2} with given control $U_h\in \mathbb{U}_{h\tau}$. Then the following estimates hold:
	\begin{align}\label{time-regularity for semi-discrete state-1}
		&\sum_{n=0}^{N}\mathbb{E}\bigg[\int_{t_{n}}^{t_{n+1}}\|\Delta_h (X_{1,h}(t)-X_{1,h}(t_{n+1}))\|_{\mathbb{L}^2_x}^2\,{\rm d}t +\sum_{n=0}^{N}\int_{t_{n}}^{t_{n+1}}\| (X_{1,h}(t)-X_{1,h}(t_{n}))\|_{\mathbb{L}^2_x}^2\,{\rm d}t\bigg]\notag\\&\qquad\le C\tau^2 \big(\|X_{1,0}\|_{\mathbb{H}^3_x}^2+ \|X_{2,0}\|_{\mathbb{H}_x^2}^2+ \mathbb{E}\big[\|\sigma\|_{\mathbb{L}^2_t\mathbb{H}_x^2}^2+ \|\nabla\Delta_hU_{h}\|_{\mathbb{L}_{t, x}^2}^2\big]\big),
	\end{align}
	and
	\begin{align}\label{time-regularity for semi-discrete state-2}
		&\sum_{n=0}^{N}\mathbb{E}\bigg[\int_{t_{n}}^{t_{n+1}}\|\nabla(X_{2,h}(t)-X_{2,h}(t_{n+1}))\|_{\mathbb{L}^2_x}^2\,{\rm d}t + \sum_{n=0}^{N}\int_{t_{n}}^{t_{n+1}}\|(X_{2,h}(t)-X_{2,h}(t_{n+1}))\|_{\mathbb{L}^2_x}^2\,{\rm d}t\bigg]\notag\\&\qquad\le C\tau \big(\|X_{1,0}\|_{\mathbb{H}^3_x}^2+ \|X_{2,0}\|_{\mathbb{H}^2_x}^2+ \mathbb{E}\big[\|\sigma\|_{\mathbb{L}^2_t\mathbb{H}^2_x}^2+ \|\nabla\Delta_hU_{h}\|_{\mathbb{L}_{t, x}^2}^2\big]\big).
	\end{align}
\end{proposition}
\begin{proof}
	For the proof, one can follow similar lines as in the proof of \cite[Lemma 3.9]{ProhlWang2021}. It is a direct consequence of Proposition \ref{Proposition 3.3}.
\end{proof}
The following result addresses the approximation in time of the ${\bf BSPDE}_h$ \eqref{discrete state equation}. 
\begin{proposition}[Time-regularity of adjoint variable]\label{time-regularity of adjoint variable} Let Assumption~\ref{BB} hold. Let $(Y_{1,h}, Y_{2,h}, Z_{1,h}, Z_{2,h})$ be the unique solution to ${\bf BSPDE}_h$ \eqref{discrete state equation}. Then the exists $C>0$ such that
\begin{align}\label{today03}
&\mathbb{E}\bigg[\|Y_{2,h}-\Pi_\tau Y_{2,h}\|^2_{\mathbb{L}^2_{t, x}}\bigg]\le C\tau\big( \mathbb{E}\big[\|X_{1,h}^*\|_{\mathbb{L}^2_{t, x}}^2\big]+\|\widetilde{X}\|_{C_t\mathbb{L}^2_x}^2\big),\\
&\mathbb{E}\bigg[\|\nabla Y_{2,h}-\Pi_\tau \nabla Y_{2,h}\|^2_{\mathbb{L}^2_{t, x}}\bigg]+ \mathbb{E}\bigg[\|Y_{1,h}-\Pi_\tau Y_{1,h}\|^2_{\mathbb{L}^2_{t, x}}\bigg]\le C\tau\big( \mathbb{E}\big[\|\nabla X_{1,h}^*\|_{\mathbb{L}^2_{t, x}}^2\big]+\|\nabla\widetilde{X}\|_{\mathbb{L}^2_{t, x}}^2\big),\label{today47}\\
        &\sum_{n=0}^{N-1}\bigg[\mathbb{E}\bigg[\int_{t_n}^{t_{n+1}}\|\nabla Y_{2,h}(t)-\nabla Y_{2,h}(t_{n+1})\|^2_{\mathbb{L}^2_{t, x}}\,{\rm d}t\bigg]+ \mathbb{E}\bigg[\int_{t_{n}}^{t_{n+1}}\|Y_{1,h}(t)-Y_{1,h}(t_{n+1})\|^2_{\mathbb{L}^2_{t, x}}\,{\rm d}t\bigg]\notag\\&\qquad\le C\tau\big( \mathbb{E}\big[\|\nabla X_{1,h}^*\|_{\mathbb{L}^2_{t, x}}^2\big]+\|\nabla\widetilde{X}\|_{\mathbb{L}^2_{t, x}}^2\big).\label{today480}
\end{align}
\end{proposition}
\begin{proof}
From \eqref{discrete adjoint ode} we have $\mathbb{P}$-almost surely, for every $t\in[t_n,t_{n+1}]$,
\[
Y_{2,h}(t)-Y_{2,h}(t_n)
= -\int_{t_n}^t Y_{1,h}(s)\,\,{\rm d}s
+ \int_{t_n}^t Z_{2,h}(s)\,\,{\rm d}W(s).
\]
Hence, by taking the $L^2_x$-norm, squaring, integrating in time and taking expectation, we obtain
\begin{align*}
\mathbb{E}\bigg[\int_{t_n}^{t_{n+1}}\|Y_{2,h}(t)-Y_{2,h}(t_n)\|_{\mathbb{L}^2_x}^2\,\,{\rm d}t\bigg]
&\le \,\mathbb{E}\bigg[\int_{t_n}^{t_{n+1}}\Big\|\int_{t_n}^t Y_{1,h}(s)\,\,{\rm d}s\Big\|_{\mathbb{L}^2_x}^2\,{\rm d}t\bigg]\\
&\qquad + \,\mathbb{E}\bigg[\int_{t_n}^{t_{n+1}}\Big\|\int_{t_n}^t Z_{2,h}(s)\,\,{\rm d}W(s)\Big\|_{\mathbb{L}^2_x}^2\,{\rm d}t\bigg].
\end{align*}
For the deterministic integral we use Cauchy--Schwarz in time to get
\[
\Big\|\int_{t_n}^t Y_{1,h}(s)\,\,{\rm d}s\Big\|_{\mathbb{L}^2_x}^2
\le (t-t_n)\int_{t_n}^t\|Y_{1,h}(s)\|_{\mathbb{L}^2_x}^2\,\,{\rm d}s
\le \tau \int_{t_n}^{t_{n+1}}\|Y_{1,h}(s)\|_{\mathbb{L}^2_x}^2\,\,{\rm d}s,
\]
and therefore
\[
\mathbb{E}\bigg[\int_{t_n}^{t_{n+1}}
\Big\|\int_{t_n}^t Y_{1,h}(s)\,\,{\rm d}s\Big\|_{\mathbb{L}^2_x}^2\,{\rm d}t\bigg]
\le \tau^2\,\mathbb{E}\bigg[\int_{t_n}^{t_{n+1}}\|Y_{1,h}(s)\|_{\mathbb{L}^2_x}^2\,\,{\rm d}s\bigg].
\]
For the stochastic integral we apply the Itô isometry to get
\begin{align*}
\mathbb{E}\bigg[\int_{t_n}^{t_{n+1}}\Big\|\int_{t_n}^t Z_{2,h}(s)\,\,{\rm d}W(s)\Big\|_{\mathbb{L}^2_x}^2\,{\rm d}t\bigg]
&= \mathbb{E}\bigg[\int_{t_n}^{t_{n+1}} \int_{t_n}^t \|Z_{2,h}(s)\|_{\mathbb{L}^2_x}^2\,\,{\rm d}s\,\,{\rm d}t\bigg]\\
&\le \tau \mathbb{E}\bigg[\int_{t_n}^{t_{n+1}}\|Z_{2,h}(s)\|_{\mathbb{L}^2_x}^2\,\,{\rm d}s\bigg].
\end{align*}
We combine the above two estimates to obtain
\[
\mathbb{E}\bigg[\int_{t_n}^{t_{n+1}}\|Y_{2,h}(t)-Y_{2,h}(t_n)\|_{\mathbb{L}^2_x}^2\,\,{\rm d}t\bigg]
\le \tau^2\,\mathbb{E}\bigg[\int_{t_n}^{t_{n+1}}\|Y_{1,h}(s)\|_{\mathbb{L}^2_x}^2\,{\rm d}s\bigg]
+ \tau\,\mathbb{E}\bigg[\int_{t_n}^{t_{n+1}}\|Z_{2,h}(s)\|_{\mathbb{L}^2_x}^2\,{\rm d}s\bigg].
\]
By summing this inequality over $n=0,\dots,N-1$ and using the a priori bound \eqref{estimate for discrete adjoint 1} for the semi-discrete adjoint pair $(Y_{1,h},Z_{1,h})$ yields the desired estimate \eqref{today03}. We can follow similar lines as used for estimate~\eqref{today03} to obtain estimates~\eqref{today47} and \eqref{today480}.
\end{proof}
\begin{proposition}[time-regularity for semi-discrete optimal control $U_h^*$]\label{Time-regularity for semi-discrete optimal control}Let Assumption~\ref{BB} hold. Let $U_h^*$ be the unique semi-discrete optimal control to {\bf SLQ}$_h$  \eqref{3.1}-\eqref{3.2}. Then the following time-regularity holds:
	\begin{align}\label{time-regularity for semi-discrete control}
		\mathbb{E}\big[\|U_h^*-\Pi_\tau U_h^*\|_{\mathbb{L}^2_{t, x}}^2\big]\le C\tau\big(\|X_{2,0}\|_{\mathbb{H}^1_0}^2+\|X_{1,0}\|_{\mathbb{H}^2_x}^2+\|\widetilde{X}\|_{C_t\mathbb{H}_0^1}^2+\mathbb{E}\big[\|\sigma\|_{\mathbb{L}^2_{t}\mathbb{H}^1_0}^2 \big]\big).
	\end{align}
\end{proposition}
\begin{proof} It is direct consequence of the semi-discrete optimality condition \eqref{3.3} and Proposition \ref{time-regularity of adjoint variable}.
\end{proof}
\begin{proposition}\label{Propostion01}Let $(Y_{1,h}, Y_{2,h}, Z_{1,h}, Z_{2,h})$ be solution to ${\bf BSPDE}_h$ \eqref{discrete adjoint ode}, then there exists $C>0$ such that
\begin{align}\label{today48}
&\tau\sum_{n=0}^{N-1}\mathbb{E}\big[\|\nabla\widehat{Y}_{2,h}(t_{n+1})-\nabla\widehat{Y}_{2,h}(t_n)\|_{\mathbb{L}^2_x}^2\big]+\tau\sum_{n=0}^{N-1}\mathbb{E}\big[\|\widehat{Y}_{1,h}(t_{n+1})-\widehat{Y}_{1,h}(t_n)\|_{\mathbb{L}^2_x}^2\big]\notag\\&\le C \tau(\|X_{1,0}\|_{\mathbb{H}_0^1}^2+ \|X_{2,0}\|_{\mathbb{L}^2_x}^2+\|\widetilde{X}\|_{C_t\mathbb{H}_0^1}^2+ \mathbb{E}\big[\|\sigma\|_{\mathbb{L}^2_{t}\mathbb{H}_0^1}^2\big]).
\end{align}
\end{proposition}
\begin{proof} 
Recall that $\widehat{Y}_{2,h}(t_n) = \frac{1}{\tau} \int_{t_{n-1}}^{t_n} Y_{2,h}(t) {\rm  d}t$ for $n=1,\dots,N$ and $\hat{Y}_{2,h}(t_0)={Y}_{2,h}(t_0)$, from Definition~\eqref{jhs}. By the triangle inequality and Cauchy--Schwarz inequality, for $n=1,\dots,N-1$,
\begin{align*}
\tau \|\nabla\widehat{Y}_{2,h}(t_{n+1})-\nabla\widehat{Y}_{2,h}(t_n)\|_{\mathbb{L}^2_x}^2 &= \frac{1}{\tau} \bigg\| \int_{t_n}^{t_{n+1}} \nabla Y_{2,h}(t) {\rm  d}t - \int_{t_{n-1}}^{t_n} \nabla Y_{2,h}(t) {\rm  d}t \bigg\|_{\mathbb{L}^2_x}^2 \\
&= \frac{1}{\tau} \bigg\| \int_{t_n}^{t_{n+1}} \nabla (Y_{2,h}(t) - Y_{2,h}(t_n)) {\rm  d}t - \int_{t_{n-1}}^{t_n} \nabla (Y_{2,h}(t) - Y_{2,h}(t_n)) {\rm  d}t \bigg\|_{\mathbb{L}^2_x}^2 \\
&\le  \int_{t_n}^{t_{n+1}} \|\nabla Y_{2,h}(t) - \nabla Y_{2,h}(t_n)\|_{\mathbb{L}^2_x}^2 {\rm  d}t +  \int_{t_{n-1}}^{t_n} \|\nabla Y_{2,h}(t) - \nabla Y_{2,h}(t_n)\|_{\mathbb{L}^2_x}^2 {\rm  d}t.
\end{align*}
For $n=0$,
\[
\tau \|\nabla\widehat{Y}_{2,h}(t_1) - \nabla\widehat{Y}_{2,h}(t_0)\|_{\mathbb{L}^2_x}^2 = \frac{1}{\tau} \bigg\| \int_{t_0}^{t_1} \nabla (Y_{2,h}(t) - Y_{2,h}(t_0)) {\rm  d}t \bigg\|_{\mathbb{L}^2_x}^2 \le \int_{t_0}^{t_1} \|\nabla Y_{2,h}(t) - \nabla Y_{2,h}(t_0)\|_{\mathbb{L}^2_x}^2 {\rm  d}t.
\]
By summing over $n=0,\dots,N-1$, taking expectations, we obtain
\begin{align*}
&\tau \sum_{n=0}^{N-1} \mathbb{E} \big[ \|\nabla\widehat{Y}_{2,h}(t_{n+1}) - \nabla\widehat{Y}_{2,h}(t_n)\|_{\mathbb{L}^2_x}^2 \big] \le \sum_{n=0}^{N-1} \mathbb{E} \bigg[ \int_{t_n}^{t_{n+1}} \|\nabla Y_{2,h}(t) - \nabla Y_{2,h}(t_n)\|_{\mathbb{L}^2_x}^2 {\rm  d}t \bigg]\\&\qquad +  \sum_{n=1}^{N-1} \mathbb{E} \bigg[ \int_{t_{n-1}}^{t_n} \|\nabla Y_{2,h}(t) - \nabla Y_{2,h}(t_n)\|_{\mathbb{L}^2_x}^2 {\rm  d}t \bigg].
\end{align*}
The second sum shifts to $\sum_{n=0}^{N-2} \mathbb{E} [ \int_{t_n}^{t_{n+1}} \|\nabla Y_{2,h}(t) - \nabla Y_{2,h}(t_{n+1})\|_{\mathbb{L}^2_x}^2 {\rm  d}t ]$. By using \eqref{today47} and \eqref{today480}, the right-hand side is bounded by $C \tau \big( \mathbb{E} \big[ \|\nabla X_{1,h}^*\|_{\mathbb{L}^2_{t, x}}^2+ \|\nabla\widetilde{X}\|_{\mathbb{L}^2_{t, x}}^2\big] \big)$. 

Similarly, we obtain the bound for $\tau \sum_{n=0}^{N-1} \mathbb{E} [ \|\widehat{Y}_{1,h}(t_{n+1}) - \widehat{Y}_{1,h}(t_n)\|_{\mathbb{L}^2_x}^2 ]$ by decomposing the differences of averages for $Y_{1,h}$, applying triangle and H\"older inequalities in the same manner, summing and taking expectations to express it in terms of forward and backward time differences, and bounding by using of \eqref{today47} and \eqref{today480}.
\end{proof}

\section{Proof of Proposition~\ref{stability bound for discrete state}}\label{appendix2}
\begin{proof} 
For convenience, we denote $(X_{1,h\tau}, X_{2,h\tau})\equiv(\mathcal{X}_{1,h\tau}^0[U_{h\tau}], \mathcal{X}_{2,h\tau}^0[U_{h\tau}])$. The scheme reads: for $n = 0, \ldots, N-1$,
\begin{align}
X_{1,h\tau}(t_{n+1}) - X_{1,h\tau}(t_n) &= \frac{\tau}{2} \left( X_{2,h\tau}(t_{n+1}) + X_{2,h\tau}(t_n) \right), \label{eq:scheme1-app} \\
X_{2,h\tau}(t_{n+1}) - X_{2,h\tau}(t_n) &= \frac{\tau}{2} \left[ \Delta_h \left( X_{1,h\tau}(t_{n+1}) + X_{1,h\tau}(t_n) \right) + U_{h\tau}(t_n) \right] + \gamma X_{1,h\tau}(t_n) \Delta_{n+1} W, \label{eq:scheme2-app}
\end{align}
with $$X_{1,h\tau}(0) = X_{2,h\tau}(0) = 0.$$ Recall the Poincar\'e inequality: for $v\in\mathbb{H}_0^1$, $$\|v\|_{\mathbb{L}^2_x}^2 \leq c_P \|\nabla v\|_{\mathbb{L}^2_x}^2,$$ where $c_P > 0$ depends on the domain.

\noindent
We define $$\mathcal{Y}_n := \|\nabla X_{1,h\tau}(t_n)\|_{\mathbb{L}^2_x}^2 + \|X_{2,h\tau}(t_n)\|_{\mathbb{L}^2_x}^2.$$ To derive the energy balance, apply the identity $\big\langle a - b, a + b \big\rangle = \|a\|^2 - \|b\|^2$. By taking the gradient of \eqref{eq:scheme1-app} and the inner product with $\nabla (X_{1,h\tau}(t_{n+1}) + X_{1,h\tau}(t_n))$, we yield
\[
\|\nabla X_{1,h\tau}(t_{n+1})\|_{\mathbb{L}^2_x}^2 - \|\nabla X_{1,h\tau}(t_n)\|_{\mathbb{L}^2_x}^2 = \frac{\tau}{2} \big\langle \nabla (X_{2,h\tau}(t_{n+1}) + X_{2,h\tau}(t_n)), \nabla (X_{1,h\tau}(t_{n+1}) + X_{1,h\tau}(t_n)) \big\rangle.
\]
Taking the inner product of \eqref{eq:scheme2-app} with $X_{2,h\tau}(t_{n+1}) + X_{2,h\tau}(t_n)$ gives
\begin{align*}
&\|X_{2,h\tau}(t_{n+1})\|_{\mathbb{L}^2_x}^2 - \|X_{2,h\tau}(t_n)\|_{\mathbb{L}^2_x}^2 = \frac{\tau}{2} \big\langle \Delta_h (X_{1,h\tau}(t_{n+1}) + X_{1,h\tau}(t_n)), X_{2,h\tau}(t_{n+1}) + X_{2,h\tau}(t_n) \big\rangle \\
&\quad + \frac{\tau}{2} \big\langle U_{h\tau}(t_n), X_{2,h\tau}(t_{n+1}) + X_{2,h\tau}(t_n) \big\rangle + \gamma \big\langle X_{1,h\tau}(t_n) \Delta_{n+1} W, X_{2,h\tau}(t_{n+1}) + X_{2,h\tau}(t_n) \big\rangle.
\end{align*}
Adding these equations, the deterministic cross terms cancel because $\big\langle \Delta_h v, w \big\rangle = -\big\langle \nabla v, \nabla w \big\rangle$, leading to
\begin{align}\label{energy banla}
\mathcal{Y}_{n+1} - \mathcal{Y}_n = \frac{\tau}{2} \big\langle U_{h\tau}(t_n), X_{2,h\tau}(t_{n+1}) + X_{2,h\tau}(t_n) \big\rangle + \gamma \big\langle X_{1,h\tau}(t_n) \Delta_{n+1} W, X_{2,h\tau}(t_{n+1}) + X_{2,h\tau}(t_n) \big\rangle.
\end{align}
To expand the stochastic term, we substitute $$X_{2,h\tau}(t_{n+1}) = X_{2,h\tau}(t_n) + \frac{\tau}{2} [\Delta_h (X_{1,h\tau}(t_{n+1}) + X_{1,h\tau}(t_n)) + U_{h\tau}(t_n)] + \gamma X_{1,h\tau}(t_n) \Delta_{n+1} W$$ from \eqref{eq:scheme2-app} to yield
\[
X_{2,h\tau}(t_{n+1}) + X_{2,h\tau}(t_n) = 2 X_{2,h\tau}(t_n) + \frac{\tau}{2} \Delta_h (X_{1,h\tau}(t_{n+1}) + X_{1,h\tau}(t_n)) + \frac{\tau}{2} U_{h\tau}(t_n) + \gamma X_{1,h\tau}(t_n) \Delta_{n+1} W.
\]
The stochastic term in \eqref{energy banla} then becomes
\begin{align*}
&\gamma \big\langle X_{1,h\tau}(t_n) \Delta_{n+1} W, X_{2,h\tau}(t_{n+1}) + X_{2,h\tau}(t_n) \big\rangle = 2\gamma \big\langle X_{1,h\tau}(t_n) \Delta_{n+1} W, X_{2,h\tau}(t_n) \big\rangle \\&\qquad+ \frac{\gamma\tau}{2} \big\langle X_{1,h\tau}(t_n) \Delta_{n+1} W, \Delta_h (X_{1,h\tau}(t_{n+1}) + X_{1,h\tau}(t_n)) \big\rangle  + \frac{\gamma\tau}{2} \big\langle X_{1,h\tau}(t_n) \Delta_{n+1} W, U_{h\tau}(t_n) \big\rangle\\&\qquad + \gamma^2 \| X_{1,h\tau}(t_n) \Delta_{n+1} W \|_{\mathbb{L}^2_x}^2.
\end{align*}
Summing the energy balance \eqref{energy banla} from $n=0$ to $m-1$ (with $\mathcal{Y}_0=0$) and taking expectations gives $$\mathbb{E}[\mathcal{Y}_m] = I_1 + I_2 + I_3 + I_4 + I_5,$$ where
\begin{align*}
I_1 &= \sum_{n=0}^{m-1} \frac{\tau}{2} \mathbb{E}[\big\langle U_{h\tau}(t_n), X_{2,h\tau}(t_{n+1}) + X_{2,h\tau}(t_n) \big\rangle], \\
I_2 &= 2\gamma \sum_{n=0}^{m-1} \mathbb{E}[\big\langle X_{1,h\tau}(t_n) \Delta_{n+1} W, X_{2,h\tau}(t_n) \big\rangle] = 0 \quad (\text{since } \mathbb{E}[\Delta_{n+1} W]=0 \text{ and independence}), \\
I_3 &= \frac{\gamma\tau}{2} \sum_{n=0}^{m-1} \mathbb{E}[\big\langle X_{1,h\tau}(t_n) \Delta_{n+1} W, \Delta_h (X_{1,h\tau}(t_{n+1}) + X_{1,h\tau}(t_n)) \big\rangle], \\
I_4 &= \frac{\gamma\tau}{2} \sum_{n=0}^{m-1} \mathbb{E}[\big\langle X_{1,h\tau}(t_n) \Delta_{n+1} W, U_{h\tau}(t_n) \big\rangle], \\
I_5 &= \gamma^2 \sum_{n=0}^{m-1} \mathbb{E}[\|X_{1,h\tau}(t_n) \Delta_{n+1} W \|_{\mathbb{L}^2_x}^2] = \gamma^2 \tau \sum_{n=0}^{m-1} \mathbb{E}[\| X_{1,h\tau}(t_n) \|_{\mathbb{L}^2_x}^2] \leq c_P \gamma^2 \tau \sum_{n=0}^{m-1} \mathbb{E}[\|\nabla X_{1,h\tau}(t_n) \|_{\mathbb{L}^2_x}^2].
\end{align*}
For $I_1$, Young's inequality with $\delta>0$ gives $$I_1 \leq \frac{\tau}{2\delta} \sum_{n=0}^{m-1} \mathbb{E}[\|U_{h\tau}(t_n)\|_{\mathbb{L}^2_x}^2] + \frac{\delta \tau}{2} \sum_{n=0}^{m-1} \mathbb{E}[\|X_{2,h\tau}(t_n)\|_{\mathbb{L}^2_x}^2] + \frac{\tau \delta}{4} \mathbb{E}[\|X_{2,h\tau}(t_m)\|_{\mathbb{L}^2_x}^2].$$
For $I_3$, by using $\big\langle v, \Delta_h w \big\rangle = -\big\langle \nabla v, \nabla w \big\rangle$ and Young's inequality with $\delta>0$ gives
\[
I_3 \leq \left( \frac{\gamma^2 \tau^2}{4\delta} + \frac{\delta \tau}{2} \right) \sum_{n=0}^{m-1} \mathbb{E}[\|\nabla X_{1,h\tau}(t_n)\|_{\mathbb{L}^2_x}^2] + \frac{\tau \delta}{4} \mathbb{E}[\|\nabla X_{1,h\tau}(t_m)\|_{\mathbb{L}^2_x}^2].
\]
For $I_4$, Young's inequality with $\delta>0$ implies
\[
I_4 \leq \frac{c_P \gamma^2 \tau^2}{4} \sum_{n=0}^{m-1} \mathbb{E}[\|\nabla X_{1,h\tau}(t_n)\|_{\mathbb{L}^2_x}^2] + \frac{\tau}{4} \sum_{n=0}^{m-1} \mathbb{E}[\|U_{h\tau}(t_n)\|_{\mathbb{L}^2_x}^2].
\]
By combining all bounds, we obtain $$(1 - \frac{\tau \delta}{2}) \mathbb{E}[\mathcal{Y}_m] \leq c_1 \tau \sum_{n=0}^{m-1} \mathbb{E}[\mathcal{Y}_n] + c_2 \tau \sum_{n=0}^{m-1} \mathbb{E}[\|U_{h\tau}(t_n)\|_{\mathbb{L}^2_x}^2],$$ where $c_1 = c_P \gamma^2 + \frac{c_P \gamma^2 \tau}{4} + \frac{\gamma^2 \tau}{4\delta} + \delta$ and $c_2 = \frac{1}{4} + \frac{1}{2\delta}$. By applying the discrete Gronwall's inequality for $0 < \delta < 2/\tau$ we obtain
\[
\mathbb{E}[\|\nabla X_{1,h\tau}(t_m)\|_{\mathbb{L}^2_x}^2 + \|X_{2,h\tau}(t_m)\|_{\mathbb{L}^2_x}^2] \leq c_{11} e^{c_{21} T} \mathbb{E}[\|U_{h\tau}\|_{\mathbb{L}^2_{t,x}}^2],
\]
with $c_{11} = c_2 / (1 - \tau \delta / 2)$ and $c_{21} = c_1 / (1 - \tau \delta / 2)$.
\end{proof}
\begin{remark}[$\mathbb{L}^2$-bound with explicit constants]
	To clarify the energy estimate in the proof, we set \( \delta = 1 \) and \( \tau < 1 \), and apply the Poincaré inequality \( \|v\|_{\mathbb{L}_x^2}^2 \leq c_P \|\nabla v\|_{\mathbb{L}_x^2}^2 \), with \( c_P = \bigg(\frac{\text{diam}(D)}{\pi} \bigg)^2\); see \cite{AcostaDuran1995, PayneWeinberger1960}. This gives for any $m\in \{1,...,N\}$
	\begin{align}\label{today1414}
		\mathbb{E} \big[ \| \mathcal{X}^0_{1,h\tau}[U_{h\tau}](t_m) \|_{\mathbb{L}^2_x}^2 \big] \leq c_P c_1 e^{c_2 T} \mathbb{E} \big[ \| U_{h\tau} \|_{\mathbb{L}^2_{t, x}}^2 \big],
	\end{align}
	where \( c_1 = c_P \gamma^2 + \frac{\gamma^2 \tau}{4} (2c_P + 1) + 1 \), \( c_2 = 1 \).
	
	\noindent
	For the case \( \gamma = 0 \):
	\begin{align}\label{today1515}
		\mathbb{E} \left[ \| \mathcal{X}_{1,h\tau}^0[U_{h\tau}](t_m) \|_{\mathbb{L}^2_x}^2 \right] \leq c_P e^{T} \mathbb{E} \left[ \| U_{h\tau} \|_{\mathbb{L}^2_{t, x}}^2 \right],
	\end{align}
	since \( c_1 = 1 \), \( c_2 = 1 \).
\end{remark}
\section*{Acknowledgments}
The author would like to express sincere gratitude to Andreas Prohl for providing valuable ideas and insightful suggestions that greatly improved this work.


\begin{thebibliography}{9}

\bibitem{AcostaDuran1995}
G. Acosta and R. G. Durán, \emph{An optimal Poincaré inequality in $L^p$ for convex domains}, Indiana Univ. Math. J. \textbf{44}(2) (1995), pp.\ 621--635.

\bibitem{Archibald_Baoand_Young;}
R. Archibald, F. Bao, and J. Yong, \emph{A stochastic gradient descent approach for stochastic optimal control}, East Asian J. Appl. Math. \textbf{10}(4) (2020), pp.\ 635--658.

\bibitem{Archibald_Bao_Yong_Zhou}
R. Archibald, F. Bao, J. Yong, and T. Zhou, \emph{An efficient numerical algorithm for solving data-driven feedback control problems}, J. Sci. Comput. \textbf{85} (2020), Article 51.

\bibitem{Bakan}
H. \"Oz Bakan, \emph{An efficient algorithm for stochastic optimal control problems by means of a least-squares Monte-Carlo method}, 2022.

\bibitem{bellman1957dynamic}
R. Bellman, \emph{Dynamic Programming}, Princeton University Press, Princeton, NJ, 1957.

\bibitem{BenderDenk2007}
C. Bender and R. Denk, \emph{A forward scheme for backward SDE}, Stochastic Process. Appl. \textbf{117} (2007), pp.\ 1793--1812.

\bibitem{BenderSteiner}
C. Bender and J. Steiner, \emph{Least–Squares Monte Carlo for Backward SDE}, in \emph{Numerical Methods in Finance}, eds. R. Carmona et al., Springer, Berlin, Heidelberg, 2012, pp.\ 257--289.

\bibitem{BouchardTouzi2004}
B. Bouchard and N. Touzi, \emph{Discrete-time approximation and Monte–Carlo simulation of backward stochastic differential equations}, Stochastic Process. Appl. \textbf{111} (2004), pp.\ 175--206.

\bibitem{Brenner&Scott}
S. C. Brenner and L. R. Scott, \emph{The mathematical theory of finite element methods}, 3rd ed., Springer, Berlin, Heidelberg, New York, 2008.

\bibitem{Ciarlet}
P. G. Ciarlet, \emph{The finite element method for elliptic problems}, North-Holland, Amsterdam, 1978.

\bibitem{ChaudharyProhl}
A. Chaudhary, F. Merle, A. Prohl, and Y. Wang, \emph{An efficient discretization to simulate the solution of linear–quadratic stochastic boundary control problem}, IMA J. Numer. Anal. (2025) \textbf{00},1-55.

\bibitem{ChessariEtAl2023}
J. Chessari, R. Kawai, Y. Shinozaki, and T. Yamada, \emph{Numerical methods for backward stochastic differential equations: A survey}, Probab. Surv. \textbf{20} (2023), pp.\ 486--567.

\bibitem{Chow2014}
P.-L. Chow, \emph{Stochastic Partial Differential Equations}, CRC Press, Boca Raton, FL, 2015.

\bibitem{CohenLarssonSigg2017}
D. Cohen, S. Larsson, and D. Sigg, \emph{A fully discrete finite element approximation of the linear stochastic wave equation}, SIAM J. Numer. Anal. \textbf{55}(2) (2017), pp.\ 763--780.

\bibitem{Duns&Prohl}
T. Dunst and A. Prohl, \emph{The forward--backward stochastic heat equation: numerical analysis and simulation}, SIAM J. Sci. Comput. \textbf{38}(5) (2016), pp.\ A2725--A2755.

\bibitem{PardouxPeng1990}
E. Pardoux and S. Peng, \emph{Adapted solution of a backward stochastic differential equation}, Systems \& Control Letters \textbf{14}(1) (1990), pp.\ 55--61.


\bibitem{MaYong1999}
J. Ma and J. Yong, \emph{Forward–Backward Stochastic Differential Equations and Their Applications}, Lecture Notes in Mathematics, vol.\ 1702, Springer, 1999.

\bibitem{EngelTrautmannVexler2019}
S. Engel, P. Trautmann, and B. Vexler, \emph{Optimal finite element error estimates for an optimal control problem governed by the wave equation with controls of bounded variation}, IMA J. Numer. Anal. \textbf{41}(4) (2021), 2639–2667.


\bibitem{FengPandaProhl2024}
X. Feng, A. A. Panda, and A. Prohl, \emph{Higher order time discretization for the stochastic semilinear wave equation with multiplicative noise}, IMA J. Numer. Anal. \textbf{44}(2) (2024), pp.\ 836--885.

\bibitem{Gyorfi2002}
L. Györfi, M. Kohler, A. Krzyżak, and H. Walk, \emph{A Distribution-Free Theory of Nonparametric Regression}, Springer, New York, 2002.

\bibitem{gobet2005}
E. Gobet, J.-P. Lemor, and X. Warin, \emph{A regression-based Monte Carlo method to solve backward stochastic differential equations}, Ann. Appl. Probab. \textbf{15} (2005), pp.\ 2172--2202.

\bibitem{Hinze et al. book}
M. Hinze, R. Pinnau, M. Ulbrich, and S. Ulbrich, \emph{Optimization with PDE constraints}, Mathematical Modelling: Theory and Applications, vol.\ 23, Springer, New York, 2009.


\bibitem{KloedenPlaten2001}
P. E. Kloeden and E. Platen, \emph{Numerical Solution of Stochastic Differential Equations}, Springer, Berlin, Heidelberg, 1992 (reprinted 2001).

\bibitem{KronerKunischVexler2009}
A. Kröner, K. Kunisch, and B. Vexler, \emph{Semismooth Newton methods for optimal control of the wave equation with control constraints}, SIAM J. Control Optim. \textbf{49}(2) (2011), pp.\ 830--858.

\bibitem{LangerEtAl2024}
U. Langer, R. Löscher, O. Steinbach, and H. Yang, \emph{Robust finite element solvers for distributed hyperbolic optimal control problems}, arXiv:2404.03756 (2024).

\bibitem{Lemor2006}
J.-P. Lemor, \emph{Numerical analysis of the regression-based Monte Carlo method for BSDE}, Research Report, INRIA, 2006.

\bibitem{Li Zhou}
B. Li and Q. Zhou, \emph{Discretization of a distributed optimal control problem with a stochastic parabolic equation driven by multiplicative noise}, J. Sci. Comput. \textbf{87} (2021), Article 45.

\bibitem{Lions1971}
J.-L. Lions, \emph{Optimal Control of Systems Governed by Partial Differential Equations}, Springer, Berlin, Heidelberg, 1971.

\bibitem{Longstaff2001}
F. A. Longstaff and E. S. Schwartz, \emph{Valuing American Options by Simulation: A Simple Least--Squares Approach}, Rev. Financ. Stud. \textbf{14} (2001), pp.\ 113--147.

\bibitem{LoscherSteinbach2022}
U. Löscher and O. Steinbach, \emph{Space--Time Finite Element Methods for Distributed Optimal Control of the Wave Equation}, SIAM J. Numer. Anal. \textbf{62}(1) (2024).

\bibitem{Lu2014}
Q. Lü and X. Zhang, \emph{Optimal Feedback for Stochastic Linear Quadratic Control and Backward Stochastic Riccati Equations in Infinite Dimensions}, Mem. Amer. Math. Soc., vol.\ 294 (1467), American Mathematical Society, 2024.

\bibitem{LuZhang2021}
Q. Lü and X. Zhang, \emph{Mathematical Control Theory for Stochastic Partial Differential Equations}, Springer, Cham, 2021.


\bibitem{MenaDammStillfjord}
T. Damm, H. Mena, and T. Stillfjord, \emph{Numerical solution of the finite horizon stochastic linear quadratic control problem}, Numer. Linear Algebra Appl. \textbf{24}(4) (2017), e2090.

\bibitem{Milstein2006}
G. N. Milstein and M. V. Tretyakov, \emph{Stochastic numerics for mathematical physics}, Springer, Berlin, Heidelberg, 2004.

\bibitem{Ne}
Y. Nesterov, \emph{Introductory lectures on convex optimization}, Applied Optimization, vol.\ 87, Kluwer Academic Publishers, Boston, MA, 2004.

\bibitem{PayneWeinberger1960}
L. E. Payne and H. F. Weinberger, \emph{An optimal Poincaré inequality for convex domains}, Arch. Rational Mech. Anal. \textbf{5} (1960), pp.\ 286--292.

\bibitem{ProhlWang2021}
A. Prohl and Y. Wang, \emph{Strong error estimates for a space--time discretization of the linear--quadratic control problem with the stochastic heat equation with linear noise}, IMA J. Numer. Anal. \textbf{42}(4) (2021), pp.\ 3386--3429.

\bibitem{Prohl&Wang1}
A. Prohl and Y. Wang, \emph{Strong rates of convergence for a space--time discretization of the backward stochastic heat equation, and of a linear--quadratic control problem for the stochastic heat equation}, ESAIM Control Optim. Calc. Var. \textbf{27} (2021), Article 20.

\bibitem{ProhlWangbook}
A. Prohl and Y. Wang, \emph{Numerical methods for optimal control problems with SPDEs}, Springer Nature, (2025).

\bibitem{B.Li}
 B. Li, Q. Zhou, \emph{Discretization of a distributed optimal control problem with a stochastic
parabolic equation driven by multiplicative noise}. J. Sci. Comp. 87, 1-37 (2021).

\bibitem{SteinbachZank2022}
O. Steinbach and A. Zank, \emph{An inf--sup stable variational formulation for PDE-constrained LQ optimal control}, Numer. Math. \textbf{150} (2022), pp.\ 75--103.

\bibitem{Tessitore}
G. Tessitore, \emph{Existence, uniqueness and space regularity of the adapted solutions of a backward SPDE}, Stoch. Anal. Appl. \textbf{14} (1996), pp.\ 461--486.

\bibitem{TrautmannVexlerZlotnik2018a}
P. Trautmann, B. Vexler, and A. Zlotnik, \emph{On a finite element method for measure-valued optimal control problems governed by the 1D generalized wave equation}, C. R. Math. \textbf{356}(5) (2018), pp.\ 523--531.

\bibitem{TrautmannVexlerZlotnik2018b}
P. Trautmann, B. Vexler, and A. Zlotnik, \emph{Finite element error analysis for measure-valued optimal control problems governed by a 1D wave equation with variable coefficients}, Math. Control Relat. Fields \textbf{8}(2) (2018), pp.\ 411--449.

\bibitem{Troeltzsch2010}
F. Tröltzsch, \emph{Optimal Control of Partial Differential Equations: Theory, Methods and Applications}, American Mathematical Society, Providence, RI, 2010.

\bibitem{wang2020}
P. Wang, Y. Wang, Q. Lü, and X. Zhang, \emph{Numerics for stochastic distributed parameter control systems: a finite transposition method}, in \emph{Handbook of Numerical Analysis}, vol.\ 23 (Numerical Control: Part A), Elsevier, Amsterdam, 2022, pp.\ 201--232.

\bibitem{XiaEtAl2019}
J. Xia, S. Yuan, and P. Zhang, \emph{Energy harvesting from stochastic water waves: modeling and control}, Ocean Eng. \textbf{180} (2019), pp.\ 134--145.

\bibitem{Yanqing 2021}
Y. Wang, \emph{Error analysis of a discretization for stochastic linear quadratic control problems governed by SDE}, IMA J. Math. Control Inform. \textbf{38} (2021), pp.\ 1148--1173.

\bibitem{Yanqing 2023}
Y. Wang, \emph{Error analysis of the feedback controls arising in the stochastic linear quadratic control problems}, J. Syst. Sci. Complex. \textbf{36} (2023), pp.\ 1540--1559. DOI:10.1007/s11424-023-1102-7.

\bibitem{Zuazua2005}
E. Zuazua, \emph{Propagation, observation, and control of waves approximated by finite difference methods}, SIAM Rev. \textbf{47}(2) (2005), pp.\ 197--243.


\end{thebibliography}
\end{document}